\documentclass{amsart}
\usepackage{graphicx}
\usepackage{indentfirst}
\usepackage{hyperref}

\usepackage{latexsym}
\usepackage{amsmath}
\usepackage{amsthm}
\usepackage{amssymb}
\usepackage{graphicx}
\usepackage{pdfpages}
\usepackage[labelformat=empty]{caption}

\newtheorem{theorem}{Theorem}[section]
\newtheorem{lemma}[theorem]{Lemma}
\newtheorem{cor}[theorem]{Corollary}

\theoremstyle{definition}

\newtheorem{prop}[theorem]{Proposition}
\theoremstyle{remark}
\newtheorem{remark}[theorem]{Remark}

\newcommand{\vip}{\vskip.2cm}
\newcommand{\R}{\mathbb{R}}
\newcommand{\cX}{\boldsymbol{X}}
\newcommand{\E}{\mathbb{E}}
\newcommand{\Et}{\E_\theta}
\newcommand{\Var}{\mathbb{V}\mathrm{ar}}
\newcommand{\cA}{\mathcal{A}}
\newcommand{\ct}{\boldsymbol{t}}
\newcommand{\cF}{\boldsymbol{F}}
\newcommand{\cx}{\boldsymbol{x}}
\newcommand{\cc}{\boldsymbol{c}}
\newcommand{\cf}{\boldsymbol{f}}
\newcommand{\cV}{\mathbf{V}}
\newcommand{\ce}{\boldsymbol{e}}
\newcommand{\cL}{\boldsymbol{L}}
\newcommand{\cJ}{\boldsymbol{J}}
\newcommand{\cM}{\boldsymbol{M}}
\newcommand{\cD}{\mathcal{D}}
\newcommand{\mL}{\mathcal{L}}
\newcommand{\bL}{\boldsymbol{\mathcal{L}}}
\newcommand{\indiq}{{\bf 1}}
\newcommand{\bun}{{\mathbf 1}_N}

\hypersetup{pdfmenubar=true}

\begin{document}
\title[Statistical inference for Hawkes process]{Statistical inference for a partially observed interacting system of Hawkes processes}

\author{Chenguang LIU}

\address{Sorbonne Universit\'e- LPSM, Campus Pierre et Marie Curie, 4 place Jussieu, 75252 ´
PARIS CEDEX 05, LPSM, F-75005 Paris, France.}

\email{LIUCG92@gmail.com}

\begin{abstract}
We observe the actions of a $K$ sub-sample of $N$ individuals up to time $t$ for some large $K\le N$.
We model the relationships of individuals by i.i.d. Bernoulli($p$)-random variables, where $p\in (0,1]$ is an unknown parameter. The rate of action of each individual depends on some unknown parameter $\mu> 0$ and
  on the sum of some function $\phi$ of the ages of the actions of the individuals which influence him. The function $\phi$ is unknown but we assume it rapidly decays.  The aim of this paper is to estimate the parameter $p$ asymptotically as $N\to \infty$, $K\to \infty$, and $t\to \infty$. Let $m_t$ be the average number of actions per individual up to time $t$.
In the subcritical case, where $m_t$ increases linearly, we build an estimator of $p$ with the rate
of convergence $\frac{1}{\sqrt{K}}+\frac{N}{m_t\sqrt{K}}+\frac{N}{K\sqrt{m_t}}$.
In the supercritical case, where  $m_{t}$ increases exponentially fast, we build an estimator of $p$
with the rate of convergence $\frac{1}{\sqrt{K}}+\frac{N}{m_{t}\sqrt{K}}$.
\end{abstract}



\maketitle

\tableofcontents

\label{key}\section{Introduction} 
\subsection{Motivation}
 Hawkes processes have been used to model interactions between multiple entities evolving through time.   For  example in  neuroscience,   Reynaud-Bouret et al. \cite{PRBB} use  multivariate Hawkes processes to model the spikes of  different neurons.  In finance,  Bauwens and Hautsch  in \cite{Q}  give an order book model.  Social networks interactions are considered in Blundell et al. \cite{F},
Simma-Jordan \cite{ASMI}, Zhou et al. \cite{zhzh}. There are even some application in criminology, see e.g. Mohler,  Short, Brantingham,  Schoenberg and Tita in \cite{S}.

Concerning the statistical inference for Hawkes processes, mainly the case of fixed finite dimension
$N$ has been studied  in the asymptotic $t\to \infty.$  In parametric models, Ogata has studied the maximum likelihood estimator  in \cite{o1}.  Non-parametric models  were considered by Bacry and Muzzy \cite{bmu2}, Hansen et al. \cite{hrr}, Reynaud-Bouret et al. \cite{rs,rrgt,PRBB} and  Rasmussen \cite{r} with a  Bayesian approach,
   see remarks \ref{11} and \ref{12} for details.
 \vip
 
 However, in the real world, we often need to consider the case when the number of individuals is large. For example, in the neuroscience, the number of the neurons are usually enormously large.  So it is natural  to consider the double asymptotic  $t\to \infty\ and\ N\to \infty$.   The studies about this case are rare.  As far as we are aware,  the only paper which consider this case is \cite{A}.

\subsection{Model}
We have $N$ individuals. Each individual $j \in \{1,\dots,N\}$ is connected to the set of individuals
$S_j=\{i \in \{1,\dots,N\} : \theta_{ij}=1\}$.
The only possible action of the individual $i$ is to send a message to all the individuals of $S_i$.
Here $Z^{i,N}_t$ stands for the number of messages sent by $i$ during $[0,t]$. The  counting process $(Z_{s}^{i,N})_{i=1...N,0\le s\le t}$ is determined by its intensity process $(\lambda_{s}^{i,N})_{i=1...N,0\le s\le t}.$ It is informally defined by 
\begin{align*}
    P\Big(Z_{t}^{i,N} \textit{has a jump in } [t, t + dt]\Big|\mathcal{F}_t\Big)= \lambda_{t}^{i,N}dt,\  i = 1,..., N
\end{align*}

where   $\mathcal{F}_{t}$ denotes the sigma-field generated by $(Z_{s}^{i,N})_{i=1...N,0\le s\le t}$ and $(\theta_{ij})_{i,j=1,...,N}$. 

The rate $\lambda^{i,N}_t$ at which $i$ sends messages can be decomposed as the sum of two effects:

$\bullet$ he sends {\it new} messages at rate $\mu$;

$\bullet$ he {\it forwards} the messages he received, after some delay (possibly infinite) depending on the 
age of the message, which induces a sending rate of the form 
$\frac{1}{N}\sum_{j=1}^{N}\theta_{ij}\int_{0}^{t-}\phi(t-s)dZ_{s}^{j,N}$.

\smallskip

If for example $\phi=\boldsymbol{1}_{[0,K]}$, then  $N^{-1} \sum_{j=1}^N \theta_{ij}\int_0^{t-}  \phi(t-s)dZ_{s}^{j,N}$ is
precisely the number of messages that the $i$-th individual received between time $t-K$ and time $t$,
divided by $N$.
\begin{remark}\label{11}
In Bacry and Muzzy  \cite{bmu2},  Hansen et al \cite{hrr},  Reynaud-Bouret et al \cite{rs,rrgt,PRBB} consider the non-parametric estimation of the following system: for fixed $N\ge 1,$ and $i,j=1,...,N$, the  counting process $(Z_{s}^{i,N})_{i=1...N,0\le s\le t}$ is determined by its intensity process $(\lambda_{s}^{i,N})_{i=1...N,0\le s\le t},$ which defined by
\begin{align}\label{BMuy}
\lambda_{t}^{i,N}:=\mu_i+\sum_{j=1}^{N}\int_{0}^{t-}\phi_{ij}(t-s)dZ_{s}^{j,N}
\end{align}
for $\mu_i>0$ and $\phi_{ij}:[0,\infty)\to [0,\infty)$ is measurable, locally integrable. They provided estimators of the $\mu_i$ and the functions $\phi_{ij}.$
For fixed $N$, our model can be seen as a special case of (\ref{BMuy}) for $\mu_i=\mu$ and $\phi_{ij}(s)=\frac{\theta_{ij}}{N}\phi(s)$. 
\end{remark} 
 
 \begin{remark}\label{12}
In \cite{r}, Rasmussen consider the Bayesian inference of the following one dimensional system: the  counting process $(Z_{s})_{0\le s\le t}$ is determined by its intensity process $(\lambda_{s})_{0\le s\le t},$ which defined by
\begin{align}\label{rms}
\lambda_{t}:=\mu_t+\int_{0}^{t-}\phi(t-s)dZ_{s}
\end{align}
which the rate term $\mu_t$ depends on time $t.$
\end{remark}

\subsection{Main Goals}
We assume that $(\theta_{ij})_{i,j=1,...,N}$ is a family  of i.i.d. Bernoulli($p$) random variables, where $p$ is an unknown parameter. The parameter $\mu$ and the function $\phi$ are also unknown. In  \cite{A}, Delattre and Fournier  consider the case when one observes the  whole sample 
$(Z_{s}^{i,N})_{i=1...N,0\le s\le t}$ and they propose an estimator of $p$.

However, in the real world, it is  often impossible to observe the whole population. Our goal in the present 
paper is to consider the case where one observes only a subsample of indivudals. In other words, we want to build an estimator of $p$ 
when observing $(Z_{s}^{i,N})_{\{i=1,...,K,\ 0\le s\le t\}}$ with $1\ll K \le N$ and with $t$ large. 
\begin{remark}
Since the family of  $(Z^{i,N})_{\{i=1,...,N\}}$ is exchangeable,  considering that the observation  is given by  the first $K$ processes is not restriction.
\end{remark}

\smallskip

Let $\Lambda=\int_{0}^{\infty}\phi(t)dt\in (0,\infty].$ In \cite{A}, we  see that growth of $Z_{t}^{1,N}$  depends 
on the value of $\Lambda p$. When $\Lambda p<1$ (subcritical case),  $Z_{t}^{1,N}$ increases (in average) 
linearly with time, while when $\Lambda p>1$ (supercritical case), it increases exponentially. 
Thus  the estimation procedure will be different in the two cases.
We will not consider the critical case when $\Lambda p=1$.
\begin{remark}
We can find   simulations for $K<N$ already in \cite{A}. The simulations are about the case $K=\frac{N}{4}.$
\end{remark}

 

\section{Main results}

\subsection{Setting }
We consider some unknown parameters $p$ $\in(0,1],\  \mu >0$ and $\phi:[0,\infty)\to [0,\infty)$. We always assume that the function $\phi$ is measurable and locally integrable. For  $N\ge 1$, we consider an i.i.d. family $(\Pi^{i}(dt,dz))_{i=1,...,N}$ of Poisson measures on $[0,\infty)\times [0,\infty)$ with intensity $dtdz$, together with $(\theta_{ij})_{i,j=1,...,N}$ is a family  of i.i.d. Bernoulli($p$) random variables which is independent of the family $(\Pi^{i}(dt,dz))_{i=1,...,N}$. We consider the following system: for all $i\in \{1,...,N\},$ all $t\geq 0,$
\begin{align}\label{sssy}
Z_{t}^{i,N}:=\int^{t}_{0}\int^{\infty}_{0}\boldsymbol{1}_{\{z\le\lambda_{s}^{i,N}\}}\Pi^{i}(ds,dz), \hbox{ where }
\lambda_{t}^{i,N}:=\mu+\frac{1}{N}\sum_{j=1}^{N}\theta_{ij}\int_{0}^{t-}\phi(t-s)dZ_{s}^{j,N}.
\end{align}
In this paper, $\int_{0}^{t}$ means $\int_{[0,t]}$, and $\int_{0}^{t-}$ means $\int_{[0,t)}$. The solution $((Z_{t}^{i,N})_{t\ge 0})_{i=1,...,N}$ is a family of counting processes. 
By \cite[Proposition 1]{A}, the system $(1)$ has a unique $(\mathcal{F}_{t})_{t\ge 0}$-measurable c\`adl\`ag 
solution, where   
$$\mathcal{F}_{t}=\sigma(\Pi^{i}(A):A\in\mathcal{B}([0,t]\times [0,\infty)),i=1,...,N)\vee 
\sigma(\theta_{ij},i,j=1,...,N),$$ as soon as $\phi$ is locally integrable.

\subsection{Assumptions}
Recall that $\Lambda=\int_{0}^{\infty}\phi(t)dt\in (0,\infty].$
We will work under one of the two following conditions:
either for some $q\geq 1$,
\renewcommand\theequation{{$H(q)$}}
\begin{equation}
\mu \in (0,\infty), \quad \Lambda p \in (0,1) \quad \hbox{and} 
\quad \int_0^\infty s^q\phi(s)ds <\infty
\end{equation}
or
\renewcommand\theequation{{$A$}}
\begin{equation}
\mu \in (0,\infty), \quad \Lambda p \in (1,\infty] \quad \hbox{and} \quad \int_0^t |d\phi(s)|
\hbox{ increases at most polynomially.}
\end{equation}
\renewcommand\theequation{\arabic{equation}}\addtocounter{equation}{-2}

\begin{remark}
In many applications, $\phi$ is smooth and  decays fast.
Hence what we have in mind is that in the subcritical case,  $(H(q))$ is satisfied for all $q\geq 1$.
In the supercritical case, $(A)$ seems very reasonable: all the nonegative polynomial functions satifies the condition.
\end{remark}
\subsection{The result in the subcritical case}
For $N\geq 1$ and for $((Z^{i,N}_t)_{t\geq 0})_{i=1,\dots,N}$ the solution of (\ref{sssy}), we set
$\bar{Z}^{N}_{t}:=N^{-1}\sum_{i=1}^N Z^{i,N}_t$ and $\bar{Z}^{N,K}_{t}:=K^{-1}\sum_{i=1}^K Z^{i,N}_t.$ Next, we introduce
$$\varepsilon _{t}^{N,K}:=t^{-1}(\bar{Z}_{2t}^{N,K}-\bar{Z}_{t}^{N,K}),\qquad \mathcal{V}_{t}^{N,K}
:=\frac{N}{K}\sum_{i=1}^{K}\Big[\frac{Z_{2t}^{i,N}-Z_{t}^{i,N}}t-\varepsilon_{t}^{N,K}\Big]^{2}-\frac{N}{t}\varepsilon_{t}^{N,K}.$$
For $\Delta>0$ such that $t/(2\Delta)\in \mathbb{N}^{*},$ we set
\begin{align}
\mathcal{W}_{\Delta,t}^{N,K}:=2\mathcal{Z}_{2\Delta,t}^{N,K}-\mathcal{Z}_{\Delta,t}^{N,K},\quad
\label{Xtnk}\mathcal{X}_{\Delta,t}^{N,K}:=\mathcal{W}_{\Delta,t}^{N,K}-\frac{N-K}{K}\varepsilon_t^{N,K}
\\
\hbox{where} \quad
\mathcal{Z}^{N,K}_{\Delta,t}:=\frac{N}{t}\sum_{a=t/\Delta}^{2t/\Delta}(\bar{Z}_{a\Delta}^{N,K}-\bar{Z}_{(a-1)\Delta}^{N,K}-\Delta\varepsilon_{t}^{N,K})^{2}.
\end{align}
\begin{remark}
The estimators we defined above were already appearing in \cite{A}. The aim of this paper is to prove the convergence of these estimators.
\end{remark}
\begin{theorem}\label{abcd}
We assume $(H(q))$ for some $q>3$. There exists constants $C<\infty$, $C'>0$ depending only on $q$, $p$, $\mu$, $\phi$
such that for all $\varepsilon\in (0,1)$, all $1\le K\le N$, all $t\geq 1$,  setting 
$\Delta_t= t/(2 \lfloor t^{1-4/(q+1)}\rfloor)$,
$$P\Big(\Big|\Psi\Big(\varepsilon_{t}^{N,K},\mathcal{V}_{t}^{N,K},\mathcal{X}_{\Delta_{t},t}^{N,K}\Big)-
(\mu,\Lambda,p)\Big|\ge \varepsilon\Big) \le \frac{C}{\varepsilon}\Big(\frac{1}{\sqrt{K}}+\frac{N}{K\sqrt{t^{1-\frac{4}{1+q}}}}+\frac{N}{t\sqrt{K}}\Big)+CNe^{-C'K}$$
with $\Psi:=\boldsymbol{1}_{D}\Phi: \R^3\mapsto\R^3$, the function 
$\Phi:=(\Phi^{(1)},\Phi^{(2)},\Phi^{(3)})$ being defined on 
$D:= \{(u, v, w)\in \mathbb{R}^{3}: w > 0 \ and\ u,\  v\ge 0\}$ by
\begin{gather*}
\Phi^{(1)}(u,v,w):=u\sqrt{\frac{u}{w}},\quad  \Phi^{(2)}(u,v,w):=\frac{v+[u-\Phi^{(1)}(u,v,w)]^{2}}{u[u-\Phi^{(1)}(u,v,w)]}, \\
\Phi^{(3)}(u,v,w):=\frac{1-u^{-1}\Phi^{(1)}(u,v,w)}{\Phi^{(2)}(u,v,w)}.
\end{gather*}
\end{theorem}
 
 We quote \cite[Remark 2]{A}, which says that the mean number of actions per individual per unit 
of time increases linearly. 

\begin{remark}\label{re12}
Assume $H(1)$. Then for all $\varepsilon>0$,
$$
\lim_{(N,t)\to (\infty,\infty)} P \Big(\Big| \frac{\bar Z^{N,K}_t}t- \frac \mu {1-\Lambda p}\Big|
\geq \varepsilon \Big)=0.
$$
So roughly, if observing $((Z_{s}^{i,N})_{s \in [0,t]})_{i=1,...,K}$, we observe approximately $Kt$ actions.
\end{remark}

\begin{remark}
If the function $\phi$   decays fast, for example $\phi(s)= a e^{-bs}$ or $c\indiq_{D}$ where $D$ is some compact set. In these situations, the function $\phi$ can satisfy the assumptions for arbitrary $q>0$. Hence, we can almost replace  $\frac{N}{K\sqrt{t}}$ by $\frac{N}{K\sqrt{t^{1-\frac{4}{1+q}}}}.$  
\end{remark}

\begin{remark}
We are going to consider two special cases:

$\bullet$ When $K\sim N,$ we have 
$$
(\frac{1}{\sqrt{K}}+\frac{N}{K\sqrt{t^{1-\frac{4}{1+q}}}}+\frac{N}{t\sqrt{K}})+CNe^{-C'K}\sim (\frac{1}{\sqrt{N}}+\frac{1}{\sqrt{t^{1-\frac{4}{1+q}}}}+\frac{\sqrt{N}}{t})+CNe^{-C'N}. 
$$
Hence, in order to ensure the convergence, we just need $\frac{\sqrt{N}}{t}\to 0.$

$\bullet$ Assume $K\sim \gamma \log N$ and $\gamma C'>1,$ where $C'$ is as in theorem \ref{abcd},  we have 
$$
(\frac{1}{\sqrt{K}}+\frac{N}{K\sqrt{t^{1-\frac{4}{1+q}}}}+\frac{N}{t\sqrt{K}})+CNe^{-C'K}\sim (\frac{1}{\sqrt{\log N}}+\frac{N}{\log N\sqrt{t^{1-\frac{4}{1+q}}}}+\frac{N}{t\sqrt{\log N}})+CN^{1-\gamma C'}. 
$$

Hence, in order to ensure the convergence, we just need $\frac{N}{\log N\sqrt{t^{1-\frac{4}{1+q}}}}+\frac{N}{t\sqrt{\log N}}\to 0,$ which equivalent to 
$\frac{N}{\log N\sqrt{t^{1-\frac{4}{1+q}}}}\to 0.$
\end{remark}

\subsection{The result in the supercritical case}

Here we define $\bar{Z}_{t}^{N,K}$ as previously and we set
\begin{align}
\label{UP}\mathcal{U}_{t}^{N,K}:=\Big[\frac{N}{K}\sum_{i=1}^{K}\Big(\frac{Z_{t}^{i,N}-\bar{Z}_{t}^{N,K}}{\bar{Z}_{t}^{N,K}}\Big)^{2}-
\frac{N}{\bar{Z}_{t}^{N,K}}\Big]\boldsymbol{1}_{\{\bar{Z}_{t}^{N,K}>0\}}\\
 \label{PU}\hbox{and} \quad \mathcal{P}_{t}^{N,K}:=\frac{1}{\mathcal{U}_{t}^{N,K}+1}\boldsymbol{1}_{\{\mathcal{U}_{t}^{N,K}\ge 0\}}.
\end{align}

\begin{theorem} \label{supercrit}
We assume $(A)$ and define $\alpha_{0}$ by $p\int_{0}^{\infty}e^{-\alpha_{0}t}\phi(t)dt=1$ (recall that by $(A)$,
$\Lambda p = p \int_{0}^{\infty}\phi(t)dt>1$).
 For all $\eta>0$, there is a constant $C_{\eta}>0$ (depending on $p,\mu,\phi,\eta$), such that for all $N\ge K\ge 1$, all $\varepsilon\in(0,1)$,
$$P(|\mathcal{P}_{t}^{N,K}-p|\ge \varepsilon)\le \frac{C_{\eta}e^{4\eta t}}{\varepsilon}\Big(\frac{N}{\sqrt{K}e^{\alpha_{0}t}}+\frac{1}{\sqrt{K}}\Big).$$
\end{theorem}
Next, we quote \cite[Remark 5]{A}.

\begin{remark} \label{remsup}
Assume $(A)$ and consider $\alpha_0>0$ such that $p \int_0^\infty e^{-\alpha_0 t}\phi(t)dt =1$.
Then for all $\eta>0$,
$$\lim_{t\to \infty}\lim_{(N,K)\to (\infty,\infty)} P(\bar{Z}_{t}^{N,K}\in[e^{(\alpha_{0}-\eta)t},e^{(\alpha_{0}+\eta)t}])=1.$$
So roughly, if observing $((Z_{s}^{i,N})_{s \in [0,t]})_{i=1,...,K}$, we observe around $Ke^{\alpha_{0}t}$ actions.
\end{remark}

\section{On the choice of the estimators}

In the whole paper, we denote by $\E_\theta$ the conditional expectation knowing $(\theta_{ij})_{i,j=1,\dots,N}$.
Here we explain informally why the estimators should converge.

\subsection{The subcritical  case}
We define $A_{N}(i,j):=N^{-1}\theta_{ij}$ and the matrix $(A_{N}(i,j))_{i,j\in\{1,...,N\}}$, as well as 
$Q_{N}:=(I-\Lambda A_{N})^{-1}$ on the event on which $I-\Lambda A_N$ is invertible.

\smallskip

Define $\widetilde{\varepsilon} _{t}^{N,K}:=t^{-1}\bar{Z}_{t}^{N,K},\ K\le N$. We expect that,
for $t$ large enough, $Z_{t}^{i,N} \simeq \mathbb{E}_{\theta}[Z_{t}^{i,N}]$. 
And, by definition of $Z_{t}^{i,N}$, see \eqref{sssy}, it is not hard to get  
$$\mathbb{E}_{\theta}[Z_{t}^{i,N}]=\mu t+N^{-1}\sum_{j=1}^{N}\theta_{ij}\int_{0}^{t}\phi(t-s)
\mathbb{E}_{\theta}[Z_{s}^{j,N}]ds.$$
Hence, assuming that $\gamma_{N}(i)=\lim_{t\to \infty}t^{-1}\mathbb{E}_{\theta}[Z_{t}^{i,N}]$  exists for each $i=1,...,N$
and observing that $\int_{0}^{t}\phi(t-s)sds\simeq \Lambda t$, we find that the vector 
$\boldsymbol{\gamma}_{N}=(\gamma_N(i))_{i=1,\dots,N}$ should satisfy
$\boldsymbol{\gamma}_{N}=\mu\boldsymbol{1}_{N}+\Lambda A_{N}\boldsymbol{\gamma}_{N} $,
where $\boldsymbol{1}_{N}$ is the vector defined by $\boldsymbol{1}_{N}(i)=1$ for all $i=1,\dots,N$.
Thus we  deduce that $\boldsymbol{\gamma}_{N}=\mu(I-\Lambda A_{N})^{-1}\boldsymbol{1}_{N}=\mu\boldsymbol{\ell}_{N}$, where we have set
$$
\boldsymbol{\ell}_N:=Q_N\boldsymbol{1}_N,\ \ell_{N}(i):=\sum_{j=1}^{N}Q_{N}(i,j),\ \bar{\ell}_N:=\frac{1}{N}\sum_{i=1}^N\ell_N(i),\ \bar{\ell}^K_N:=\frac{1}{K}\sum_{i=1}^K\ell_N(i)
$$
So we  expect that  $Z_{t}^{i,N}\simeq \mathbb{E}_{\theta}[Z_{t}^{i.N}]\simeq \mu \ell_{N}(i)t$,
whence $\widetilde{\varepsilon}_{t}^{N,K}=t^{-1}\bar{Z}_{t}^{N,K}\simeq \mu\bar{\ell}_{N}^{K}$.

\smallskip

We informally show that
$\ell_{N}(i)\simeq 1+\Lambda(1-\Lambda p)^{-1}L_{N}(i)$, where $L_{N}(i):=\sum_{j=1}^{N}A_{N}(i,j)$:
when $N$ is large, $\sum_{j=1}^N A_{N}^{2}(i,j)= N^{-2}\sum_{j=1}^N\sum_{k=1}^N \theta_{ik}\theta_{kj}
\simeq p  N^{-1}\sum_{k=1}^N \theta_{ik} = p L_N(i)$. And one gets convinced similarly that
for any $n\in \mathbb{N}_*$, roughly, $\sum_{j=1}^{N}A_{N}^{n}(i,j)\simeq p^{n-1}L_{N}(i)$. 
So
$$\ell_{N}(i)=\sum_{n\ge 0}\Lambda^{n}\sum_{j=1}^{N}A_{N}^{n}(i,j)\simeq 1+\sum_{n\ge 1}\Lambda^{n}p^{n-1}L_{N}(i)=1+\frac{\Lambda}{1-\Lambda p}L_{N}(i).$$ 
But $(NL_{N}(i))_{i=1,...,N}$ are i.i.d. Bernoulli$(N,p)$ random variables, so that
$\bar{\ell}^{K}_{N}\simeq 1 + \Lambda p (1-\Lambda p)^{-1}=(1-\Lambda p)^{-1}$.
Finally, we have explained why $\widetilde{\varepsilon} _{t}^{N,K}$ should resemble $\mu (1-\Lambda p)^{-1}$.

\smallskip

Knowing $(\theta_{ij})_{i,j=1..N}$, the process $Z_{t}^{1,N}$  resembles a Poisson process,
so that $\Var_\theta (Z_{t}^{1,N})\simeq \Et[Z_{t}^{1,N}]$, whence
$$
\Var(Z_{t}^{1,N})=\Var(\mathbb{E}_{\theta}[Z_{t}^{1,N}])+\mathbb{E}[\Var_{\theta}(Z_{t}^{1,N})]\simeq 
\Var(\mathbb{E}_{\theta}[Z_{t}^{1,N}])+\mathbb{E}[Z_{t}^{1,N}].
$$
Writing an empirical version of this equality, we find
$$
\frac 1K\sum_{i=1}^{K}(Z_{t}^{i,N}-\bar{Z}_{t}^{N,K})^{2}\simeq \frac 1K\sum_{i=1}^{K}\Big(\mathbb{E}_{\theta}[Z_{t}^{i,N}]-
\mathbb{E}_{\theta}[\bar{Z}_{t}^{N,K}]\Big)^{2}+\bar{Z}_{t}^{N,K}.
$$
And since $Z_{t}^{i,N}\simeq \mu \ell_{N}(i)t\simeq 
\mu [1+(1-\Lambda p)^{-1}\Lambda  L_N(i)]t$
as already seen a few lines above,
we find
$$
\frac 1K\sum_{i=1}^{K}(Z_{t}^{i,N}-\bar{Z}_{t}^{N,K})^{2}
\simeq \frac {\mu^{2}t^{2}\Lambda^2}{K(1-\Lambda p)^2}\sum_{i=1}^{K}(L_{N}(i)-
\bar{L}_{N}^{K})^{2}+\bar{Z}_{t}^{N,K}.
$$
But $(NL_{N}(i))_{i=1,...,N}$ are i.i.d. Bernoulli$(N,p)$ random variables, so that
\begin{align*}
\widetilde{\mathcal{V}}_{t}^{N,K}:=&
\frac N K \sum_{i=1}^K \Big[ \frac {Z^{i,N}_t}t -  \widetilde{\varepsilon}^{N,K}_t\Big]^2 
- \frac N t \widetilde{\varepsilon}^{N,K}_t\\
=&\frac N{K t^{2}}\Big[\sum_{i=1}^{K}(Z_{t}^{i,N}-\bar{Z}_{t}^{N,K})^{2}-K\bar{Z}_{t}^{N,K}
\Big] \\
\simeq &\frac{N\mu^{2}\Lambda^2}{K(1-\Lambda p)^2}\sum_{i=1}^{K}(L_{N}(i)-\bar{L}_{N}^{K} )^{2}
\simeq \frac{\mu^{2}\Lambda^{2}p(1-p)}{(1-\Lambda p)^{2}}.
\end{align*}

We finally build a third estimator.
The temporal empirical variance 
$$
\frac \Delta t\sum_{k=1}^{t/\Delta}\Big[\bar{Z}_{k\Delta}^{N,K}-\bar{Z}_{(k-1)\Delta}^{N,K}-\frac \Delta t\bar{Z}_{t}^{N,K}
\Big]^{2}
$$  
should resemble $\Var_{\theta}[\bar{Z}_{\Delta}^{N,K}] $ if $1\ll\Delta \ll t$. 
So we expect that:
$$\widetilde{\mathcal{W}}_{\Delta,t}^{N,K}:= 
\frac Nt\sum_{k=1}^{t/\Delta}\Big[\bar{Z}_{k\Delta}^{N,K}-\bar{Z}_{(k-1)\Delta}^{N,K}-\Delta t^{-1}\bar{Z}_{t}^{N,K}\Big]^{2}
\simeq \frac N \Delta \Var_{\theta}[\bar{Z}_{\Delta}^{N,K}].
$$
To understand what $\Var_{\theta}[\bar{Z}_{\Delta}^{N,K}]$ looks like, we 
introduce the centered process $U_{t}^{i,N}:=Z_{t}^{i,N}-\mathbb{E}_{\theta}[Z_{t}^{i,N}]$ and the 
martingale $M_{t}^{i,N}:=Z_{t}^{i,N}-C_{t}^{i,N}$ where $C^{i,N}$ is the compensator of $Z^{i,N}$.  
An easy computation, see \cite[Lemma 11]{A}, shows that, denoting by
$\boldsymbol{U}^{N}_{t}$ and $\boldsymbol{M}^{N}_{t}$ the vectors $(U^{i,N}_t)_{i=1,\dots,N}$ and
$(M^{i,N}_t)_{i=1,\dots,N}$,
$$
\boldsymbol{U}^{N}_{t}=
\boldsymbol{M}_{t}^{N}+A_{N}\int_{0}^{t}\phi(t-s)\boldsymbol{U}_{s}^{N}ds.
$$ 
So for large times, we  conclude that
$\boldsymbol{U}_{t}^{N}\simeq \boldsymbol{M}_{t}^{N}+\Lambda A_{N}\boldsymbol{U}_{t}^{N}$, whence
finally $\boldsymbol{U}_{t}^{N}\simeq Q\boldsymbol{M}_{t}^{N}$ and thus
$$
\frac{1}{K}\sum_{i=1}^{K}U^{i,N}_{t}\simeq \frac{1}{K} \sum_{i=1}^{K}\sum_{j=1}^{N}Q(i,j)M_{t}^{j,N}
=  \frac{1}{K} \sum_{j=1}^N c_{N}^{K}(j) M_{t}^{j,N},
$$
where we have set $c_{N}^{K}(j)=\sum_{i=1}^{K}Q_{N}(i,j)$.
But we obviously have $[M^{j,N},M^{i,N}]_{t}=\boldsymbol{1}_{\{i=j\}}Z_{t}^{j,N}$ (see \cite[Remark 10]{A}), so that
$$
\Var_{\theta}[\bar{Z}_{t}^{N,K}]=\Var_{\theta}[\bar{U}_{t}^{N,K}] \simeq 
\frac{1}{K^{2}}\sum_{j=1}^{N}(c_{N}^{K}(j))^{2} Z_{t}^{j,N}.
$$ 
Recalling that $Z_{t}^{j,N} \simeq \mu \ell_N(j) t$, we conclude that
$\Var_{\theta}[\bar{Z}_{t}^{N,K}]\simeq K^{-2}\mu t \sum_{j=1}^{N}\Big(c_{N}^{K}(j)\Big)^{2} \ell_N(j)$, whence
$$
\widetilde{\mathcal{W}}_{\Delta,t}^{N,K}\simeq \frac N\Delta\Var_{\theta}[\bar{Z}_{\Delta}^{N,K}]\simeq 
\mu \frac{N}{K^{2}}\sum_{j=1}^{N}\Big(c_{N}^{K}(j)\Big)^{2}\ell_{N}(j).
$$ 

To compute this last quantity, we start from $c_{N}^{K}(j)=\sum_{n\ge 0}\sum_{i=1}^{K}\Lambda^{n}A_{N}^{n}(i,j)$.
But we have $\sum_{i=1}^K A_{N}^{2}(i,j)= N^{-2}\sum_{i=1}^K\sum_{k=1}^N \theta_{ik}\theta_{kj}
\simeq p  KN^{-2}\sum_{k=1}^N \theta_{kj} = p KN^{-1} C_N(j)$. And one gets convinced similarly that
for any $n\in \mathbb{N}_*$, roughly, $\sum_{i=1}^{K}A_{N}^{n}(i,j)\simeq KN^{-1} p^{n-1}C_{N}(j)$. 
So we conclude that
$c_{N}^{K}(j)\simeq A_{N}^{0}(i,j)+\frac{K\Lambda}{N(1-\Lambda p)} C_{N}(j)$. Consequently,
$c_{N}^{K}(j)\simeq 1+\frac{K}{N}\frac{\Lambda p}{(1-\Lambda p)}$ for $j\in \{1,...,K\}$ and
$c_{N}^{K}(j)\simeq \frac{K}{N}\frac{\Lambda p}{(1-\Lambda p)}$ for $j\in \{K+1,...,N\}$.
We finally get, recalling that $\ell_N(j)\simeq (1-\Lambda p)^{-1}$,  
\begin{align*}
\widetilde{\mathcal{W}}_{\Delta,t}^{N,K}\simeq &\mu\frac{N}{K^{2}}\sum_{j=1}^{N}\Big(c_{N}^{K}(j)\Big)^{2} \ell_{N}(j)\\
\simeq& \mu \frac{N}{K^2}\Big(\frac K {1-\Lambda p} \Big[1+\frac {K \Lambda p}{N(1-\Lambda p)}\Big]^2 
+ \frac {N-K} {1-\Lambda p} \Big[\frac {K \Lambda p}{N(1-\Lambda p)}\Big]^2 
\Big)\\
\simeq& \frac{\mu}{(1-\Lambda p)^{3}}+\frac{(N-K)\mu}{K(1-\Lambda p)}.
\end{align*}
All in all, we should have $\widetilde{\mathcal{X}}_{\Delta,t}^{N,K}\simeq \frac{\mu}{(1-\Lambda p)^{3}}.$

It readily follows that $\Psi(\widetilde{\varepsilon}^{N,K}_{t}, \widetilde{\mathcal{V}}^{N,K}_{t}, \widetilde{\mathcal{X}}^{N,K}_{\Delta, t})$
should resemble $(\mu,\Lambda,p)$.

\smallskip

The three estimators $\varepsilon^{N,K}_{t}, \mathcal{V}^{N,K}_{t}, \mathcal{X}^{N,K}_{\Delta, t}$ 
are very similar to 
$\widetilde{\varepsilon}^{N,K}_{t},\ \widetilde{\mathcal{V}}^{N,K}_{t},\ \widetilde{\mathcal{X}}^{N,K}_{\Delta, t}$ 
and should converge to the same limits.
Let us explain why we have introduced $\varepsilon^{N,K}_{t}, \mathcal{V}^{N,K}_{t}, \mathcal{X}^{N,K}_{\Delta, t}$,
of which the expressions are more complicated.
The main idea is that, see \cite[Lemma 16 (ii)]{A}, $\mathbb{E}[Z^{i,N}_t]= \mu \ell_N(i) t + \chi^N_i \pm t^{1-q}$
(under $(H(q))$), for some finite random variable $\chi^N_i$.
As a consequence, $t^{-1}\mathbb{E}[Z^{i,N}_{2t}-Z^{i,N}_t]$ converges to $\mu \ell_N(i)$ considerably much faster,
if $q$ is large, than $t^{-1}\mathbb{E}_\theta [Z^{i,N}_{t}]$ (for which the error is of order $t^{-1}$).

\subsection{The  supercritical  case}

We now turn to the supercritical case where $\Lambda p>1$.
We introduce the  $N\times N$ matrix $A_{N}(i,j)=N^{-1}\theta_{ij}$.

\smallskip

We expect that $Z_{t}^{i,N} \simeq H_{N}\mathbb{E}_{\theta}[Z_{t}^{i,N}]$, when $t$ is large, 
for some random $H_{N}>0$ not 
depending on $i$. 
Since $\Lambda p>1$, the process  should increase like an exponential function, i.e. there should be
$\alpha_N>0$ such that for all $i=1,\dots,N$,
$\mathbb{E}_{\theta}[Z_{t}^{i,N}]\simeq \gamma_{N}(i)e^{\alpha_{N}t}$ for $t$ very large,
where $\gamma_{N}(i)$ is some positive random constant.
We recall that $\mathbb{E}_{\theta}[Z_{t}^{i,N}]=\mu t+N^{-1}\sum_{j=1}^{N}\theta_{ij}\int_{0}^{t}\phi(t-s)
\mathbb{E}_{\theta}[Z_{s}^{j,N}]ds$. We insert $\mathbb{E}_{\theta}[Z_{t}^{i,N}]\simeq \gamma_{N}(i)e^{\alpha_{N}t}$ 
in this equation and let $t$ go to  infinite: we informally get 
$\boldsymbol{\gamma}_{N}=A_{N}\boldsymbol{\gamma}_{N}\int_{0}^{\infty}e^{-\alpha_{N}s}\phi(s)ds$. 
In other words, $\boldsymbol{\gamma}_{N}=(\gamma_N(i))_{i=1,\dots,N}$ is an eigenvector of $A_N$ for the eigenvalue 
$\rho_N:=(\int_{0}^{\infty}e^{-\alpha_{N}s}\phi(s)ds)^{-1}.$ 

\smallskip

But $A_N$ has nonnegative entries. Hence by the Perron-Frobenius theorem, it has a unique (up to normalization)
eigenvector $\boldsymbol{V}_{N}$ with nonnegative entries (say, such that $\|\boldsymbol{V}_{N}\|_{2}=\sqrt{N}$), 
and this vector corresponds to the maximum eigenvalue $\rho_N$ of $A_N$.
So there is a (random) constant $\kappa_N$ such that $\boldsymbol{\gamma}_{N}\simeq \kappa_N \boldsymbol{V}_{N}$.
All in all, we find that $Z_{t}^{i,N} \simeq \kappa_N H_{N} e^{\alpha_N t} \boldsymbol{V}_{N}(i)$.
We define $\boldsymbol{V}_N^K=I_K\boldsymbol{V}_N,$ where $I_K$ is the $N\times N$-matrix 
defined by $I_K(i,j)={\bf 1}_{\{i=j\leq K\}}$.

\smallskip

As in the subcritical case, the variance $K^{-1}\sum_{i=1}^{K}(Z_{t}^{i,N}-\bar{Z}_{t}^{N,K})^{2}$
should look like 
$$
\frac 1K\sum_{i=1}^{K}(\mathbb{E}_{\theta}[Z_{t}^{i,N}]
-\mathbb{E}_{\theta}[\bar{Z}_{t}^{N,K}])^{2}+\bar{Z}_{t}^{N,K}
\simeq \frac{\kappa_N^2H_N^2e^{2\alpha_{N}t}}K\sum_{i=1}^{K}(V_{N}(i)-\bar{V}_{N}^{K})^{2}+\bar{Z}_{t}^{N,K},
$$
where as usual $\bar V^K_N:= K^{-1}\sum_{i=1}^K V_N(i)$.
We  also get $\bar{Z}_{t}^{N,K}\simeq \kappa_N H_N \bar{V}_{N}^{K}e^{\alpha_{N}t}$. Finally,
$$
\mathcal{U}_{t}^{N,K}=\frac {N}{K (\bar{Z}_{t}^{N,K})^{2}}
[\sum_{i=1}^{K}(Z_{t}^{i,N}-\bar{Z}_{t}^{N,K})^{2}-K\bar{Z}_{t}^{N,K}]\boldsymbol{1}_{\{\bar{Z}_{t}^{N,K}>0\}}\simeq
\frac N {K (\bar{V}_{N}^{K})^{2}}\sum_{i=1}^{K}(V_{N}(i)-\bar{V}_{N}^{K})^{2}.
$$

Next, we  consider the term $(\bar{V}_{N}^{K})^{-2}\sum_{i=1}^{K}(V_{N}(i)-\bar{V}_{N}^{K})^{2}$.
By a rough estimation, $A^{2}_{N}(i,j)\simeq \frac{p^{2}}{N}$. Because  $I_{K}A_{N}^{2}\boldsymbol{V}_{N}=\rho_{N}^{2}\boldsymbol{V}_{N}^{K}$, we have $\rho_{N}^{2}\boldsymbol{V}_{N}^{K}\simeq p^{2}\bar{V}_{N}\boldsymbol{1}_{K}$, where $\boldsymbol{1}_K$ is the $N$ dimensional vector of which the first $K$ elements are $1$ and others are $0.$ By the same reason, we have $\rho_{N}^{2}\boldsymbol{V}_{N}\simeq p^{2}\bar{V}_{N}\boldsymbol{1}_{N}.$ So $\boldsymbol{V}_{N}^{K}=I_{K}A_{N}\boldsymbol{V}_{N}/\rho_{N}\simeq k_{N}I_{K}A_{N}\boldsymbol{1}_{N}$, where $k_{N}=(p^{2}/\rho_{N}^{3})\bar{V}_{N}$. In other words, the vector $(k_{N})^{-1}\boldsymbol{V}^{K}_{N}$ is almost like the vector $\boldsymbol{L}_{N}^{K}=I_{K}A_{N}\boldsymbol{1}_{N}$.
Finally,  we  expect that
$$
\mathcal{U}_{t}^{N,K}\simeq\frac{N}{K}(\bar{V}_{N}^{K})^{-2}\sum_{i=1}^{K}(V_{N}(i)-\bar{V}_{N}^{K})^{2}\simeq \frac{N}{K}(\bar{L}_{N}^{K})^{-2}\sum_{i=1}^{K}(L_{N}(i)-\bar{L}_{N}^{K})^{2}\simeq p^{-2}p(1-p)= \frac{1}{p}-1,
$$
whence ${\mathcal P}^{N,K}_t\simeq p$.

\section{Optimal rates in some toy models}

The goal of this section is to verify, using some toy models, that the rates of convergence
of our estimators, see Theorems \ref{abcd} and \ref{supercrit}, are not far from being optimal.

\subsection{The first example}
Consider $\alpha_0\ge 0$ and two unknown parameters $\Gamma>0$ and $p\in(0,1]$.
Consider an i.i.d. family  $(\theta_{ij})_{i,j=1...N}$ of Bernoulli($p$)-distributed random variables, 
where $N\ge$ 1.  We set $\lambda_{t}^{i,N}=N^{-1}\Gamma e^{\alpha_{0}t}\sum_{j=1}^{N}\theta_{ij}$ and we 
introduce the processes
$(Z_{t}^{1,N})_{t\ge 0}$, ...., $(Z_{t}^{N,N})_{t\ge 0}$ which are, conditionally on $(\theta_{ij})$, independent 
inhomogeneous Poisson process with intensities $(\lambda_{t}^{1,N})_{t\ge 0}$, ..., $(\lambda_{t}^{N,N})_{t\ge 0}$.
We only observe $(Z_{s}^{i,N})_{s\in[0,t],\,i=1,...K}$, where $K\le N$ and we want to estimate the parameter $p$ 
in the asymptotic $(K,N,t)\to (\infty,\infty,\infty)$.
This model is a simplified version of the one studied in our paper. And roughly speaking, 
the mean number of jumps per individuals until time $t$ resembles 
$m_{t}=\int_{0}^{t}e^{\alpha_{0}s}ds$.
When $\alpha_{0}=0$, this mimics the subcritical case, while when $\alpha_0>0$, this mimics the supercritical case.
Remark that $(Z_{t}^{i,N})_{i=1,...K}$ is a sufficient statistic, since $\alpha_0$ is known.

\vip

We use the central limit theorem in order to perform 
a Gaussian approximation of $Z_{t}^{i,N}$. It is easy to show that:
\begin{align*}
\lambda_{t}^{i,N}&= \Gamma e^{\alpha_{0}t}\Big[\frac{1}{\sqrt{N}}\sqrt{p(1-p)}\frac{1}{\sqrt{Np(1-p)}}\sum_{j=1}^{N}(\theta_{ij}-p)+p\Big]
\end{align*}
and $\frac{1}{\sqrt{Np(1-p)}}\sum_{j=1}^{N}(\theta_{ij}-p)$ converges in law to a Gaussian random variable $G_{i}\sim \mathcal{N}(0,1)$, where $G_{i}$ is an i.i.d Gaussian family, as $N\to\infty$, for each $i$. Thus 
\begin{align*}
\lambda_{t}^{i,N}\simeq \Gamma e^{\alpha_{0}t}[\sqrt{N^{-1}p(1-p)}G_{i}+p].
\end{align*}
Moreover, conditionally on $(\theta_{ij})_{i,j=1,\dots,N}$, $Z_{t}^{i,N}$ is a Poisson random variable with mean $\int_{0}^{t}\lambda_{s}^{i,N}ds$. Thus, as $t$ is large, we have
$Z^{i,N}_t\simeq \int_0^t \lambda^{i,N}_s ds  + \sqrt{\int_0^t \lambda^{i,N}_s ds}H_i$
where $(H_i)_{i=1,...,N}$ is a family of $\mathcal{N}(0,1)$-distributed random variables, independent of 
$(G_i)_{i=1,\dots,N}$.
Since $(m_t)^{-1}N^{-1/2}\ll(m_t)^{-1}$,
we obtain
$(m_t)^{-1}Z^{i,N}_t \simeq \Gamma p + \Gamma \sqrt{N^{-1}p(1-p)}G_i + \sqrt{(m_t)^{-1}\Gamma p} H_i$,
of which the law is nothing but $\mathcal{N}( \Gamma p, N^{-1}\Gamma^2p(1-p)+ (m_t)^{-1}\Gamma p )$.

\vip

By the above discussion, we construct the following toy model: one observes $(X_{t}^{i,N})_{i=1,...K}$, where $(X_{t}^{i,N})_{i=1,...N}$ are i.i.d and $\mathcal{N}(\Gamma p, N^{-1}\Gamma^{2}p(1-p)+(m_{t})^{-1}\Gamma p)$-distributed.  
Moreover we assume that $\Gamma p$ is known.
So we can use the well-known statistic result: the empirical variance 
$S_{t}^{N,K}=K^{-1}\sum_{i=1}^{K}(X_{t}^{i,N}-\Gamma p)^{2}$ is the best estimator of 
$N^{-1}\Gamma^{2}p(1-p)+(m_{t})^{-1}\Gamma p$ (in any reasonnable sense). 
So $T_{t}^{N,K}= N(\Gamma p)^{-2}(S_{t}^{N,K}-(\Gamma p)/m_{t}) $ 
is the best estimator of $(\frac{1}{p}-1)$. As
\begin{align*}
\mathrm{Var} (S_{t}^{N,K})=\frac{1}{K}\mathrm{Var} [(X_{t}^{1,N}-\Gamma p)^{2}]=\frac{2}{K}\Big(\frac{\Gamma^{2}p(1-p)}N
+\frac{\Gamma p}{m_{t}}\Big)^{2},
\end{align*}
we have
\begin{align*}
\mathrm{Var} (T_{t}^{N,K})=\frac2{(\Gamma p)^{4}}\Big(\frac{\Gamma^{2}p(1-p)}{\sqrt{K}}+\frac{N\Gamma p}{m_{t}\sqrt{K}}
\Big)^2.
\end{align*}
In other words, we cannot estimate  $\Big(\frac{1}{p}-1\Big)$ with a precision better than $\Big(\frac{1}{\sqrt{K}}+\frac{N}{m_{t}\sqrt{K}}\Big),$ which implies that we cannot estimate $p$ with a precision better than $\Big(\frac{1}{\sqrt{K}}+\frac{N}{m_{t}\sqrt{K}}\Big)$.

\smallskip
\subsection{The second example}

In the second part of this section, we are going to explain why there is a term $\frac{N}{K\sqrt{t^{1-\frac{4}{1+q}}}}$ in the  subcritical case.

\vip

We consider discrete times $t=1,...,T$ and two unknown parameters $\mu>0$ and $p\in(0,1]$.
Consider an i.i.d. family  $(\theta_{ij})_{i,j=1...N}$ of Bernoulli($p$)-distributed random variables, 
where $N\ge 1$. 
We set $Z^{i,N}_0=0$ for all $i=1,\dots,N$ and assume that, conditionally
on $(\theta_{ij})_{i,j=1,\dots N}$ and $(Z_{s}^{j,N})_{s=0,\dots,t,j=1\dots,N}$, the random variables 
$(Z_{t+1}^{i,N}-Z_{t}^{i,N})$ (for $i=1,\dots,N$) are independent and $\mathcal{P}(\lambda_{t}^{i,N})$-distributed, 
where $\lambda_{t}^{i,N}=\mu +\frac{1}{N}\sum_{j=1}^{N}\theta_{ij}(Z_{t}^{j,N}-Z_{t-1}^{j,N})$. 
This process $(Z_{t}^{i,N})_{i=1,\dots,N,t=0,\dots T}$ resembles the system of Hawkes processes studied 
in the present paper. 

\vip

 By \cite [theorem 2]{I}, we have when time $t$ is large, the process $\boldsymbol{Z}_{t}^{N}$ is similar to a d-dimensional diffusion process  $(I-A_{N})^{-1}\Sigma^\frac{1}{2}\boldsymbol{B}_{t}+\mathbb{E}_{\theta}[\boldsymbol{Z}_{t}^{N}]$, where $\boldsymbol{B}_{t}$ is a N-dimensional Brownian Motion and $\Sigma$ is the diagonal matrix such that $\Sigma_{ii}=((I-A_N)^{-1}\mu)_i.$  Hence $(Z_{t+1}^{i,N}-Z_{t}^{i,N})-\Et[Z_{t}^{i,N}-Z_{t-1}^{i,N}]$  (for $i=1,\dots,N$ and $t=1,...,T$) are independent.
 Since  $\mathbb{E}_{\theta}[\boldsymbol{Z}_{t}^{N}]$ is similar to $\frac{\mu t}{1-p}$ when both $N$ and $t$ are large. Hence $\lambda_{t}^{i,N}\simeq \Et[\lambda_{t}^{i,N}]\simeq \frac{\mu}{1-p}.$ Then by Gaussian approximation,  we can roughly replace $(Z_{t}^{j,N}-Z_{t-1}^{j,N})_{j=1,\dots,N}$ in the expression of
$(\lambda_{t}^{i,N})_{i=1,\dots,N}$ by $(\frac{\mu}{1-p}+Y^{j,N}_t)_{j=1,\dots,N}$, for an i.i.d. array $(Y^{j,N}_t)_{j=1,\dots,N,t=1,\dots,T}$
of $\mathcal{N}(0,\frac{\mu}{1-p})$-distributed random variables. Also, we replace the $\mathcal{P}(\lambda_{t}^{i,N})$
law by its Gaussian approximation.

\vip

We thus introduce the following model, with unknown parameters $\mu>0$ and $p\in (0,1)$.
We start with three independent families of i.i.d. random variables, namely 
$(\theta_{ij})_{i,j=1,\dots,N}$ with law Bernoulli$(p)$, and 
$(Y^{j,N}_t)_{j=1,\dots,N,t=1,\dots,T}$ with law $\mathcal{N}(0,\frac{\mu}{1-p})$ and $(A^{j,N}_t)_{j=1,\dots,N,t=1,\dots,T}$ with law $\mathcal{N}(0,1)$.
We then set, for each $t=1,\dots,T$ and each $i=1,\dots,N$, 
$$
a_{t}^{i,N}=\mu+\frac{1}{N}\sum_{j=1}^{N}\theta_{ij}\Big(\frac{\mu}{1- p}+Y_{t}^{j,N}\Big) \quad \hbox{and}\quad 
X_t^{i,N}=a_{t}^{i,N}+ \sqrt{a_{t}^{i,N}}A^{i,N}_t.
$$

We compute the covariances. First, for all $i=1,\dots,N$ and all $t=1,\dots,T$,
\begin{align*}
\mathrm{Var}(X_{t}^{i,N})
&=\mathbb{E}[(a_{t}^{i,N}+\sqrt{a_{t}^{i,N}}A_{t}^{i,N}-\frac{\mu}{1- p})^{2}]
\\&=\mathbb{E}\Big[\Big(\frac{\mu}{N(1-p)}\sum_{k=1}^{N}(\theta_{ik}-p)+\frac{1}{N}\sum_{k=1}^{N}\theta_{ik}Y_{t}^{k,N}+\sqrt{a_{t}^{i,N}}A_{t}^{i,N}\Big)^{2}\Big]\\
&=\frac{p\mu^2}{N(1-p)}+\frac{p\mu^2}{N(1-p)^2}+\frac{\mu}{(1-p)}.
\end{align*}
Next, for $i\ne j$ and all $t=1,\dots,T$,
\begin{align*}
\mathrm{Cov}(X_{t}^{i,N}, X_{t}^{j,N})
&=\mathbb{E}\Big[\Big(a_{t}^{i,N}+\sqrt{a_{t}^{i,N}}A_{t}^{i,N}-\frac{\mu}{(1-p)}\Big)\Big(a_{t}^{j,N}+\sqrt{a_{t}^{j,N}}A_{t}^{j,N}-\frac{\mu}{(1-p)}\Big)\Big]
\\&=\mathbb{E}\Big[\frac{1}{N^{2}}\sum_{k=1}^{N}\theta_{jk}\theta_{ik}(Y_{t}^{k,N})^{2}\Big]
=\frac{p^{2}}{N}\frac{\mu^2}{(1-p)^2}.
\end{align*}
For $s\ne t$ and $i=1,\dots,N$,

\begin{align*}
\mathrm{Cov}(X_{t}^{i,N}, X_{s}^{i,N})
=&\mathbb{E}\Big[\Big(a_{t}^{i,N}+\sqrt{a_{t}^{i,N}}A_{t}^{i,N}-\frac{\mu}{(1-p)}\Big)\Big(a_{s}^{i,N}+\sqrt{a_{s}^{i,N}}A_{s}^{i,N}-\frac{\mu}{(1-p)}\Big)\Big]\\
=&\Big(\frac{\mu}{1-p}\Big)^2\mathrm{Var} \Big(\frac{1}{N}\sum_{j=1}^{N}\theta_{ij} \Big)
=\frac{p\mu^2}{N(1-p)}.
\end{align*}
Finally, for $s\ne t$ and $i\ne j$,
\begin{align*}
\mathrm{Cov}(X_{t}^{i,N}, X_{s}^{j,N})
&=\mathbb{E}\Big[\Big(a_{t}^{i,N}+\sqrt{a_{t}^{i,N}}A_{t}^{i,N}-\mu-p\Big)\Big(a_{s}^{j,N}+\sqrt{a_{s}^{j,N}}A_{t}^{j,N}-\mu-p\Big)\Big]
=0.
\end{align*}
Over all we have $\mathrm{Cov}(X_{t}^{i,N}, X_{s}^{j,N})=C_{\mu,p,N}((i,t),(j,s))$, where
\begin{align*}
C_{\mu,p,N}((i,t),(j,s))=
\begin{cases} 
\frac{p\mu^2}{N(1-p)}+\frac{p\mu^2}{N(1-p)^2}+\frac{\mu}{(1-p)} &\hbox{if} \;\; i= j,\ t= s,\\
\frac{p^{2}}{N}\frac{\mu^2}{(1-p)^2}  &\hbox{if} \;\;i\ne j,\ t= s,\\
\frac{p\mu^2}{N(1-p)} &\hbox{if} \;\;i=j,\ t\ne s,\\
0 &\hbox{if} \;\; i\ne j,\ t\ne s.
\end{cases}
\end{align*}
Form the covariance function above, we can ignore the covariance when $t\ne s.$ So, we construct a new covariance function:
\begin{align*}
\widetilde{C}_{\mu,p,N}((i,t),(j,s))=
\begin{cases} 
\frac{p\mu^2}{N(1-p)}+\frac{p\mu^2}{N(1-p)^2}+\frac{\mu}{(1-p)} &\hbox{if} \;\; i= j,\ t= s,\\
\frac{p^{2}}{N}\frac{\mu^2}{(1-p)^2}  &\hbox{if} \;\;i\ne j,\ t= s,\\
0 &\hbox{if} \;\;i=j,\ t\ne s,\\
0 &\hbox{if} \;\; i\ne j,\ t\ne s.
\end{cases}
\end{align*}
We thus consider the following toy model: for two unknown parameters $\mu>0$ and $p\in (0,1)$,
we observe $(U^{i,N}_s)_{i=1,\dots,K,s=0,\dots,T}$, for some Gaussian array $(U^{i,N}_s)_{i=1,\dots,N,s=0,\dots,T}$
with covariance matrix $\widetilde{C}_{\mu,p,N}$ defined above and we want to estimate $p$.
If assuming that $\frac{\mu}{1-p}$ is known, it is well-known that
the temporal empirical variance 
$S_{T}^{N,K}=\frac{1}{T}\sum_{t=1}^{T}(\bar{U}_{t}^{N,K}-\frac{\mu}{1-p})^{2}$, where 
$\bar{U}_{t}^{N,K}=\frac{1}{K}\sum_{i=1}^{K}U_{t}^{i,N}$, is the best estimator of 
$\frac{(2p-p^{2})\mu^2}{NK(1-p)^2}+\frac{\mu}{K(1-p)}+\frac{p^2(K-1)}{NK}\frac{\mu^2}{(1-p)^2},$  (in all the usual senses).
Consequently, 
$C_{T}^{N,K}=\frac{N}{K-1}(\frac{\mu}{1-p})^{-2}[KS_{T}^{N,K}-\frac{\mu}{1-p}]$ is the best estimator of $p^2$. 
And 
$$
\mathrm{Var}(C_{T}^{N,K})=\frac{1}{T}\frac{N^{2}}{(K-1)^{2}}K^{2}\frac{1}{K^{2}}
\Big[\rho+\frac{(K-1)\alpha}{N}\Big]^{2}\simeq \frac{N^2}{TK^2}.
$$
where $\rho=\frac{(2p-p^{2})\mu^2}{N(1-p)^2}+\frac{\mu}{(1-p)}$ and $\alpha=\frac{p^{2}\mu^2}{(1-p)^2}$
Hence for this Gaussian toy model, it is not possible to estimate $p^2$ (and thus $p$) 
with a precision better than $\frac{N}{K}\frac{1}{\sqrt{T}}$.

\subsection{Conclusion}
Using the first example, it seems that it should not be possible to estimate $p$ faster than
$N/(\sqrt{K} e^{\alpha_0 t}) + 1/\sqrt K$. in the supercritical case.
Using the two examples, it seems that it should not be possible to estimate $p$ faster than
$N/(t\sqrt{K}) + 1/\sqrt K + N/(K\sqrt t)$ in the subcritical case.

\section{Analysis of a random matrix in the subcritical case}
\subsection{Some notations}
For $r\in [1,\infty)$ and $x\in \R^N$, we set $\|\boldsymbol{x}\|_{r}=(\sum_{i=1}^{N}|x_{i}|^{r})^{\frac{1}{r}}$, and$\ \|\boldsymbol{x}\|_{\infty}=\max_{i=1...N}|x_{i}|$. For $M$ a $N\times N$ matrix, we denote by $|||M|||_{r}$ is the operator norm associated to $ \|\cdot \|_{r}$, that is $|||M|||_{r}=\sup_{\boldsymbol{x}\in R^{n}}\|M\boldsymbol{x}\|_{r}/\|\boldsymbol{x}\|_{r}$. We  have the special cases
$$
|||M|||_{1}=\sup_{\{j=1,...,N\}}\sum_{i=1}^{N}|M_{ij}|,\quad |||M|||_{\infty}=\sup_{\{i=1,...,N\}}\sum_{j=1}^{N}|M_{ij}|.
$$
We also have the inequality
$$
|||M|||_{r}\le|||M|||_{1}^{\frac{1}{r}}|||M|||_{\infty}^{1-\frac{1}{r}}\quad \hbox{for any}\quad r\in [1,\infty).
$$

We define $A_{N}(i,j):=N^{-1}\theta_{ij}$ and the matrix $(A_{N}(i,j))_{i,j\in\{1,...,N\}}$, as well as 
$Q_{N}:=(I-\Lambda A_{N})^{-1}$ on the event on which $I-\Lambda A_N$ is invertible.

\vip
Next, we are going to give the event $\Omega_{N,K}$ , which  we mainly work on it in this paper.

For $1\leq K \leq N$, we introduce the $N$-dimensional vector $\boldsymbol{1}_K$ defined by
$\boldsymbol{1}_K(i)=\indiq_{\{1\leq i\leq K\}}$ for $i=1,\dots,N$, and the $N\times N$-matrix $I_K$  
defined by $I_K(i,j)={\bf 1}_{\{i=j\leq K\}}$.

\vip

We assume here that $\Lambda p\in\ (0,1)$ and we set 
$a=\frac{1+\Lambda p}{2}\in\ (0,1).$ Next, we introduce  the events
\begin{eqnarray*}
&\Omega_{N}^{1}:=\Big\{\Lambda |||A_{N}|||_{r}\le a ,\ \hbox{for all } \  r\in[1,\infty]\Big\},\quad\\ &\mathcal{F}_{N}^{K,1}:=\Big\{\Lambda |||I_{K}A_{N}|||_{r}\le \Big(\frac{K}{N}\Big)^{\frac{1}{r}}a, \hbox{for}\  \hbox{all}\  r\in [1,\infty)\Big\},\\
&\mathcal{F}_{N}^{K,2}:=\Big\{\Lambda |||A_{N}I_{K}|||_{r}\le \Big(\frac{K}{N}\Big)^{\frac{1}{r}}a,\ \hbox{for}\   \hbox{all}\  r\in [1,\infty)\Big\},\\
&\Omega^{1}_{N,K}:= \Omega^1_N \cap\mathcal{F}_{N}^{K,1}, \quad \Omega^{1}_{N,K}:= \Omega^1_N \cap\mathcal{F}_{N}^{K,2},
\quad  \Omega_{N,K}=\Omega^{1}_{N,K}\cap\Omega^{2}_{N,K}.
\end{eqnarray*}

We set 
$\boldsymbol{\ell}_N:=Q_N\boldsymbol{1}_N,\ \ell_{N}(i):=\sum_{j=1}^{N}Q_{N}(i,j),\ \bar{\ell}_N:=\frac{1}{N}\sum_{i=1}^N\ell_N(i),\ \bar{\ell}^K_N:=\frac{1}{K}\sum_{i=1}^K\ell_N(i)$.

\vip

We also set $c^K_{N}(j):=\sum_{i=1}^{K}Q_{N}(i,j),\
\bar{c}^K_N:=\frac{1}{N}\sum_{j=1}^Nc^K_N(j)$.

\vip

We let $\boldsymbol{L}_N:=A_N\boldsymbol{1}_N,\ L_{N}(i):=\sum_{j=1}^{N}A_{N}(i,j),\ \bar{L}_N:=
\frac{1}{N}\sum_{i=1}^NL_N(i),\ \bar{L}^K_N:=\frac{1}{K}\sum_{i=1}^KL_N(i)$ and
$\boldsymbol{C}_N:=A_N^*\boldsymbol{1}_N,\ C_{N}(j):=\sum_{i=1}^{N}A_{N}(i,j),\
\bar{C}_N:=\frac{1}{N}\sum_{j=1}^NC_N(j),\ \bar{C}^K_N:=\frac{1}{K}\sum_{j=1}^KC_N(i)$
and consider the event
$$
\mathcal{A}_{N}:=\{\|\boldsymbol{L}_{N}-p\boldsymbol{1}_{N}\|_{2}+\|\boldsymbol{C}_{N}-p\boldsymbol{1}_{N}\|_{2}\le N^{\frac{1}{4}}\}.
$$ 
where
$\boldsymbol{L}_{N}$ is the vectors $(L_N(i))_{i=1,\dots,N}.$
We also set $x_{N}(i)=\ell_{N}(i)-\bar{\ell}_{N}$,  $\boldsymbol{x}_{N}=(x_N(i))_{i=1,\dots,N}$,
$X_{N}(i)=L_{N}(i)-\bar{L}_{N}$ and $\boldsymbol{X}_{N}=(X_N(i))_{i=1,\dots,N}.$
We finally put $X_{N}^{K}(i)=(L_{N}(i)-\bar{L}^{K}_{N})\indiq_{\{i\leq K\}}$ and $\boldsymbol{X}^{K}_{N}=
(X_N^K(i))_{i=1,\dots,N}=\boldsymbol{L}_{N}^K- \bar{L}^{K}_{N}\boldsymbol{1}_K$, as well as
$x_{N}^{K}(i)=(\ell_{N}(i)-\bar{\ell}^{K}_{N})\indiq_{\{i\leq K\}}$ and $\boldsymbol{x}^{K}_{N}=
(x_N^K(i))_{i=1,\dots,N}=\boldsymbol{\ell}_{N}^K- \bar{\ell}^{K}_{N}\boldsymbol{1}_K$.

\subsection{Review of some lemmas found in \cite{A}}
In this subsection we recall results  from \cite{A} showing that $\mathcal{A}_{N}$ and $\Omega_{N}^{1}$ are big, and  upper-bounds concerning $x_n$ and $\boldsymbol{X}_{N}$.
\begin{lemma}\label{lo}
We assume that $\Lambda p<1.$ Then
$\Omega_{N,K}\subset\Omega_{N}^{1}\subset\{|||Q_{N}|||_{r}\le C, $ for all $r \in[1,\infty]\}\subset\{\sup_{i=1...N}\ell_{N}(i)\le C\}$, where $C=(1-a)^{-1}$. For any $\alpha>0$, there exists a constant $C_{\alpha}$ such that 
$$P(\mathcal{A}_{N})\ge 1-C_{\alpha}N^{-\alpha}.$$

\end{lemma}
\begin{proof}
See \cite[Notation 12 and Proposition 14, Step 1]{A}.
\end{proof}

\begin{lemma}\label{omegaN1}
Assume that $\Lambda p<1$. Then,    $$P(\Omega_{N}^{1})\ge1-C\exp(-cN)$$ 
for some constants $C>0$ and $c>0$.
\end{lemma} 

\begin{proof} 
See  \cite[Lemma 13]{A}.
\end{proof}

\begin{lemma}\label{lll}
Assume that $\Lambda p<1$. Then
$$
\mathbb{E}\Big[\boldsymbol{1}_{\Omega_{N}^{1}}\Big|\bar{\ell}_{N}-\frac{1}{1-\Lambda p}\Big|^{2}\Big]\le\frac{C}{N^{2}}.
$$
\end{lemma}

\begin{proof} 
See \cite[Proposition 14]{A}.
\end{proof}

\begin{lemma}\label{xxx}
Assume that $\Lambda p<1$, set $b=\frac{2+\Lambda p}{3}$ and consider $N_{0}$ the smallest integer such that
$a+\Lambda N_{0}^{-\frac{1}{4}}\le b$. For all $N\ge N_{0},$
$$
(i)\boldsymbol{1}_{\Omega_{N}^{1}\cap\mathcal{A}_{N}}\|\boldsymbol{x}_{N}\|_{2}\le C\|\cX_{N}\|_{2}, \quad (ii)\mathbb{E}[\|\boldsymbol{X}_{N}\|_{2}^{4}]\le C, \quad (iii)\mathbb{E}[\|A_{N}\boldsymbol{X_{N}}\|^{2}_{2}]\le CN^{-1}.
$$
\end{lemma}

\begin{proof} 
See \cite[Proof of Proposition 14, Steps 2 and 4]{A}.
\end{proof}
\begin{remark}
In Lemma \ref{xxx}, the condition   $\Lambda p<1$ is not necessary for $(ii)$ and $(iii).$
\end{remark}
\begin{lemma}\label{EZ}
Assume that $\Lambda p<1$ and set $k:=\Lambda^{-1}\int_{0}^{\infty}s\phi(s)ds$, then for $n\ge 0$, $t\ge 0,$
$$
\int_{0}^{t}s\phi^{*n}(t-s)ds=\Lambda^{n}t-n\Lambda^{n}k+\varepsilon_{n}(t),
$$
where $0\le\varepsilon_{n}(t)\le C\min\{n^{q}\Lambda^{n}t^{1-q},n\Lambda^{n}k\}$ and where $\phi^{*n}(s)$ is the n-times convolution of $\phi$. We adopt the convention that $\phi^{*0}(s)ds=\delta_0(ds)$, whence in particular
$\int_{0}^{t}s\phi^{*0}(t-s)ds=t$.
\end{lemma}

\begin{proof} 
See \cite[Lemma 15]{A}.
\end{proof}
\smallskip

\subsection{Probabilistic lower bound}
In this subsection, we are going to prove that the set $\Omega_{N,K}$ has high probability, which will allow to work on the set $\Omega_{N,K}$ for all our study.
\begin{lemma}\label{ONK}
Assume that $\Lambda p<1$. It holds that $$P(\Omega_{N,K})\ge 1-CNe^{-cK}$$ for some constants $C>0$ and $c>0$.
\end{lemma}

\begin{proof}
On $\Omega_{N,K}^{1}$, we have 
$$
N|||I_{K}A_{N}|||_{1}=\sup_{j=1,...,N}\sum_{i=1}^{K}\theta_{ij}=\max\{X_{1}^{N,K},... ,X_{N}^{N,K}\},
$$ 
where $X_{i}^{N,K}=\sum_{j=1}^{K}\theta_{ij}$ for $i=1,...,N$ are i.i.d and Binomial$(K,p)$-distributed.
So, 
\begin{eqnarray*}
P\Big(\Lambda\frac{N}{K}|||I_{K}A_{N}|||_{1}\ge a\Big)&=&P\Big(\max\{X_{1}^{N,K},...X_{N}^{N,K}\}\ge\frac{Ka}{\Lambda}\Big)  \le NP\Big(X_{1}^{N,K}\ge\frac{Ka}{\Lambda}\Big)\\
&\le& NP\Big(|X_{1}^{N,K}-Kp|\ge K\Big(\frac{a}{\Lambda}-p\Big)\Big)\le 2Ne^{-2K(\frac{a}{\Lambda}-p)^{2}}.
\end{eqnarray*}
The last equality follows from Hoeffding inequality.
On the event $\Omega_{N}^{1}\cap \{\Lambda\frac{N}{K}|||I_{K}A_{N}|||_{1}\le a\},$ 
we have 
$$
|||I_{K}A_{N}|||_{r}\le|||I_{K}A_{N}|||_{1}^{\frac{1}{r}}\|I_{K}A_{N}|||_{\infty}^{1-\frac{1}{r}}\le|||I_{K}A_{N}|||_{1}^{\frac{1}{r}}||A_{N}|||_{\infty}^{1-\frac{1}{r}}\le\Big(\frac{a}{\Lambda}\frac{K}{N}\Big)^{\frac{1}{r}}\Big(\frac{a}{\Lambda}\Big)^{1-\frac{1}{r}}=\frac{a}{\Lambda}\Big(\frac{K}{N}\Big)^{\frac{1}{r}}.
$$
We conclude that $\Omega^{1}_{N,K}=\Omega_{N}^{1}\cap\{(\frac{N}{K})|||I_{K}A_{N}|||_{1}\le a\}$. And by Lemma \ref{omegaN1},  we  deduce that $P(\Omega^{1}_{N,K})\ge 1-CNe^{-cK}.$
By the same way, we prove that $P(\Omega^{2}_{N,K})\ge 1-CNe^{-cK}.$  
Finally by the definition of $\Omega_{N,K}$, we have $P(\Omega_{N,K})\ge 1-CNe^{-cK}.$
\end{proof}

\subsection{Matrix analysis for the first estimator}
The aim of this subsection is to prove that $\bar \ell_N^K \simeq 1/(1-\Lambda p)$
and to study the rate of convergence.

\begin{lemma}\label{lA}
Assume $\Lambda p<1$. Then
$$
\mathbb{E}\Big[\boldsymbol{1}_{\Omega_{N,K}}|\bar{\ell}_{N}^{K}-1-\Lambda p\bar{\ell}_{N}|^{2}\Big]\le\frac{C}{NK}.
$$
\end{lemma}

\begin{proof}
Recall that $\boldsymbol{\ell}_{N}=Q_{N}\boldsymbol{1_{N}}$, whence $Q_{N}^{-1}\boldsymbol{\ell}_{N}=\boldsymbol{1}_{N}$. And since, $Q_{N}=(I-\Lambda A_{N})^{-1}$, we have $Q_{N}^{-1}\boldsymbol{\ell}_{N}=(I-\Lambda A_{N})\boldsymbol{\ell}_{N}=\boldsymbol{1}_{N}$ and thus $\boldsymbol{\ell}_{N}=\boldsymbol{1_{N}}+\Lambda A_{N}\ell_{N}$. We conclude that
\begin{align*}
  &\bar{\ell}_{N}^{K}=\frac{1}{K}(\boldsymbol{\ell}_{N},\boldsymbol{1}_{K})=1+\frac{\Lambda}{K}\sum^{K}_{i=1}\sum_{j=1}^{N}A(i,j)\ell_{N}(j)=1+\frac{\Lambda}{K}\sum_{j=1}^{N}C_{N}^{K}(j)\ell_{N}(j),
\end{align*}
where $C_{N}^{K}(j):=\sum_{i=1}^{K}A(i,j)=\frac{1}{N}\sum_{i=1}^{K}\theta_{ij}.$
By some easy computing, we have
\begin{equation}\label{sde}
\mathbb{E}\Big[\Big(\sum_{j=1}^{N}\Big[C_{N}^{K}(j)-\frac{Kp}{N}\Big]^{2}\Big)^{2}\Big]\le C\frac{K^{2}}{N^{2}},
\end{equation}
whence 
\begin{align*}
&\mathbb{E}\Big[\boldsymbol{1}_{\Omega_{N,K}\cap\mathcal{A}_{N}}\Big|\bar{\ell}_{N}^{K}-1-\Lambda p\bar{\ell}_{N}\Big|^{2}\Big]\\
=&\mathbb{E}\Big[\boldsymbol{1}_{\Omega_{N,K}\cap\mathcal{A}_{N}}\Big|\frac{\Lambda}{K}\sum_{j=1}^{N}\Big(C_{N}^{K}(j)-\frac{Kp}{N}\Big)\ell_{N}(j)\Big|^{2}\Big]
\\
=&\mathbb{E}\Big[\boldsymbol{1}_{\Omega_{N,K}\cap\mathcal{A}_{N}}\Big|\frac{\Lambda}{K}\sum_{j=1}^{N}\Big(C_{N}^{K}(j)-\frac{Kp}{N}\Big)(\ell_{N}(j)-\bar{\ell}_{N})+\bar{\ell}_{N}\frac{\Lambda}{K}\sum_{j=1}^{N}\Big(C_{N}^{K}(j)-\frac{Kp}{N}\Big)\Big|^{2}\Big]
\\
\le&
2\mathbb{E}\Big[\boldsymbol{1}_{\Omega_{N,K}\cap\mathcal{A}_{N}}\Big|\frac{\Lambda}{K}\sum_{j=1}^{N}\Big(C_{N}^{K}(j)-\frac{Kp}{N}\Big)(\ell_{N}(j)-\bar{\ell}_{N})\Big|^{2}\Big]\\&+2\mathbb{E}\Big[\boldsymbol{1}_{\Omega_{N,K}\cap\mathcal{A}_{N}}\Big|\bar{\ell}_{N}\frac{\Lambda}{K}\sum_{j=1}^{N}\Big(C_{N}^{K}(j)-\frac{Kp}{N}\Big)\Big|^{2}\Big].
\end{align*}
Consequently,
\begin{align*}
&\mathbb{E}\Big[\boldsymbol{1}_{\Omega_{N,K}\cap\mathcal{A}_{N}}\Big|\bar{\ell}_{N}^{K}-1-\Lambda p\bar{\ell}_{N}\Big|^{2}\Big]\\
  \le& 
2\Big(\frac{\Lambda}{K}\Big)^{2}\mathbb{E}\Big[\boldsymbol{1}_{\Omega_{N,K}\cap\mathcal{A}_{N}}\|\boldsymbol{x}_{N}\|^{4}_{2}\Big]^{\frac{1}{2}}\mathbb{E}\Big[\boldsymbol{1}_{\Omega_{N,K}\cap\mathcal{A}_{N}}\Big(\sum_{j=1}^{N}\Big(C_{N}^{K}(j)-\frac{Kp}{N}\Big)^{2}\Big)^{2}\Big]^{\frac{1}{2}}\\
&\ +2\frac{\Lambda^{2}}{K^{2}}\mathbb{E}\Big[\boldsymbol{1}_{\Omega_{N,K}\cap\mathcal{A}_{N}}\Big|\bar{\ell}_{N}\Big|^{2}\Big|\sum_{j=1}^{N}\Big(C_{N}^{K}(j)-\frac{Kp}{N}\Big)\Big|^{2}\Big]
\\
\le &
\frac{C}{NK}\mathbb{E}\Big[\boldsymbol{1}_{\Omega_{N,K}\cap\mathcal{A}_{N}}\|\boldsymbol{x}_{N}\|^{4}_{2}\Big]^{\frac{1}{2}}
+2\frac{\Lambda^{2}}{K^{2}}\mathbb{E}\Big[\boldsymbol{1}_{\Omega_{N,K}\cap\mathcal{A}_{N}}\Big|\bar{\ell}_{N}\Big|^{2}\Big|\sum_{j=1}^{N}\Big(C_{N}^{K}(j)-\frac{Kp}{N}\Big)\Big|^{2}\Big].
\end{align*}
By Lemma \ref{xxx}, we know that $
\mathbb{E}[\boldsymbol{1}_{\Omega_{N,K}\cap\mathcal{A_{N}}}\|\boldsymbol{x}_{N}\|^{4}_{2}]\le C.$
By Lemma \ref{lo}, $\bar{\ell}_{N}$ and $\bar{\ell}_{N}^{K}$ are bounded  on the set $\Omega_{N,K}$,
whence, recalling \eqref{sde}, and since $\{C_{N}^{K}(j)-\frac{Kp}{N}\}_{j=1,...,N}$ are independent, we conclude that 
\begin{align*}
    \mathbb{E}\Big[\boldsymbol{1}_{\Omega_{N,K}\cap\mathcal{A}_{N}}\Big|\bar{\ell}_{N}\Big|^{2}\Big|\sum_{j=1}^{N}\Big(C_{N}^{K}(j)-\frac{Kp}{N}\Big)\Big|^{2}\Big]\le& C\mathbb{E}\Big[\boldsymbol{1}_{\Omega_{N,K}\cap\mathcal{A}_{N}}\Big|\sum_{j=1}^{N}\Big(C_{N}^{K}(j)-\frac{Kp}{N}\Big)\Big|^{2}\Big]\\
    \le & C\mathbb{E}\Big[\sum_{j=1}^{N}\Big(C_{N}^{K}(j)-\frac{Kp}{N}\Big)^{2}\Big]\\
    \le& \frac{CK}{N}.
\end{align*}
Hence
$$
\mathbb{E}\Big[\boldsymbol{1}_{\Omega_{N,K}\cap{\mathcal{A_{N}}}}\Big|\bar{\ell}_{N}^{K}-1-\Lambda p\bar{\ell}_{N}\Big|^{2}\Big]\le\frac{C}{NK}.
$$
We finally apply Lemma \ref{lo} with e.g. $\alpha=2$ and get
\begin{align*}
\mathbb{E}\Big[\boldsymbol{1}_{\Omega_{N,K}}\Big|\bar{\ell}_{N}^{K}-1-\Lambda p\bar{\ell}_{N}\Big|^{2}\Big]
&=\mathbb{E}\Big[\boldsymbol{1}_{\Omega_{N,K}\cap{\mathcal{A_{N}}}}\Big|\bar{\ell}_{N}^{K}-1-\Lambda p\bar{\ell}_{N}\Big|^{2}\Big]+\mathbb{E}\Big[\boldsymbol{1}_{\Omega_{N,K}\cap{\mathcal{A_{N}^{C}}}}\Big|\bar{\ell}_{N}^{K}-1-\Lambda p\bar{\ell}_{N}\Big|^{2}\Big]\\
&\le\frac{c}{NK}+\frac{C}{N^2}\le\frac{C}{NK}.
 \end{align*}
\end{proof}

The next lemma is the main result of the subsection.

\begin{lemma}\label{ellp}
If $\Lambda p<1$, we have
$$
\mathbb{E}\Big[\boldsymbol{1}_{\Omega_{N,K}}\Big|\bar{\ell}_{N}^{K}-\frac{1}{1-\Lambda p}\Big|^{2}\Big]\le\frac{C}{NK}.
$$
\end{lemma}
\begin{proof}
Observing that $1/(1-\Lambda p)= 1+ \Lambda p/(1-\Lambda p)$, we write
\begin{align*}
  \mathbb{E}\Big[\boldsymbol{1}_{\Omega_{N,K}}\Big|\bar{\ell}_{N}^{K}-\frac{1}{1-\Lambda p}\Big|^{2}\Big] \le& 2\mathbb{E}\Big[\boldsymbol{1}_{\Omega_{N,K}}\Big|\bar{\ell}_{N}^{K}-1-\Lambda p\bar{\ell}_{N}\Big|^{2}\Big]+2 \mathbb{E}\Big[\boldsymbol{1}_{\Omega_{N,K}}\Big|\Lambda p\bar{\ell}_{N}-\frac{\Lambda p}{1-\Lambda p}\Big|^{2}\Big] \\
  \le& 2\mathbb{E}\Big[\boldsymbol{1}_{\Omega_{N,K}}\Big|\bar{\ell}_{N}^{K}-1-\Lambda p\bar{\ell}_{N}\Big|^{2}\Big]+ 2(\Lambda p)^{2}\mathbb{E}\Big[\boldsymbol{1}_{\Omega_{N}^{1}}\Big|\bar{\ell}_{N}-\frac{1}{1-\Lambda p}\Big|^{2}\Big]. 
\end{align*} 
We complete the proof applying Lemmas \ref{lll} and \ref{lA}.
\end{proof}

\subsection{Matrix analysis for the second estimator}
The aim of this subsection is to prove that
$\frac{N}{K}\|\boldsymbol{x}_{N}^{K}\|_{2}^{2}\simeq\Lambda^{2}p(1-p)/(1-\Lambda p)^{2}$
and to study the rate of convergence.

\begin{lemma}\label{IAX}
Assume that $p\in(0,1]$. It holds that
$$ 
\mathbb{E}[\|I_{K}A_{N}\boldsymbol{X}_{N}\|_{2}^{2}]\le CKN^{-2}.
$$
\end{lemma}
\begin{proof}
By Lemma \ref{xxx},  we already know that
$\mathbb{E}[\|A_{N}\boldsymbol{X}_{N}\|_{2}^{2}]\le \frac{C}{N}$, whence
\begin{align*}
\mathbb{E}\Big[\|I_{K}A_{N}\boldsymbol{X}_{N}\|_{2}^{2}\Big]=\sum_{i=1}^{K}\mathbb{E}\Big[\Big(\sum_{j=1}^{N}\frac{\theta_{ij}}{N}(L_{N}(j)-\bar{L}_{N})\Big)^{2}\Big]
  =\frac{K}{N^{2}}\mathbb{E}\Big[\Big(\sum_{j=1}^{N}\theta_{1j}(L_{N}(j)-\bar{L}_N)\Big)^{2}\Big],
\end{align*}
which equals $\frac{K}{N}\mathbb{E}[\|A_{N}\boldsymbol{X}_{N}\|_{2}^{2}]$ and thus is bounded by
$CKN^{-2}$.
\end{proof}

\begin{lemma}\label{xlX}
Assume that $\Lambda p <1$. It holds that
$$\mathbb{E}[\boldsymbol{1}_{\Omega_{N,K}\cap \mathcal{A}_{N}}\|\boldsymbol{x}_{N}^{K}-\bar{\ell}_{N}\Lambda \boldsymbol{X}_{N}^{K}\|^{2}_{2}]\le CN^{-1}.$$
\end{lemma}
\begin{proof}
By definition, $\boldsymbol{\ell}_{N}^{K}=I_{K}\boldsymbol{\ell}_{N}=\boldsymbol{1}_{K}+\Lambda I_{K}A_{N}\boldsymbol{\ell}_{N},$ so that
$$\bar{\ell}_{N}^{K}=\frac{1}{K}(\boldsymbol{1}_{K},\boldsymbol{\ell}_{N}^{K})= \frac{1}{K}(\boldsymbol{1}_{K},I_{K}\boldsymbol{\ell}_{N})=\frac{1}{K}(\boldsymbol{1}_{K},\boldsymbol{1}_{K}+\Lambda I_{K}A_{N}\boldsymbol{\ell}_{N})=1+\frac{\Lambda}{K}(I_{K}A_{N}\boldsymbol{\ell}_{N},\boldsymbol{1}_{K}).$$
And, recalling that $\boldsymbol{x}_{N}^{K}=\boldsymbol{\ell}_{N}^{K}-\bar{\ell}_{N}^{K}\boldsymbol{1}_{K}$, we find
\begin{align*}
\boldsymbol{x}_{N}^{K}&=\boldsymbol{1}_{K}+\Lambda I_{K}A_{N}\boldsymbol{\ell}_{N}-[1+\frac{\Lambda}{K}(I_{K}A_{N}\boldsymbol{\ell}_{N},\boldsymbol{1}_{K})]\boldsymbol{1}_{K}\\
&= \Lambda I_{K}A_{N}\boldsymbol{\ell}_{N}-\frac{\Lambda}{K}(I_{K}A_{N}\boldsymbol{\ell}_{N},\boldsymbol{1}_{K})\boldsymbol{1}_{K}\\&= \Lambda I_{K}A_{N}(\boldsymbol{\ell}_{N}-\bar{\ell}_{N}\boldsymbol{1}_{N})-\frac{\Lambda}{K}(I_{K}A_{N}(\boldsymbol{\ell}_{N}-\bar{\ell}_{N}\boldsymbol{1}_{N}),\boldsymbol{1}_{K})\boldsymbol{1}_{K}\\
&\qquad +\bar{\ell}_{N}[\Lambda I_{K}A_{N}\boldsymbol{1}_{N}-\frac{\Lambda}{K}(I_{K}A_{N}\boldsymbol{1}_{N},\boldsymbol{1}_{K})\boldsymbol{1}_{K}]
\\&=\Lambda I_{K}A_{N}\boldsymbol{x}_{N}-\frac{\Lambda}{K}(I_{K}A_{N}\boldsymbol{x}_{N},\boldsymbol{1}_{K})\boldsymbol{1}_{K}+\bar{\ell}_{N}[\Lambda I_{K}A_{N}\boldsymbol{1}_{N}-\frac{\Lambda}{K}(I_{K}A_{N}\boldsymbol{1}_{N},\boldsymbol{1}_{K})\boldsymbol{1}_{K}]
\\&= \Lambda I_{K}A_{N}\boldsymbol{x}_{N}-\frac{\Lambda}{K}(I_{K}A_{N}\boldsymbol{x}_{N},\boldsymbol{1}_{K})\boldsymbol{1}_{K}+\Lambda\bar{\ell}_{N} \boldsymbol{X}_{N}^{K}.
\end{align*}
We deduce that
\begin{align*}
\boldsymbol{x}_{N}^{K}-\Lambda\bar{\ell}_{N}\boldsymbol{X}_{N}^{K}&=\Lambda I_{K}A_{N}\boldsymbol{x}_{N}-\frac{\Lambda}{K}(I_{K}A_{N}\boldsymbol{x}_{N}, \boldsymbol{1}_{K})\boldsymbol{1}_{K} \\
&=\Lambda I_{K}A_{N}(\boldsymbol{x}_{N}-\Lambda\bar{\ell}_{N}\boldsymbol{X}_{N})-\frac{\Lambda}{K}(I_{K}A_{N}\boldsymbol{x}_{N}, \boldsymbol{1}_{K})\boldsymbol{1}_{K}+\bar{\ell}_{N}\Lambda^{2} I_{K}A_{N}\boldsymbol{X}_{N}\\
&=\Lambda I_{K}A_{N}(\boldsymbol{x}_{N}-\Lambda\bar{\ell}_{N}\boldsymbol{X}_{N})+\bar{\ell}_{N}\Lambda^{2}I_{K}A_{N}\boldsymbol{X}_{N}-\frac{\Lambda}{K}\Big[\sum_{i=1}^{K}\sum_{j=1}^{N}A_{N}(i,j)x_{N}(j)\Big]\boldsymbol{1}_{K}\\
&=\Lambda I_{K}A_{N}(\boldsymbol{x}_{N}-\Lambda\bar{\ell}_{N}\boldsymbol{X}_{N})+\bar{\ell}_{N}\Lambda^{2} I_{K}A_{N}\boldsymbol{X}_{N}-\frac{\Lambda}{K}\sum_{j=1}^{N}\Big[C_{N}^{K}(j)-\frac{K}{N}p\Big]x_{N}(j)\boldsymbol{1}_{K}. 
\end{align*}
In the last  step, we used that $\sum_{i=1}^{N}x(i)=0$.
As a conclusion,
\begin{align*}
\|\boldsymbol{x}_{N}^{K}-\bar{\ell}_{N}\Lambda \boldsymbol{X}_{N}^{K}\|^{2}_{2} \le & 3(\Lambda\|I_{K}A_{N}(\boldsymbol{x}_{N}-\bar{\ell}_{N}\boldsymbol{X}_{N})\|_{2})^{2}+3(\Lambda^{2}\bar{\ell}_{N}\|I_{K}A_{N}\boldsymbol{X}_{N}\|_{2})^{2}\\
  &+    3\Lambda^{2} K^{-1}\Big(\sum_{j=1}^{N}\Big[C_{N}^{K}(j)-\frac{K}{N}p\Big]x_{N}(j)\Big)^{2}. 
\end{align*}
By the Cauchy-Schwarz inequality, \eqref{sde} and Lemma \ref{xxx}, we have
\begin{align*}
&\mathbb{E}\Big[\boldsymbol{1}_{\Omega_{N,K}\cap \mathcal{A}_{N}}\Big(\sum_{j=1}^{N}\Big[C_{N}^{K}(j)-\frac{K}{N}p\Big]x_{N}(j)\Big)^{2}\Big]\\
\le& \mathbb{E}\Big[\boldsymbol{1}_{\Omega_{N,K}\cap \mathcal{A}_{N}}\Big(\sum_{j=1}^{N}\Big[C_{N}^{K}(j)-\frac{K}{N}p\Big]^{2}\Big)\Big]\mathbb{E}\Big[\boldsymbol{1}_{\Omega_{N,K}\cap \mathcal{A}_{N}}\Big(\sum_{j=1}^{N}x_{N}^{2}(j)\Big)\Big]\le \frac{CK}{N}.
\end{align*}
We also know from \cite[Proposition 14,  step 7, line 12]{A} that
$\mathbb{E}[\boldsymbol{1}_{\Omega_{N,K}\cap \mathcal{A}_{N}}\|\boldsymbol{x}_{N}-\bar{\ell}_{N}\boldsymbol{X}_{N}\|_{2}^{2}]\le \frac{C}{N}$. And also, by the definition, $\|A_{N}\|_{2}$ is bounded on
$\mathcal{A_{N}}$. 
So $$\mathbb{E}\Big[|||I_{K}A_{N}|||_{2}^{2}\|\boldsymbol{x}_{N}-\bar{\ell}_{N}\boldsymbol{X}_{N}\|_{2}^{2}\Big]\le \frac{C}{N}.$$ 
Recalling Lemma \ref{IAX} and that $\bar \ell_N$ is bounded on $\Omega_{N,K}$, the conclusion follows.
\end{proof}

\begin{lemma}\label{xLX}
Assume that $\Lambda p <1$. It holds that
$$\mathbb{E}\Big[\boldsymbol{1}_{\Omega_{N,K}\cap \mathcal{A}_{N}}\Big|\|\boldsymbol{x}_{N}^{K}\|_{2}^{2}-(\Lambda \bar{\ell}_{N})^{2}\|\boldsymbol{X}_{N}^{K}\|_{2}^{2}\Big|\Big]\le \frac{C\sqrt{K}}{N}.$$
\end{lemma}
\begin{proof}
We start from
\begin{eqnarray*}
\Big|\|\boldsymbol{x}_{N}^{K}\|_{2}^{2}-(\Lambda \bar{\ell}_{N})^{2}\|\boldsymbol{X}_{N}^{K}\|_{2}^{2}\Big|
\le \|\boldsymbol{x}_{N}^{K}-(\Lambda \bar{\ell}_{N})\boldsymbol{X}_{N}^{K}\|_{2}(\|\boldsymbol{x}_{N}^{K}\|_{2}+(\Lambda \bar{\ell}_{N})\|\boldsymbol{X}_{N}^{K}\|_{2}),
\end{eqnarray*}
whence
\begin{align*}
&\mathbb{E}\Big[\boldsymbol{1}_{\Omega_{N,K}\cap \cA_{N}}\Big|\|\boldsymbol{x}_{N}^{K}\|_{2}^{2}-(\Lambda \bar{\ell}_{N})^{2}\|\boldsymbol{X}_{N}^{K}\|_{2}^{2}\Big|\Big] \\
&\le
\mathbb{E}\Big[\boldsymbol{1}_{\Omega_{N,K}\cap \cA_{N}}\|\boldsymbol{x}_{N}^{K}-(\Lambda \bar{\ell}_{N})\boldsymbol{X}_{N}^{K}\|_{2}(\|\boldsymbol{x}_{N}^{K}\|_{2}+(\Lambda \bar{\ell}_{N})\|\boldsymbol{X}_{N}^{K}\|_{2})\Big].
\end{align*}
By the Cauchy-Schwarz inequality.
\begin{align*}
&\mathbb{E}\Big[\boldsymbol{1}_{\Omega_{N,K}\cap \mathcal{A}_{N}}\|\boldsymbol{x}_{N}^{K}-(\Lambda \bar{\ell}_{N})\boldsymbol{X}_{N}^{K}\|_{2}\Big(\|\boldsymbol{x}_{N}^{K}\|_{2}+(\Lambda \bar{\ell}_{N})\|\boldsymbol{X}_{N}^{K}\|_{2}\Big)\Big]\\
\le&\mathbb{E}\Big[\boldsymbol{1}_{\Omega_{N,K}\cap \mathcal{A}_{N}}\|\boldsymbol{x}_{N}^{K}-\bar{\ell}_{N}\Lambda \boldsymbol{X}_{N}^{K}\|^{2}_{2}\Big]^{\frac{1}{2}}\mathbb{E}\Big[\boldsymbol{1}_{\Omega_{N,K}\cap \cA_{N}}\Big(\|\boldsymbol{x}_{N}^{K}\|_{2}+(\Lambda \bar{\ell}_{N})\|\boldsymbol{X}_{N}^{K}\|_{2}\Big)^{2}\Big]^{\frac{1}{2}}.
\end{align*}
Lemma \ref{xlX} directly tells us that $\mathbb{E}[\boldsymbol{1}_{\Omega_{N,K}\cap \cA_{N}}\|\boldsymbol{x}_{N}^{K}-\bar{\ell}_{N}\Lambda \boldsymbol{X}_{N}^{K}\|^{2}_{2}]\leq C/N$.

\vip

Next, it is easy to prove, using that $\|\boldsymbol{X}_{N}^{K}\|_{2}^{2}=\sum_{i=1}^K(L_N(i)-\bar L_N^K)$,
that $NL_N(1),\dots,NL_N(K)$ are i.i.d. and Binomial$(N,p)$, that
\begin{equation}\label{ZE}
  \mathbb{E}\Big[\Big(\frac{N}{K} \|\boldsymbol{X}_{N}^{K}\|_{2}^{2}-p(1-p)\Big)^{2}\Big]\le CK^{-1},
\end{equation}
whence, recalling that $\bar\ell_N$ is bounded on $\Omega_{N,K}$,
$$\mathbb{E}\Big[\boldsymbol{1}_{\Omega_{N,K}\cap \mathcal{A}_{N}}(\bar\ell_N)^2
  \frac{N}{K}\|\boldsymbol{X}_{N}^{K}\|^{2}_{2}\Big] \le C\mathbb{E}\Big[\Big(\frac{N}{K}\|\boldsymbol{X}_{N}^{K}\|_{2}^{2}-p(1-p)\Big)^{2}\Big]^{\frac{1}{2}}+C\le C.$$
Then, by Lemma \ref{xlX} again,
\begin{align*}
\mathbb{E}\Big[\boldsymbol{1}_{\Omega_{N,K}\cap \cA_{N}}\|\boldsymbol{x}_{N}^{K}\|^{2}_{2}\Big]\le 2\mathbb{E}\Big[\boldsymbol{1}_{\Omega_{N,K}\cap \cA_{N}}\|\boldsymbol{x}_{N}^{K}-\bar{\ell}_{N}\Lambda \boldsymbol{X}_{N}^{K}\|^{2}_{2}\Big]+2\mathbb{E}\Big[\boldsymbol{1}_{\Omega_{N,K}\cap \cA_{N}}\|\bar{\ell}_{N}\Lambda \boldsymbol{X}_{N}^{K}\|^{2}_{2}\Big] \le  C\frac{K}{N}.
\end{align*}
The conclusion follows.
\end{proof}

\begin{lemma}\label{lXp}
Assume that $\Lambda p <1$. It holds that
$$\mathbb{E}\Big[\boldsymbol{1}_{\Omega_{N,K}\cap \cA_{N}}\Big|\frac{N}{K}(\bar{\ell}_{N})^{2}\|\boldsymbol{X}_{N}^{K}\|_{2}^{2}-\frac{p(1-p)}{(1-\Lambda p)^{2}}\Big|\Big]\le\frac{C}{\sqrt{K}}.$$
\end{lemma}

\begin{proof}
We define 
$$
d_{N}^{K}=\mathbb{E}\Big[\boldsymbol{1}_{\Omega_{N,K}\cap \cA_{N}}\Big|\frac{N}{K}(\bar{\ell}_{N})^{2}\|\boldsymbol{X}_{N}^{K}\|_{2}^{2}-\frac{p(1-p)}{(1-\Lambda p)^{2}}\Big|\Big].
$$ 
Then $d_{N}^{K}\le a_{N}^{K}+b_{N}^{K}$, where 
\begin{align*}
&a_{N}^{K}=\frac{N}{K}\mathbb{E}\Big[\boldsymbol{1}_{\Omega_{N,K}\cap \cA_{N}}\Big|(\bar{\ell}_{N})^{2}-(1-\Lambda p)^{-2}\Big|\|\boldsymbol{X}_{N}^{K}\|_{2}^{2}\Big],\\ 
&b_{N}^{K}=(1-\Lambda p)^{-2}\mathbb{E}\Big[\boldsymbol{1}_{\Omega_{N,K}\cap \cA_{N}}\Big|\frac{N}{K}\|\boldsymbol{X}_{N}^{K}\|_{2}^{2}-p(1-p)\Big|\Big].
\end{align*}
First, \eqref{ZE} directly implies that $b_N^K\leq C/\sqrt{K}$.
Next, \eqref{ZE} also implies that  $\mathbb{E}[(\frac{N}{K})^{2}\|\boldsymbol{X}_{N}^{K}\|_{2}^{4}]\le C,$
whence $a_N^K \leq C/\sqrt{K}$ by Lemma \ref{lll}. This completes the proof.
\end{proof}

Here is the main lemma of this subsection.

\begin{lemma}\label{mainM2}
Assume that $\Lambda p <1$. It holds that
$$
\mathbb{E}\Big[\boldsymbol{1}_{\Omega_{N,K}\cap \cA_{N}}\Big|\frac{N}{K}\|\boldsymbol{x}_{N}^{K}\|_{2}^{2}-
  \frac{\Lambda^{2}p(1-p)}{(1-\Lambda p)^{2}}\Big|\Big]\le\frac{C}{\sqrt{K}}.
$$
\end{lemma}

\begin{proof}
It directly follows from Lemmas \ref{xLX} and \ref{lXp} that
\begin{align*}
  &\mathbb{E}\Big[\boldsymbol{1}_{\Omega_{N,K}\cap \cA_{N}}\Big|\frac{N}{K}\|\boldsymbol{x}_{N}^{K}\|_{2}^{2}-
    \frac{\Lambda^{2}p(1-p)}{(1-\Lambda p)^{2}}\Big|\Big]\\
\le& \Lambda^{2}\mathbb{E}\Big[\boldsymbol{1}_{\Omega_{N,K}\cap \cA_{N}}\Big|\frac{N}{K}(\bar{\ell}_{N})^{2}\|\boldsymbol{X}_{N}^{K}\|_{2}^{2}-\frac{p(1-p)}{(1-\Lambda p)^{2}}\Big|\Big]\\
& +\frac{N}{K}\mathbb{E}\Big[\boldsymbol{1}_{\Omega_{N,K}\cap \cA_{N}}\Big|\|\boldsymbol{x}_{N}^{K}\|_{2}^{2}-(\Lambda \bar{\ell}_{N})^{2}\|\boldsymbol{X}_{N}^{K}\|_{2}^{2}\Big|\Big] \\
\le& C\Big(\frac{1}{\sqrt{K}}+\frac{N}{K}\frac{\sqrt{K}}{N}\Big),
\end{align*}
from which the conclusion.
\end{proof}

\subsection{Matrix analysis for the third estimator}
We define $\mathcal{W}^{N,K}_{\infty,\infty}:=\frac{\mu N}{K^{2}}\sum_{j=1}^{N}(c_{N}^{K}(j))^{2}\ell_{N}(j)-\frac{N-K}{K}\bar{\ell}_N^K,$ $\mathcal{X}^{N,K}_{\infty,\infty}:=\mathcal{W}^{N,K}_{\infty,\infty}-\frac{\mu(N-K)}{K}\bar{\ell}_N^K.$
The aim of this subsection is to prove that
$\mathcal{X}^{N,K}_{\infty,\infty}\simeq \mu/(1-\Lambda p)^{3}$
and to study the rate of convergence.

\begin{lemma}\label{FNK}
Assume that $\Lambda p <1$. It holds that
$$
\mathbb{E}[||\boldsymbol{F}_{N}^{K}||^{2}_{2}]\le\frac{CK}{N},
$$
where $\boldsymbol{F}_{N}^{K}:=\boldsymbol{1}_{K}^T A_N 
-\frac{1}{N}(\boldsymbol{1}_{K}^TA_{N},\boldsymbol{1}_{N}^T)\boldsymbol{1}_{N}^T$ is a row vector.
\end{lemma}

\begin{proof}
Since the inequality $\sum_{i=1}^n(x_i-\bar{x})^2\le \sum_{i=1}^n(x_i-m)^2$, where $\bar{x}=\frac{1}{N}\sum_{i=1}^n x_i,$ is correct for any real sequence $\{x_i\}_{i=1,...,n}$ and real number $p.$
By definition,
\begin{align*}
\mathbb{E}[||\boldsymbol{F}_{N}^{K}||^{2}_{2}]&=\mathbb{E}\Big[\sum_{j=1}^{N}\Big\{\frac{1}{N}\sum_{i=1}^{K}\theta_{ij}-\frac{1}{N^{2}}\sum_{i=1}^{K}\sum_{l=1}^{N}\theta_{il}\Big\}^{2}\Big]\\
&\le \mathbb{E}\Big[\sum_{j=1}^{N}\Big\{\frac{1}{N}\sum_{i=1}^{K}\theta_{ij}-\frac{Kp}{N}\Big\}^{2}\Big]\le \frac{1}{N}\mathbb{E}\Big[\Big\{\sum_{i=1}^{K}(\theta_{i1}-p)\Big\}^{2}\Big]\le \frac{CK}{N}.
\end{align*}
\end{proof}

\begin{lemma}\label{cc1}
Assume that $\Lambda p <1$. It holds that
$$
\mathbb{E}\Big[\boldsymbol{1}_{\Omega_{N,K}}\|\boldsymbol{t}_{N}^{K}\|^{2}_{2}\Big]\le C\frac{K^2}{N^2}.
$$
where $\boldsymbol{c}_N^K= \indiq_K^T Q_N$,
$\bar{c}_{N}^{K}:=\frac{1}{N}\sum_{j=1}^Nc_{N}^{K}(j),$ and
$$
\boldsymbol{t}_{N}^{K}:=\boldsymbol{c}_{N}^{K}-\bar{c}_{N}^{K}\boldsymbol{1}^{T}_{N}-\boldsymbol{1}^{T}_{K}+\frac{K}{N}\boldsymbol{1}^{T}_{N}.$$
\end{lemma}

\begin{proof}
By definition,  $\boldsymbol{c}_{N}^{K}:=\boldsymbol{1}^{T}_{K}Q_{N},$
$\bar{c}_{N}^{K}=\frac{1}{N}(\boldsymbol{c}_{N}^{K},\boldsymbol{1}_{N}),$ $Q_{N}=(I-\Lambda A_{N})^{-1}$, so that
$$\boldsymbol{c}_{N}^{K}=\boldsymbol{1}^{T}_{K}+\Lambda \boldsymbol{c}_{N}^{K}A_{N}, \qquad 
\bar{c}_{N}^{K}=\frac{1}{N}(\boldsymbol{c}_{N}^{K},\boldsymbol{1}_{N}^T)
=\frac{K}{N}+\frac{\Lambda}{N}(\boldsymbol{c}_{N}^{K}A_{N}, \boldsymbol{1}_{N}^T).$$
We deduce that
\begin{align}
\boldsymbol{t}_{N}^{K}&=\boldsymbol{c}_{N}^{K}-\bar{c}_{N}^{K}\boldsymbol{1}^{T}_{N}-\boldsymbol{1}^{T}_{K}+\frac{K}{N}\boldsymbol{1}^{T}_{N}\notag\\
&=\Lambda \boldsymbol{c}_{N}^{K}A_{N}-\frac{\Lambda}{N}(\boldsymbol{c}_{N}^{K}A_{N}, \boldsymbol{1}_{N}^T)\boldsymbol{1}^{T}_{N}\notag\\
 &=\Lambda \boldsymbol{t}_{N}^{K}A_{N}-\frac{\Lambda}{N}(\boldsymbol{t}_{N}^{K}A_{N},\boldsymbol{1}_{N}^T)\boldsymbol{1}^{T}_{N}+\Lambda \bar{c}_{N}^{K}\boldsymbol{1}^{T}_{N}A_{N}-\frac{\Lambda}{N}\bar{c}_{N}^{K}(\boldsymbol{1}^{T}_{N}A_N,1_{N}^T)\boldsymbol{1}^{T}_{N} \notag\\
 &\qquad +\Lambda \boldsymbol{1}_{K}^{T}A_{N}-\frac{\Lambda}{N}(\boldsymbol{1}^{T}_{K}A_{N},\boldsymbol{1}_{N}^T)
 \boldsymbol{1}^{T}_{N}-\Lambda \frac{K}{N}\boldsymbol{1}_{N}^{T}A_{N}+\frac{\Lambda}{N}\frac{K}{N}(\boldsymbol{1}_{N}^{T}A_{N},\boldsymbol{1}_{N}^T) \notag\\
 &=\Lambda \boldsymbol{t}_{N}^{K}A_{N}-\frac{\Lambda}{N}(\boldsymbol{t}_{N}^{K}A_{N},\boldsymbol{1}_{N}^T)\boldsymbol{1}^{T}_{N}+\Lambda\bar{c}_{N}^{K}\boldsymbol{X}_{N}^{T}-\Lambda \boldsymbol{F}^{K}_{N}-\Lambda \frac{K}{N}\boldsymbol{X}_{N}^{T}.\label{bbg}
\end{align}
where $\boldsymbol{X}_{N}^{T}=\boldsymbol{1}_{N}^{T}A_{N}-\frac{1}{N}(\boldsymbol{1}_{N}^{T}A_{N},\boldsymbol{1}_{N}^T).$
And it is clear that
$\frac{N}{K}\bar{c}_{N}^{K}=\bar{\ell}_{N}^{K}.$

\vip
 
By Lemma \ref{lo}, $\bar{\ell}_{N}$ and $\bar{\ell}_{N}^{K}$ are bounded  on the set $\Omega_{N,K}$, whence,
using Lemma \ref{xxx},
$$
\mathbb{E}\Big[\boldsymbol{1}_{\Omega_{N,K}}\Big(||\Lambda\bar{c}_{N}^{K}\boldsymbol{X}_{N}^{T}||_{2}^2+
||\Lambda\frac K N \boldsymbol{X}_{N}^{T}||_2^2 \Big) \Big ] \leq C \frac {K^2}{ N^2}.
$$
Next, Lemma \ref{FNK} tells us that
$$
\mathbb{E}\Big[\boldsymbol{1}_{\Omega_{N,K}}\|\boldsymbol{F}_{N}^{K}\|^2_{2}\Big]
\le\mathbb{E}\Big[\|\boldsymbol{F}_{N}^{K}\|^2_{2}\Big]\le\frac{CK^2}{N^2}.
$$
Observing that $\sum_{i=1}^{N}t_{N}^{K}(i)=0,$ we see that
\begin{align*}
||\frac{\Lambda}{N}(\boldsymbol{t}_{N}^{K}A_{N},\boldsymbol{1}_{N}^T)\boldsymbol{1}_{N}^T||_2^2=&
\frac{\Lambda^2}N(\boldsymbol{t}_{N}^{K}A_{N},\boldsymbol{1}_{N}^T)^2\\
=&\frac{\Lambda^2}{N^{3}}\Big(\sum^{N}_{i=1}\sum_{j=1}^{N}(\theta_{ij}-p)t_{N}^{K}(i)\Big)^2\\
=&\frac{\Lambda^2}{N}\Big(\sum^{N}_{i=1}(C_N(i)-p)t_{N}^{K}(i)\Big)^2,
\end{align*}
so that
\begin{align*}
\mathbb{E}\Big[\boldsymbol{1}_{\Omega_{N,K}\cap \cA_N}\|\frac{\Lambda}{N}(\boldsymbol{t}_{N}^{K}A_{N},\boldsymbol{1}_{N}^T)\boldsymbol{1}_{N}^T\|^2_{2}\Big]\le&\frac{\Lambda^2}{N}\mathbb{E}\Big[\boldsymbol{1}_{\Omega_{N,K}\cap \cA_N}
    \|\boldsymbol{t}_{N}^{K}\|^2_{2}\|C_{N}-p\boldsymbol{1}_{N}\|^2_{2}\Big]\\
\leq& \frac{\Lambda^2}{N^{1/2}} \mathbb{E}\Big[\boldsymbol{1}_{\Omega_{N,K}\cap \cA_N}
    \|\boldsymbol{t}_{N}^{K}\|^2_{2}\Big]
\end{align*}
by definition of $\cA_N$. Since finally $||\Lambda \ct_N^K A_N||_2 \leq |||\Lambda A_N |||_2 ||\Lambda \ct_N^K ||_2
\leq a ||\Lambda \ct_N^K ||_2$ on $\Omega_{N,K}$ with $a=(1+\Lambda p)/2$, we conclude that
$$
\mathbb{E}\Big[\boldsymbol{1}_{\Omega_{N,K}\cap \cA_N}\|\boldsymbol{t}_{N}^{K}\|^2_{2}\Big]
\leq C \frac{K^2}{N^2} + (a+ \Lambda^2 N^{-1/2})
\mathbb{E}\Big[\boldsymbol{1}_{\Omega_{N,K}\cap \cA_N}\|\boldsymbol{t}_{N}^{K}\|^2_{2}\Big].
$$
Since $(a+ \Lambda^2 N^{-1/2})<(a+1)/2<1$ for all $N$ large enough, we conclude that, for some constant $C>0$,
for all $N\geq 1$,
$$
\mathbb{E}\Big[\boldsymbol{1}_{\Omega_{N,K}\cap \cA_N}\|\boldsymbol{t}_{N}^{K}\|^2_{2}\Big]
\leq C \frac{K^2}{N^2}.
$$
Finally, observing that $\|\boldsymbol{t}_{N}^{K}\|^2_{2}$ is obviously bounded by $C N$ on
$\Omega_{N,K}$ and recalling that $\mathbb{P}(\cA_N) \geq 1- C/N^3$ by Lemma \ref{lo}, we easily conclude that
$$
\mathbb{E}\Big[\boldsymbol{1}_{\Omega_{N,K}}\|\boldsymbol{t}_{N}^{K}\|^2_{2}\Big]
\leq C \frac{K^2+1}{N^2} \leq C \frac{K^2}{N^2}
$$
as desired.
\end{proof}

\begin{lemma}\label{f}
Assume that $\Lambda p <1$. It holds that
$$
\mathbb{E}\Big[\boldsymbol{1}_{\Omega_{N,K}}||\boldsymbol{f}_{N}^{K}||^{2}_{2}\Big]\le C\frac{K^{2}}{N^{2}}, \quad \mathbb{E}\Big[\boldsymbol{1}_{\Omega_{N,K}}\Big|(\boldsymbol{f}_{N}^{K},\boldsymbol{1}_{K}^T)\Big|\Big]\le \frac{CK}{N}.
$$
where $\boldsymbol{f}_{N}^{K}:=\boldsymbol{t}_{N}^{K}I_{K}$.
\end{lemma}
\begin{proof}
The first inequality is obvious from Lemma \ref{cc1} because $||\boldsymbol{f}_{N}^{K}||
\leq||\boldsymbol{t}_{N}^{K}||$.
For the second inequality, by \eqref{bbg}, we have
\begin{align*}
    &\mathbb{E}\Big[\boldsymbol{1}_{\Omega_{N,K}}\Big|(\boldsymbol{f}_{N}^{K},\boldsymbol{1}_{K}^T)\Big|\Big]=\mathbb{E}\Big[\boldsymbol{1}_{\Omega_{N,K}}\Big|(\ct_{N}^{K}I_{K},\boldsymbol{1}_{K}^T)\Big|\Big]\\
    &=\mathbb{E}\Big[\boldsymbol{1}_{\Omega_{N,K}}\Big|\Big(\Lambda\boldsymbol{t}_{N}^{K}A_{N}I_{K}-\frac{\Lambda}{N}(\boldsymbol{t}_{N}^{K}A_{N},\boldsymbol{1}_{N}^T)\boldsymbol{1}^{T}_{K}+\Lambda\bar{\boldsymbol{c}}_{N}^{K}\boldsymbol{X}_{N}^{T}I_{K}-\Lambda \boldsymbol{F}^{K}_{N}I_{K}-\Lambda \frac{K}{N}\boldsymbol{X}_{N}^{T}I_{K},\boldsymbol{1}_{K}^T\Big)\Big|\Big]\\
   & \le \frac{CK}{N}\mathbb{E}\Big[\Big|(\cX_{N}^{T}I_{K},\boldsymbol{1}_{K}^T)\Big|\Big]+C\mathbb{E}\Big[\Big|(\boldsymbol{F}_{N}^{K}I_{K},\boldsymbol{1}_{K}^T)\Big|\Big]\\
    &\qquad+\frac{CK}{N}\mathbb{E}\Big[\boldsymbol{1}_{\Omega_{N,K}}\Big|(\ct_{N}^{K}A_{N},\boldsymbol{1}_{N}^T)\Big|\Big]+\Lambda\mathbb{E}\Big[\boldsymbol{1}_{\Omega_{N,K}}\Big|(\ct_{N}^{K}A_{N}I_{K},\boldsymbol{1}_{K}^T)\Big|\Big].
\end{align*}
We used that $(N/K)\bar c^K_N=\bar \ell^K_N$ is bounded on $\Omega_{N,K}$.
First,
\begin{align*}
\mathbb{E}\Big[\Big|(\cX_{N}^{T}I_{K},\boldsymbol{1}^T_{K})\Big|^{2}\Big]&=\mathbb{E}\Big[\Big|\sum_{i=1}^{K}X_{N}(i)\Big|^{2}\Big]=\mathbb{E}\Big[\Big|\sum_{i=1}^{K}\Big(L_{N}(i)-p\Big)+K(p-\bar{L}_{N})\Big|^{2}\Big]\\
    &\le 2\mathbb{E}\Big[\Big|\sum_{i=1}^{K}\Big(L_{N}(i)-p\Big)\Big|^{2}\Big]+2K^{2}\mathbb{E}\Big[(p-\bar{L}_{N})^{2}\Big]\le \frac{CK}{N} \leq C,
\end{align*}
using only that $NL_N(1),\dots,NL_N(N)$ are i.i.d. and Binomial$(N,p)$-distributed.
Next,
\begin{align*}
(\cF_{N}^{K}I_{K},\boldsymbol{1}_{K})&=\frac{1}{N}\Big(\sum_{j=1}^{K}\sum_{i=1}^{K}\theta_{ij}-\frac{K}{N}\sum_{i=1}^{K}\sum_{j=1}^{N}\theta_{ij}\Big)
\\&=\frac{1}{N}\Big[\frac{N-K}{N}\sum_{j=1}^{K}\sum_{i=1}^{K}(\theta_{ij}-p)-\frac{K}{N}\sum_{i=1}^{K}\sum_{j=k+1}^{N}(\theta_{ij}-p)\Big],
\end{align*}
so that
$$
\mathbb{E}[|(\cF_{N}^{K}I_K,\boldsymbol{1}^T_{K})|]\le \mathbb{E}\Big[\Big|(\cF_{N}^{K}I_K,\boldsymbol{1}^T_{K})\Big|^{2}\Big]^{\frac{1}{2}}
\leq \frac C N \Big[ \frac{N-K}N K + \frac KN \sqrt{K(N-K)}\Big]
\le \frac{CK}{N}.
$$
Next, since $\sum_{i=1}^N t_N^K(i)=0$, 
 \begin{align*}
     \mathbb{E}\Big[\indiq_{\Omega_{N,K}}\Big|(\ct_{N}^{K}A_{N},\boldsymbol{1}^T_{N})\Big|\Big]&=\mathbb{E}\Big[\indiq_{\Omega_{N,K}}\frac{1}{N}\Big|\sum_{i,j=1}^{N}\theta_{ij}t_{N}^{K}(i)\Big|\Big]\\
     &=\mathbb{E}\Big[\indiq_{\Omega_{N,K}}\frac{1}{N}\Big|\sum_{i,j=1}^{N}(\theta_{ij}-p)t_{N}^{K}(i)\Big|\Big]\\
     &\le \frac{C}{N}\mathbb{E}\Big[\indiq_{\Omega_{N,K}}\sum_{i=1}^{N}\Big(t_{N}^{K}(i)\Big)^{2}\Big]^{\frac{1}{2}}\mathbb{E}\Big[\sum_{i=1}^{N}\Big(\sum_{j=1}^{N}(\theta_{ij}-p)\Big)^{2}\Big]^{\frac{1}{2}}\le C
 \end{align*}
 by Lemma \ref{cc1}. Finally,
 \begin{align*}
 \mathbb{E}\Big[\indiq_{\Omega_{N,K}}\Big|&(\ct_{N}^{K}A_{N}I_{K},\boldsymbol{1}^T_{K})\Big|\Big]=\frac{1}{N}\mathbb{E}\Big[\indiq_{\Omega_{N,K}}\Big|\sum_{i,j=1}^{K}t_{N}^{K}(i)\theta_{ij}\Big|\Big]\\
 &\le \frac{1}{N}\mathbb{E}\Big[\indiq_{\Omega_{N,K}}\Big|\sum_{i,j=1}^{K}t_{N}^{K}(i)(\theta_{ij}-p)\Big|\Big]+\frac{Kp}{N}\mathbb{E}\Big[\indiq_{\Omega_{N,K}}\Big|\sum_{i=1}^{K}t_{N}^{K}(i)\Big|\Big]\\
 &\le \frac 1 N \mathbb{E}\Big[\indiq_{\Omega_{N,K}}\sum_{i=1}^{K}\Big(t_{N}^{K}(i)\Big)^{2}\Big]^{\frac{1}{2}}\mathbb{E}\Big[\sum_{i=1}^{K}\Big(\sum_{j=1}^{K}(\theta_{ij}-p)\Big)^{2}\Big]^{\frac{1}{2}} +\frac{Kp}{N}\mathbb{E}\Big[\indiq_{\Omega_{N,K}}\Big|(\boldsymbol{f}_{N}^{K},\boldsymbol{1}^T_{K})\Big|\Big]\\
 &\le C \frac KN + \frac{Kp}{N}\mathbb{E}\Big[\indiq_{\Omega_{N,K}}\Big|(\boldsymbol{f}_{N}^{K},\boldsymbol{1}^T_{K})\Big|\Big]
 \end{align*}
 by Lemma \ref{cc1}. All this proves that
$$
 \mathbb{E}\Big[\indiq_{\Omega_{N,K}}\Big|(\boldsymbol{f}_{N}^{K},\boldsymbol{1}^T_{K})\Big|\Big]
 \leq C \frac KN + \frac{Kp}{N}\mathbb{E}\Big[\indiq_{\Omega_{N,K}}\Big|(\boldsymbol{f}_{N}^{K},\boldsymbol{1}^T_{K})\Big|\Big],
$$
whence the conclusion since $Kp/N\leq p <1$.
\end{proof}

\begin{lemma}\label{cll}
Assume that $\Lambda p <1$. It holds that
 \begin{align*}
\mathbb{E}\Big[\boldsymbol{1}_{\Omega_{N,K}}\Big|\frac{N}{K^{2}}\Big(\bar{c}_{N}^{K}+\frac{N-K}{N}\Big)^{2}\sum_{j=1}^{K}\ell_{N}(j)+\frac{N}{K^{2}}&\Big(\bar{c}_{N}^{K}-\frac{K}{N}\Big)^{2}\sum_{j=K+1}^{N}\ell_{N}(j)\\&-\frac{N-K}{K}\bar{\ell}_N^K-\frac{1}{(1-\Lambda p)^{3}}
\Big|^{2}\Big]\le \frac{C}{NK}.
  \end{align*}
 \end{lemma}
\begin{proof}
Recall that $\frac{N}{K}\bar{c}_{N}^{K}=\bar{\ell}_{N}^{K}$, whence

\begin{align*}
&\Big(\bar{c}_{N}^{K}+\frac{N-K}{N}\Big)^{2}\sum_{j=1}^{K}\ell_{N}(j)+\Big(\bar{c}_{N}^{K}-\frac{K}{N}\Big)^{2}\sum_{j=K+1}^{N}\ell_{N}(j)
\\=&\Big(\bar{c}_{N}^{K}-\frac{K}{N}\Big)^{2}\sum_{j=1}^{N}\ell_{N}(j)+2\Big(\bar{c}_{N}^{K}-\frac{K}{N}\Big)\sum_{i=1}^{K}\ell_{N}(i)+\sum_{i=1}^{K}\ell_{N}(i)
\\
=&\Big(\frac{K}{N}\Big)^{2}(\bar{\ell}_{N}^{K}-1)^{2}N\bar{\ell}_{N}+2\frac{K}{N}(\bar{\ell}_{N}^{K}-1)K\bar{\ell}_{N}^{K}+K\bar{\ell}_{N}^{K}
\\=&\frac{K^{2}}{N}(\bar{\ell}_{N}^{K}-1)^{2}\bar{\ell}_{N}+2\frac{K^{2}}{N}(\bar{\ell}_{N}^{K}-1)\bar{\ell}_{N}^{K}+K\bar{\ell}_{N}^{K}
\\=&2\frac{K^{2}}{N}\Big(-\bar{\ell}_{N}^{K}\bar{\ell}_{N}+(\bar{\ell}_{N}^{K})^{2}\Big)+\frac{K^{2}}{N}\Big(
\bar{\ell}_{N}-\bar{\ell}_{N}^{K}\Big)+\frac{K^{2}}{N}(\bar{\ell}_{N}^{K})^{2}\bar{\ell}_{N}-\frac{K^{2}}{N}\bar{\ell}_{N}^{K}+K\bar{\ell}_{N}^{K}.
\end{align*}
Consequently,
\begin{eqnarray*}
\frac{N}{K^{2}}\Big(\bar{c}_{N}^{K}+\frac{N-K}{N}\Big)^{2}\sum_{j=1}^{K}\ell_{N}(j)+\frac{N}{K^{2}}\Big(\bar{c}_{N}^{K}-\frac{K}{N}\Big)^{2}\sum_{j=K+1}^{N}\ell_{N}(j)\\
=2\Big(-\bar{\ell}_{N}^{K}\bar{\ell}_{N}+(\bar{\ell}_{N}^{K})^{2}\Big)+\Big(\bar{\ell}_{N}-\bar{\ell}_{N}^{K}\Big)+\Big(\bar{\ell}_{N}^{K}\Big)^{2}\bar{\ell}_{N}+\frac{N-K}{K}\bar{\ell}_{N}^{K}.
\end{eqnarray*}
On the event $\Omega_{N,K}$, we have $\bar{\ell}_{N}^{K},\ \bar{\ell}_{N}$ are bounded. Hence,
\begin{align*}
\mathbb{E}\Big[\boldsymbol{1}_{\Omega_{N,K}}\Big|-\bar{\ell}_{N}^{K}\bar{\ell}_{N}+(\bar{\ell}_{N}^{K})^{2}\Big|^{2}\Big]
 &=\mathbb{E}\Big[\boldsymbol{1}_{\Omega_{N,K}}\Big|\bar{\ell}_{N}^{K}\Big|^{2}\Big|\bar{\ell}_{N}^{K}-\bar{\ell}_{N}\Big|^{2}\Big]\\ 
 &\le
 C \mathbb{E}\Big[\boldsymbol{1}_{\Omega_{N,K}}\Big|\bar{\ell}_{N}^{K}-\frac{1}{1-\Lambda p}+\frac{1}{1-\Lambda p}-\bar{\ell}_{N}\Big|^{2}\Big]\\
 &\le C \mathbb{E}\Big[\boldsymbol{1}_{\Omega_{N,K}}\Big|\bar{\ell}_{N}^{K}-\frac{1}{1-\Lambda p}\Big|^{2}\Big]+C\mathbb{E}\Big[\boldsymbol{1}_{\Omega_{N,K}}\Big|\frac{1}{1-\Lambda p}-\bar{\ell}_{N}\Big|^{2}\Big]
\le \frac{C}{NK}
\end{align*}
by Lemmas \ref{lll} and \ref{ellp}. Similarly, 
\begin{align*}
&\mathbb{E}\Big[\boldsymbol{1}_{\Omega_{N,K}}\Big(\Big|(\bar{\ell}_{N}^{K})^{2}\bar{\ell}_{N}-\frac{1}{(1-\Lambda p)^{3}}\Big|^{2} + (\bar{\ell}_{N}-\bar{\ell}_{N}^K)^2   \Big)\Big]\\ 
\le& C\mathbb{E}\Big[\boldsymbol{1}_{\Omega_{N,K}}\Big|\bar{\ell}_{N}^{K}-\frac{1}{1-\Lambda p}\Big|^{2}\Big]+C\mathbb{E}\Big[\boldsymbol{1}_{\Omega_{N,K}}\Big|\frac{1}{1-\Lambda p}-\bar{\ell}_{N}\Big|^{2}\Big]
\le \frac{C}{NK}.
\end{align*}
The conclusion follows.
\end{proof}

Here is the main result of this subsection.

\begin{lemma}\label{WWW}
Assume that $\Lambda p <1$. We have that
$$
\mathbb{E}\Big[\boldsymbol{1}_{\Omega_{N,K}}\Big|\mathcal{X}^{N,K}_{\infty,\infty}-\frac{\mu}{(1-\Lambda p)^{3}}\Big|\Big]\le \frac{C}{K}.
$$
\end{lemma}
\begin{proof}
By definition,
\begin{align*}
  \mathcal{X}^{N,K}_{\infty,\infty}-\frac{\mu}{(1-\Lambda p)^{3}}&=\frac{\mu N}{K^{2}}\sum_{j=1}^{N}\Big(c_{N}^{K}(j)\Big)^{2}\ell_{N}(j)-\frac{\mu(N-K)}{K}\bar{\ell}_N^K-\frac{\mu}{(1-\Lambda p)^{3}}=\mu\sum_{\alpha=1}^3 I_{N,K}^\alpha,
\end{align*}
where
\begin{align*}
I_{N,K}^1=&\frac{N}{K^{2}}\sum_{j=1}^{K}\Big[c_{N}^{K}(j)-\bar{c}_{N}^{K}-\frac{N-K}{N}\Big]^{2}\ell_{N}(j)+\frac{N}{K^{2}}\sum_{j=K+1}^{N}\Big[c_{N}^{K}(j)-\bar{c}_{N}^{K}+\frac{K}{N}\Big]^{2}\ell_{N}(j),\\
I_{N,K}^2=&2\frac{N}{K^{2}}\Big[\bar{c}_{N}^{K}+\frac{N-K}{N}\Big]\sum_{j=1}^{K}\ell_{N}(j)\Big[c_{N}^{K}(j)-\bar{c}_{N}^{K}-\frac{N-K}{N}\Big]\\
&+2\frac{N}{K^{2}}\Big[\bar{c}_{N}^{K}-\frac{K}{N}\Big]\sum_{j=K+1}^{N}\ell_{N}(j)\Big[c_{N}^{K}(j)-\bar{c}_{N}^{K}+\frac{K}{N}\Big],
\\
I_{N,K}^3=&\frac{N}{K^{2}}\Big[\bar{c}_{N}^{K}+\frac{N-K}{N}\Big]^{2}\sum_{j=1}^{K}\ell_{N}(j)+\frac{N}{K^{2}}\Big[\bar{c}_{N}^{K}-\frac{K}{N}\Big]^{2}\sum_{j=K+1}^{N}\ell_{N}(j)-\frac{N-K}{K}\bar{\ell}_N^K -\frac{1}{(1-\Lambda p)^{3}}.
\end{align*}
By Lemma \ref{cc1}, $\bar{\ell}_{N}$ and $\ell_{N}(j)$ are bounded  on the set $\Omega_{N,K}$ for
any $j=1,...,N$, whence
\begin{align*}
\mathbb{E}\Big[\boldsymbol{1}_{\Omega_{N,K}}|I_{N,K}^1| \Big]
\le  C \frac{N}{K^{2}}\mathbb{E}\Big[\boldsymbol{1}_{\Omega_{N,K}}||\cc_{N}^{K}-\bar{c}_{N}^{K}\boldsymbol{1}^{T}_{N}-\boldsymbol{1}^{T}_{K}+\frac{K}{N}\boldsymbol{1}^{T}_{N}||^{2}_{2}\Big]
=C \frac{N}{K^{2}}\mathbb{E}\Big[\boldsymbol{1}_{\Omega_{N,K}}||\boldsymbol{t}_{N}^{K}||^{2}_{2}\Big]
\le \frac{C}{N}. 
\end{align*}
Recall the result from Lemma \ref{cll}: we have
\begin{align*}
\mathbb{E}\Big[\boldsymbol{1}_{\Omega_{N,K}}|I^3_{N,K}|\Big] \leq  \frac{C}{\sqrt{NK}} \leq \frac C K.
\end{align*}
Next, we have $I^2_{N,K}= 2I^{2,1}_{N,K}+2I^{2,2}_{N,K}$, where
\begin{align*}
I^{2,1}_{N,K}= & \frac N{K^2}
\sum_{j=1}^{K}\ell_{N}(j)\Big[c_{N}^{K}(j)-\bar{c}_{N}^{K}-\frac{N-K}{N}\Big],\\
I^{2,2}_{N,K}= &\frac N{K^2}\Big[\bar{c}_{N}^{K}-\frac{K}{N}\Big]\Big\{\sum_{j=1}^{K}\ell_{N}(j)\Big[c_{N}^{K}(j)-\bar{c}_{N}^{K}-\frac{N-K}{N}\Big]+\sum_{j=K+1}^{N}\ell_{N}(j)\Big[c_{N}^{K}(j)-\bar{c}_{N}^{K}+\frac{K}{N}\Big]\Big\}.
\end{align*}
Since
\begin{align*}                                        
&\sum_{j=1}^{K}\Big[c_{N}^{K}(j)-\bar{c}_{N}^{K}-\frac{N-K}{N}\Big]+\sum_{j=K+1}^{N}\Big[c_{N}^{K}(j)-\bar{c}_{N}^{K}+\frac{K}{N}\Big]=0,
\end{align*}
we may write
\begin{align*}
I_{N,K}^{2,2}=&
\frac N{K^2}\Big[\bar{c}_{N}^{K}-\frac{K}{N}\Big]\Big\{\sum_{j=1}^{K}\Big(\ell_{N}(j)-\bar{\ell}_{N}\Big)\Big[c_{N}^{K}(j)-\bar{c}_{N}^{K}-\frac{N-K}{N}\Big]\\
&\hskip4cm +\sum_{j=K+1}^{N}\Big(\ell_{N}(j)-\bar{\ell}_{N}\Big)\Big[c_{N}^{K}(j)-\bar{c}_{N}^{K}+\frac{K}{N}\Big]\Big\}
\\=&\frac N{K^2}\Big[\bar{c}_{N}^{K}-\frac{K}{N}\Big]\Big\{\sum_{j=1}^{K}x_{N}(j)\Big[c_{N}^{K}(j)-\bar{c}_{N}^{K}-\frac{N-K}{N}\Big]+\sum_{j=K+1}^{N}x_{N}(j)\Big[c_{N}^{K}(j)-\bar{c}_{N}^{K}+\frac{K}{N}\Big]\Big\}\\
=&\frac N{K^2}\Big[\bar{c}_{N}^{K}-\frac{K}{N}\Big] (\boldsymbol{x}_N,\boldsymbol{t}_N).
\end{align*}
Recalling that $\bar{c}_{N}^{K}=K\bar{\ell}_{N}^{K}/N$ and that $\bar{\ell}_{N}^{K}$ is bounded on
$\Omega_{N,K}$, we conclude that $\indiq_{\Omega_{N,K}}|I_{N,K}^{2,2}| \leq C ||\boldsymbol{x}_N||_2||\boldsymbol{t}_N||_2/K$.
Using Lemmas \ref{xxx} and \ref{cc1}, we readily conclude that
$\E[\indiq_{\Omega_{N,K}} |I_{N,K}^{2,2}|]\leq C/N$.
Finally,
\begin{align*}
\mathbb{E}[\boldsymbol{1}_{\Omega_{N,K}}|I_{N,K}^{2,1}|]
\le& \frac{N}{K^{2}}\mathbb{E}\Big[\boldsymbol{1}_{\Omega_{N,K}}\Big|\sum_{j=1}^{K}(\ell_{N}(j)-\bar{\ell}_{N}^{K})\Big[c_{N}^{K}(j)-\bar{c}_{N}^{K}-\frac{N-K}{N}\Big]\Big|\Big]\\
&+\frac{N}{K^{2}}\mathbb{E}\Big[\boldsymbol{1}_{\Omega_{N,K}}\Big|\sum_{j=1}^{K}\bar{\ell}_{N}^{K}\Big[c_{N}^{K}(j)-\bar{c}_{N}^{K}-\frac{N-K}{N}\Big]\Big|\Big]\\
\le& \frac{N}{K^{2}}\mathbb{E}\Big[\boldsymbol{1}_{\Omega_{N,K}}\|\cx_{N}\|^{2}_{2}\Big]^{\frac{1}{2}}\mathbb{E}\Big[\boldsymbol{1}_{\Omega_{N,K}}\|\cf_{N}^{K}\|^{2}_{2}\Big]^{\frac{1}{2}}+\frac{N}{K^{2}}\mathbb{E}\Big[\boldsymbol{1}_{\Omega_{N,K}}\Big|(\cf_{N}^{K},\boldsymbol{1}^T_{K})\Big|\Big]\le \frac{C}{K}
\end{align*}
by Lemma \ref{f}.
The proof is complete.
\end{proof}

\section{Some auxilliary processes}

We first introduce a family of martingales: for $i=1,\dots,N$, recalling \eqref{sssy},
$$
M_{t}^{i,N}=\int_{0}^{t}\int_{0}^{\infty}  \boldsymbol{1}_{\{z\le\lambda_{s}^{i,N}\}}\widetilde{\pi}^{i}(ds,dz).
$$
where $\widetilde{\pi}^{i}(ds,dz)=\pi^i(ds,dz)-dsdz$.
We also introduce the family of centered   processes $U_{t}^{i,N}=Z_{t}^{i,N}-\mathbb{E}_{\theta}[Z_{t}^{i,N}]$.

\vip

We  denote  by $\boldsymbol{Z}_{t}^{N}$ (resp. $\boldsymbol{U}_{t}^{N}$, $\boldsymbol{M}_{t}^{N}$) the $N$ dimensional
vector  with  coordinates $Z_{t}^{i,N}$ (resp. $U_{t}^{i,N}$, $M_{t}^{i,N}$) and
set 
$$
\boldsymbol{Z}_{t}^{N,K}=I_{K}\boldsymbol{Z}_{t}^{N},\quad \boldsymbol{U}_{t}^{N,K}=I_{K}\boldsymbol{U}_{t}^{N},
$$
as well as $\bar{Z}^{N,K}_{t}=K^{-1}\sum_{i=1}^{K}Z_{t}^{i,N}$ and $\bar{U}^{N,K}_{t}=K^{-1}\sum_{i=1}^{K}U_{t}^{i,N}$.
By \cite[Remark 10 and Lemma 11]{A}, we have the following  equalities
\begin{align}
\label{ee1}&\mathbb{E}_{\theta}[\boldsymbol{Z}_{t}^{N,K}]=\mu\sum_{n\ge0}\Big[\int_{0}^{t}s\phi^{*n}(t-s)ds\Big]I_{K}A_{N}^{n}\boldsymbol{1}_{N},\\
\label{ee2}&\boldsymbol{U}_{t}^{N,K}=\sum_{n\ge0}\int_{0}^{t}\phi^{*n}(t-s)I_{K}A_{N}^{n}\boldsymbol{M}_{s}^{N}ds,\\
\label{ee3}&[M^{i,N},M^{j,N}]_{t}=\boldsymbol{1}_{\{i=j\}}Z_{t}^{i,N}.
\end{align}
We recall that $\phi^{*0}=\delta_0$, whence in particular
$\int_{0}^{t}s\phi^{*0}(t-s)ds=t$.

\begin{lemma}\label{Zt}
Assume $H(q)$ for some $q \ge 1$. There exists a constant $C$ such that

(i) for all $r$ in $[1,\infty]$, all  $t \ge0$, a.s.,
$$
\boldsymbol{1}_{\Omega_{N,K}}\|\mathbb{E}_{\theta}[\boldsymbol{Z}_{t}^{N,K}]\|_{r}\le CtK^{\frac{1}{r}},
$$

(ii) for all $r$ in $[1, \infty]$,  all $t \ge s \ge 0$, a.s.,
$$
\boldsymbol{1}_{\Omega_{N,K}}\|\mathbb{E}_{\theta}[\boldsymbol{Z}^{N,K}_{t}-\boldsymbol{Z}_{s}^{N,K}-\mu(t-s)\boldsymbol{\ell}_{N}^{K}]\|_{r}\le C(\min\{1,s^{1-q}\})K^{\frac{1}{r}}.
$$
\end{lemma}

\begin{proof}
(i) We start from \eqref{ee1}.
Recall that $\Lambda=\int_{0}^{\infty}\phi(s)ds,$ whence  
$$
\int_{0}^{\infty}\phi^{*n}(s) ds \le\Lambda^{n},\qquad \int_{0}^{t}s\phi^{*n}(s) ds
\le t\int_{0}^{\infty}\phi^{*n}(s) ds\le t\Lambda^{n}.
$$ 
So on the event  $\Omega_{N,K}$, on which we have $\Lambda |||I_K A_N|||_r\leq (K/N)^{1/r}$
and $\Lambda |||A_N|||_r \leq a<1$, we have (observe that $||\indiq_K||_r=K^{1/r}$)
\begin{align*}
\|\mathbb{E}_{\theta}[\boldsymbol{Z}_{t}^{N,K}]\|_{r}&\le\mu tK^\frac{1}{r} + \mu t\sum_{n\ge1}\Lambda^{n}|||I_{K}A_{N}^{n}|||_{r}\|\boldsymbol{1}_{N}\|_{r}\\
&\le\mu tK^\frac{1}{r} + \mu t\sum_{n\ge1}\Lambda^{n}|||I_{K}A_{N}|||_{r}|||A_{N}|||_{r}^{n-1}\|\boldsymbol{1}_{N}\|_{r}
\leq C t K^{1/r}.
\end{align*}

(ii) By \eqref{ee1} and Lemma \ref{EZ}, we have 
$$
\mathbb{E}_{\theta}[\boldsymbol{Z}_{t}^{N,K}]-\mathbb{E}_{\theta}[\boldsymbol{Z}_{s}^{N,K}]=\mu(t-s)\sum_{n\ge0}\Lambda^{n}I_{K}A_{N}^{n}\boldsymbol{1}_{N}+\mu\Big(\sum_{n\ge0}[\varepsilon_{n}(t)-\varepsilon_{n}(s)]I_{K}A_{N}^{n}\boldsymbol{1_{N}}\Big)
$$
with  $0\le\varepsilon_{n}(t)\le C\min\{n^{q}\Lambda^{n}t^{1-q},n\Lambda^{n}k\}$.
Since $\sum_{n\ge0}\Lambda^{n}I_{K}A_{N}^{n}\boldsymbol{1_{N}}=I_{K}Q_{N}\boldsymbol{1}_{N}=\boldsymbol{\ell}_{N}^{K}$
on the event $\Omega_{N,K}$, 
\begin{align*}
&\|\mathbb{E}_{\theta}[\boldsymbol{Z}_{t}^{N,K}]-\mathbb{E}_{\theta}[\boldsymbol{Z}_{s}^{N,K}]-\mu(t-s)\boldsymbol{\ell}_{N}^{K}\|_{r}\\
  \le& C(\min\{1,s^{1-q}\})\|\boldsymbol{1}_{N}\|_{r}\sum_{n\ge 0}n^{q}\Lambda^{n}|||I_{K}A_{N}^{n}|||_{r}\\
  \le&  C(\min\{1,s^{1-q}\})N^{1/r}\Big(\sum_{n\ge 1}n^{q}\Lambda^{n}|||I_{K}A_{N}|||_{r}|||A_{N}|||^{n-1}_{r}\Big)
  \le C\min\{1,s^{1-q}\}K^{\frac{1}{r}}.
\end{align*}
We used the very same arguments as in point (i).
\end{proof}

\section{The first estimator in the subcritical case}
Here we prove that
$\varepsilon _{t}^{N,K}=t^{-1}(\bar{Z}_{2t}^{N,K}-\bar{Z}_{t}^{N,K})\simeq \frac{\mu}{1-\Lambda p}$ and to study
the rate of convergence.

\begin{theorem}\label{lmain}
Assume $(H(q))$ for some $q\geq 1$.
There are some positive constants $C$, $C^{'}$ depending only on $p$, $\mu,\ \phi$ and $q$ such that for all 
$\varepsilon \in (0,1)$, all $N\ge K\ge 1,\ $ all $t$ $\ge 1$,
$$
P\Big( \Big|\varepsilon^{N,K}_{t}-\frac{\mu}{1-\Lambda p}\Big|\ge\varepsilon\Big)\le CNe^{-C^{'}K}+\frac{C}{\varepsilon}
\Big(\frac{1}{\sqrt{NK}}+\frac{1}{\sqrt{Kt}}+\frac{1}{t^{q}}\Big).
$$
\end{theorem}

\begin{lemma}\label{eU}
Assume $(H(q))$ for some $q\geq 1$. There is a constant $C>0$ such that a.s.,
$$
(i) \;\;  \boldsymbol{1}_{\Omega_{N,K}}\Big|\mathbb{E}_{\theta}[\varepsilon_{t}^{N,K}]-\mu\bar{\ell}_{N}^{K}\Big|\le \frac{C}{t^{q}}, \quad
(ii)\;\;\boldsymbol{1}_{\Omega_{N,K}}\mathbb{E}_{\theta}[|\bar{U}_{t}^{N,K}|^{2}]\le\frac{Ct}{K}.
$$

\end{lemma}

\begin{proof}
By  Lemma \ref{Zt} (ii), 
$$
\Big|\mathbb{E}_{\theta}[\varepsilon_{t}^{N,K}]-\mu\bar{\ell}_{N}^{K}\Big| \le\frac{1}{K}\Big\|\mathbb{E}_{\theta}\Big[\frac{\boldsymbol{Z}_{2t}^{N,K}-\boldsymbol{Z}_{t}^{N,K}}{t}\Big]-\mu\boldsymbol{\ell}_{N}^{K}\Big\|_{1}\le\frac{C}{t^{q}}.
$$
which proves (i).
Using  (\ref{ee2}),
$$
\bar{U}_{t}^{N,K}=\frac{1}{K}\sum_{n\ge 0}\int_{0}^{t}\phi^{*n}(t-s)\sum_{i=1}^{K}\sum_{j=1}^{N}A_{N}^{n}(i,j)M_{s}^{j,N}ds.
$$ 
Recalling (\ref{ee3}), it is obvious that for $n\geq 1$,
$$
\Et\Big[\Big(\sum_{i=1}^K\sum_{j=1}^N A_N^n(i,j)M^{j,N}_s\Big)^2\Big]= \sum_{j=1}^N \Big( \sum_{i=1}^K A_N^n(i,j)\Big)^2 
\Et[Z^{j,N}_s] \leq |||I_KA_N|||_1^{2n} \sum_{j=1}^N\Et[Z^{j,N}_s].
$$ 
By Lemma \ref{Zt}-(i) with $r=1$, we have $\boldsymbol{1}_{\Omega_{N,K}}\mathbb{E}_{\theta}[\sum_{i=1}^NZ_t^{i,N}]\le CtN$ and $\boldsymbol{1}_{\Omega_{N,K}}\mathbb{E}_{\theta}[\sum_{i=1}^KZ_t^{i,N}]\le CtK.$ Thus on $\Omega_{N,K}$,
\begin{align*} 
&\mathbb{E}_{\theta}[|\bar{U}_{t}^{N,K}|^{2}]^{\frac{1}{2}}\\
&\le  \frac{1}{K}\mathbb{E}_{\theta}\Big[\Big(\sum_{i=1}^KM_t^{i,N}\Big)^2\Big]^{\frac{1}{2}}+\frac{1}{K}\sum_{n\geq 1} \int_0^t \phi^{\star n}(t-s) 
\Et\Big[\Big(\sum_{i=1}^K\sum_{j=1}^N A_N^n(i,j)M^{j,N}_s\Big)^2\Big]^{1/2} ds\\
&\le \frac{1}{K}\mathbb{E}_{\theta}\Big[\sum_{i=1}^KZ_t^{i,N}\Big]^{\frac{1}{2}}+\frac{C}{K}\sum_{n\ge1}|||I_{K}A_{N}^{n}|||_{1}\int^{t}_{0}\mathbb{E}_{\theta}\Big[\sum_{i=1}^NZ_s^{i,N}\Big]^{\frac{1}{2}}\phi^{*n}(t-s)ds\\ 
&\le \frac{C\sqrt t}{\sqrt{K}}+C \frac{(tN)^{\frac{1}{2}}}{K}\sum_{n\ge1}\Lambda^{n}|||I_{K}A_{N}|||_{1}|||A_{N}|||_{1}^{n-1}.
\end{align*}
We used that $\int_0^t \sqrt s \phi^{*n}(t-s)ds \leq \sqrt t \int_0^t \phi^{*n}(t-s)ds \leq \sqrt t \Lambda^n$.
As a conclusion, still on $\Omega_{N,K}$, since $|||I_{K}A_{N}|||_{1} \leq C K/N$ and
$\Lambda|||A_N|||_1 \leq a<1$,
$$
\mathbb{E}_{\theta}[|\bar{U}_{t}^{N,K}|^{2}]^{\frac{1}{2}}\le C\sqrt t\Big(\frac{1}{\sqrt{K}}+\frac{1}{\sqrt{N}}\Big)\le \frac{C\sqrt t}{\sqrt{K}}
$$
as desired.
\end{proof}

\begin{lemma}\label{emain}
Assume $(H(q))$ for some $q\geq 1$. There is $C>0$ such that a.s.,
$$
\boldsymbol{1}_{\Omega_{N,K}}\mathbb{E}_{\theta}\Big[\Big|\varepsilon_{t}^{N,K}-\mu\bar{\ell}_{N}^{K}\Big|^{2}\Big]\le C\Big(\frac{1}{t^{2q}}+\frac{1}{tK}\Big).
$$
\end{lemma}

\begin{proof}
It suffices to write 
\begin{align*}
\mathbb{E}_{\theta}\Big[\Big|\varepsilon_{t}^{N,K}-\mu\bar{\ell}_{N}^{K}\Big|^{2}\Big]&\le2\mathbb{E}_{\theta}\Big[\Big\lvert\varepsilon^{N,K}_{t}-\mathbb{E}_{\theta}[\varepsilon_{t}^{N,K}]\Big\lvert^{2}\Big]+2\Big|\mathbb{E}_{\theta}[\varepsilon_{t}^{N,K}]-\mu\bar{\ell}_{N}^{K}\Big|^{2}\\
 &\le\frac{4}{t^{2}}\Big(\mathbb{E}_{\theta}[|\bar{U}_{2t}^{N,K}|^{2}] +\mathbb{E}_{\theta}[|\bar{U}_{t}^{N,K}|^{2}]\Big)+2\Big|\mathbb{E}_{\theta}[\varepsilon_{t}^{N,K}]-\mu\bar{\ell}_{N}^{K}\Big|^{2}
  \end{align*}
and to use Lemma \ref{eU}.
\end{proof}

Finally, we can give the proof of Theorem \ref{lmain}.

\begin{proof}
By Lemmas \ref{ellp} and \ref{emain},  we have 
\begin{align*}
\mathbb{E} \Big[\boldsymbol{1}_{\Omega_{N,K}}\Big|\varepsilon_{t}^{N,K}-\frac{\mu}{1-\Lambda p}\Big| \Big]
&\le\mathbb{E}\Big[\boldsymbol{1}_{\Omega_{N,K}}\Big|\varepsilon_{t}^{N,K}-\mu\bar{\ell}_{N}^{K}\Big|^2\Big]^\frac{1}{2}+\mu\mathbb{E}\Big[\boldsymbol{1}_{\Omega_{N,K}}\Big|\bar{\ell}_{N}^{K}-\frac{1}{1-\Lambda p}\Big|^2\Big]^\frac{1}{2}\\
&\le C\Big(\frac{1}{\sqrt{Kt}}+\frac{1}{t^{q}}+\frac{1}{\sqrt{NK}}\Big).
\end{align*}
By Chebyshev's Inequality,  we deduce
\begin{align*}
P\Big(\Big|\varepsilon ^{N,K}_{t}-\frac{\mu}{1-\Lambda p}\Big|\ge\varepsilon\Big)
&\le P(\Omega_{N,K}^{c})+P\Big(\Big\{\Big|\varepsilon ^{N,K}_{t}-\frac{\mu}{1-\Lambda p}\Big|\ge\varepsilon\Big\}\cap\Omega_{N,K}\Big)\
\\
&\le P(\Omega_{N,K}^{c})+\frac{1}{\varepsilon}\mathbb{E} \Big[\boldsymbol{1}_{\Omega_{N,K}}\Big|\varepsilon_{t}^{N,K}-\frac{\mu}{1-\Lambda p}\Big| \Big]\\
&\le CNe^{-C^{'}K}+\frac{C}{\varepsilon}\Big(\frac{1}{\sqrt{NK}}+\frac{1}{t^{q}}+\frac{1}{\sqrt{Kt}}\Big)
\end{align*}
by Lemma \ref{ONK}.
\end{proof}

\section{The second estimator in the subcritical case}
 
We now prove that $\mathcal{V} _{t}^{N,K}
:=\frac{N}{K}\sum_{i=1}^{K}\Big[\frac{Z_{2t}^{i,N}-Z_{t}^{i,N}}t-\varepsilon_{t}^{N,K}\Big]^{2}-\frac{N}{t}\varepsilon_{t}^{N,K}\simeq \frac{\mu^{2}\Lambda^{2}p(1-p)}{(1-\Lambda p)^{2}}$.

\begin{theorem}\label{vv}
Assume $H(q)$ for some $q\ge 1$. There is $C>0$ such that for all $ t\ge 1$, a.s.,
\begin{align*}
\boldsymbol{1}_{\Omega_{N,K}}\mathbb{E}_{\theta}\Big[\Big|\mathcal{V}_{t}^{N,K}-\mathcal{V}_{\infty}^{N,K}\Big|\Big]\le C\Big(\frac{N}{t\sqrt{K}}+\frac{N}{t^{q}}+\frac{N}{K\sqrt{t}}\|\boldsymbol{x}_{N}^{K}\|_{2}\Big),\ 
\hbox{where}\ \mathcal{V}_{\infty}^{N,K}:=\frac{\mu^2 N}{K}\|\cx_N^K\|^2_2.
\end{align*}
\end{theorem} 

We write  
$|\mathcal{V}_{t}^{N,K}-\mathcal{V}_{\infty}^{N,K}|\le \Delta_{t}^{N,K,1}+\Delta_{t}^{N,K,2}+\Delta_{t}^{N,K,3},$
where 
\begin{align*}
\Delta_{t}^{N,K,1}&=\frac{N}{K}\Big|\sum_{i=1}^{K}\Big[(Z_{2t}^{i,N}-Z_{t}^{i,N})/t-\varepsilon_{t}^{N,K}\Big]^{2}-\sum_{i=1}^{K}\Big[(Z_{2t}^{i,N}-Z_{t}^{i,N})/t-\mu\bar{\ell}_{N}^{K}\Big]^{2}\Big|, \\
\Delta_{t}^{N,K,2}&=\frac{N}{K}\Big|\sum_{i=1}^{K}\Big[(Z_{2t}^{i,N}-Z_{t}^{i,N})/t-\mu\ell_{N}(i)\Big]^{2}-(K/t)\varepsilon_{t}^{N,K}\Big|,\\
\Delta_{t}^{N,K,3}&=2\frac{N}{K}\Big|\sum_{i=1}^{K}\Big[Z_{2t}^{i,N}-Z_{t}^{i,N}-\mu\ell_{N}(i)\Big]\Big[\mu\ell_{N}(i)-\mu\bar{\ell}_{N}^{K}\Big]\Big|.
\end{align*}
We also write 
$\Delta_{t}^{N,K,2}\le \Delta_{t}^{N,K,21}+\Delta_{t}^{N,K,22}+\Delta_{t}^{N,K,23},$
where
\begin{align*}
\Delta_{t}^{N,K,21}&=\frac{N}{K}\Big|\sum_{i=1}^{K}\Big[(Z_{2t}^{i,N}-Z_{t}^{i,N})/t-\mathbb{E}_{\theta}[Z_{2t}^{i,N}-Z_{t}^{i,N}]/t\Big]^{2}-(K/t)\varepsilon_{t}^{N,K}\Big|,\\
\Delta_{t}^{N,K,22}&=\frac{N}{K}\Big|\sum_{i=1}^{K}\Big\{\mathbb{E_{\theta}}[(Z_{2t}^{i,N}-Z_{t}^{i,N})/t]-\mu\ell_{N}(i)\Big\}^{2}\Big|,\\
\Delta_{t}^{N,K,23}&=2\frac{N}{K}\Big|\sum_{i=1}^{K}\Big[(Z_{2t}^{i,N}-Z_{t}^{i,N})/t-\mathbb{E}_{\theta}(Z_{2t}^{i,N}-Z_{t}^{i,N})/t\Big]\Big[\mathbb{E}_{\theta}(Z_{2t}^{i,N}-Z_{t}^{i,N})/t-\mu\ell_{N}(i)\Big]\Big|.
\end{align*}
We next write  $\Delta_{t}^{N,K,21}\le \Delta_{t}^{N,K,211}+\Delta_{t}^{N,K,212}+\Delta_{t}^{N,K,213}$, where
\begin{align*}
\Delta_{t}^{N,K,211}&= \frac{N}{K}\Big|\sum_{i=1}^{K}\Big\{(U_{2t}^{i,N}-U_{t}^{i,N})^{2}/t^{2}-\mathbb{E_{\theta}}[(U_{2t}^{i,N}-U_{t}^{i,N})^{2}/t^{2}]\Big\}\Big|,\\
 \Delta_{t}^{N,K,212}&= \frac{N}{K}\Big|\sum_{i=1}^{K}\mathbb{E_{\theta}}[(U_{2t}^{i,N}-U_{t}^{i,N})^{2}/t^{2}]-\mathbb{E_{\theta}}[K\varepsilon_{t}^{N,K}/t]\Big|,\\
 \Delta_{t}^{N,K,213}&= \frac{N}{K}\Big|K\varepsilon_{t}^{N,K}/t-\mathbb{E_{\theta}}[K\varepsilon_{t}^{N,K}/t]\Big|,
\end{align*}
At the last, we write $\Delta_{t}^{N,K,3}\le \Delta_{t}^{N,K,31}+\Delta_{t}^{N,K,32}$, where
\begin{align*}
\Delta_{t}^{N,K,31}&=2\frac{N}{K}\Big|\sum_{i=1}^{K}\Big[(Z_{2t}^{i,N}-Z_{t}^{i,N})/t-\mathbb{E_{\theta}}[(Z_{2t}^{i,N}-Z_{t}^{i,N})/t]\Big]\Big[\mu\ell_{N}(i)-\mu\bar{\ell}_{N}^{K}\Big]\Big|,\\
 \Delta_{t}^{N,K,32}&=2\Big|\sum_{i=1}^{K}\Big[\mathbb{E_{\theta}}[(Z_{2t}^{i,N}-Z_{t}^{i,N})/t]-\mu\ell_{N}(i)\Big]\Big[\mu\ell_{N}(i)-\mu\bar{\ell}_{N}^{K}\Big]\Big|.
\end{align*}

\begin{lemma}\label{12345}
Assume $H(q)$ for some $q\ge 1$. Then, on the set $\Omega_{N,K}$, for $t\ge 1$, a.s.,

(i) $\mathbb{E}_{\theta}[\Delta_{t}^{N,K,1}]\le C(Nt^{-2q}+NK^{-1}t^{-1}),$

(ii) $\mathbb{E}_{\theta}[\Delta_{t}^{N,K,22}]\le CN/t^{2q},$

(iii) $\mathbb{E}_{\theta}[\Delta_{t}^{N,K,23}]\le CN/t^{q},$

(iv) $\mathbb{E}_{\theta}[\Delta_{t}^{N,K,213}]\le CNK^{-\frac{1}{2}}t^{-\frac{3}{2}},$

(v) $\mathbb{E}_{\theta}[\Delta_{t}^{N,K,32}]\le CN/t^{q}.$
\end{lemma}

\begin{proof}
(i)  Recalling the definition  $\varepsilon _{t}^{N,K}=t^{-1}(\bar{Z}_{2t}^{N,K}-\bar{Z}_{t}^{N,K})$,
\begin{align*}
\Delta_{t}^{N,K,1}
&=\frac{N}{K}\Big|\sum_{i=1}^{K}\Big[\mu\bar{\ell}_{N}^{K}-\varepsilon_{t}^{N,K}\Big]\Big[2(Z_{2t}^{i,N}-Z_{t}^{i,N})/t-\mu\bar{\ell}_{N}^{K}-\varepsilon_{t}^{N,K}\Big]\Big|
=N\Big(\varepsilon_{t}^{N,K}-\mu\bar{\ell}_{N}^{K}\Big)^{2},
\end{align*}
whence by Lemma \ref{emain},
$$\mathbb{E}_{\theta}[\Delta_{t}^{N,1,K}]=N\mathbb{E}_{\theta}\Big[\Big(\varepsilon_{t}^{N,K}-\mu\bar{\ell}_{N}^{K}\Big)^{2}\Big]\le C\Big(Nt^{-2q}+ NK^{-1}t^{-1}\Big).$$
(ii) We use  Lemma \ref{Zt}-(ii) with $r=2$:
\begin{align*}
\mathbb{E}_{\theta}[\Delta_{t}^{N,K,22}]=\frac{N}{K}\sum_{i=1}^{K}\Big\{\mathbb{E}_{\theta}[(Z_{2t}^{i,N}-Z_{t}^{i,N})/t]-\mu\ell_{N}(i)\Big\}^{2}
\le  CN/t^{2q}.
\end{align*}
(iii) By Lemma \ref{Zt}-(i) with $r=\infty$ and  \ref{Zt}-(ii),
\begin{align*}
\mathbb{E}_{\theta}[\Delta_{t}^{N,K,23}]
\le \frac{4N}{K}\Big\|\mathbb{E}_{\theta}\Big[\Big(\boldsymbol{Z}_{2t}^{N,K}-\boldsymbol{Z}_{t}^{N,K}\Big)/t\Big]-\mu\ell_{N}^{K}\Big\|_{\infty}\Big\|\mathbb{E}_{\theta}\Big[\boldsymbol{Z}_{2t}^{N,K}+\boldsymbol{Z}_{t}^{N,K}\Big]\Big\|_{1}t^{-1}\le \frac{CN}{t^{q}}.
\end{align*}
(iv) Since 
$$\Delta_{t}^{N,K,213}=Nt^{-2}\Big|\bar{U}_{2t}^{N,K}-\bar{U}^{N,K}_{t}\Big|\le Nt^{-2}\Big(|\bar{U}^{N,K}_{2t}|+|\bar{U}^{N,K}_{t}|\Big)$$
and thanks to Lemma \ref{eU}-(ii), we deduce that 
$$\mathbb{E_{\theta}}[\Delta_{t}^{N,K,213}]\le CNK^{-\frac{1}{2}}t^{-\frac{3}{2}}.$$
(v) Since $\max_{j=1,...,N}[\ell_{N}(j)]$ is bounded  on the set $\Omega_{N,K}$, by Lemma \ref{Zt}-(ii) with $r=1$,
$$
\mathbb{E}_{\theta}[\Delta_{t}^{N,K,32}]\le C \mathbb{E}_{\theta}\Big[\frac{N}{Kt}\|\boldsymbol{Z}_{2t}^{N,K}-\boldsymbol{Z}_{t}^{N,K}-\mu t\boldsymbol{\ell}_{N}^{K}\|_{1}\Big]\le \frac{CN}{t^{q}}.
$$
The proof is complete.
\end{proof}

\begin{lemma}\label{212}
Assume $H(q)$ for some $q\ge 1$. We have, for all $t\ge 1$, on the set $\Omega_{N,K}$, a.s.
$$\mathbb{E}_{\theta}[\Delta_{t}^{N,K,212}]\le C/t.$$
\end{lemma}

\begin{proof}
We  write
$\mathbb{E}_{\theta}[\Delta_{t}^{N,K,212}]\le t^{-2}\frac{N}{K}\sum_{i=1}^{K}a_{i}$,  where  
$a_{i}=|\mathbb{E}_{\theta}[(U_{2t}^{i,N}-U_{t}^{i,N})^{2}-(Z_{2t}^{i,N}-Z_{t}^{i,N})]|$, and then
$$a_{i}=b_{i}+d_{i}
\quad \hbox{where} \quad  a_i=\mathbb{E}_{\theta}[(R_{t}^{i,N})^{2}] \quad \hbox{and} \quad
b_i=2\mathbb{E}_{\theta}[(M_{2t}^{i,N}-M_{t}^{i,N})R_{t}^{i,N}],
$$ 
where, recalling \eqref{ee2}, we have $U_{2t}^{i,N}-U_{t}^{i,N}=M_{2t}^{i,N}-M_{t}^{i,N}+R_{t}^{i,N}$, with
\begin{align*}
R_{t}^{i,N}=\sum_{n\ge 1}\int_{0}^{2t}\beta_{n}(t,2t,s)\sum_{j=1}^{N}A_{N}^{n}(i,j)M_{s}^{j,N}ds \quad \hbox{with}\quad
\beta_n(t,2t,s)=\phi^{\star n}(2t-s) - \phi^{\star n}(t-s).
\end{align*}
This uses that $\mathbb{E}_{\theta}[(M_{2t}^{i,N}-M_{t}^{i,N})^2]=\mathbb{E}_{\theta}[Z_{2t}^{i,N}-Z_{t}^{i,N}]$
by \eqref{ee3}.
By the proof of \cite[Lemma 21, lines 10 and 15]{A}, we have $b_{i}\le CtN^{-1}$ and $d_{i}\le CtN^{-1}$,
whence the conclusion.
\end{proof}

Before considering the term  $\Delta_{t}^{N,K,31}$,  we review \cite[Lemma 22]{A}
(observing that $\Omega_{N,K}\subset\Omega_{N}^{1}$). 

\begin{lemma}\label{1234}
Assume $H(q)$ for some $q\ge 1$. Then for all $t \ge 1$ and $k, l, a, b \in\{1,\dots ,N\},$  all $r, s, u, v 
\in[0,t]$, on the set $\Omega_{N,K}$ a.s,
\vip
(i) $|\mathrm{Cov}_{\theta}(Z_{r}^{k,N},Z_{s}^{l,N})|=|\mathrm{Cov}_{\theta}(U_{r}^{k,N},U_{s}^{l,N})|\le Ct(N^{-1}+\boldsymbol{1}_{\{k=l\}}),$ 
\vip
(ii) $|\mathrm{Cov}_{\theta}(Z_{r}^{k,N},M_{s}^{l,N})|=|\mathrm{Cov}_{\theta}(U_{r}^{k,N},M_{s}^{l,N})|\le Ct(N^{-1}+\boldsymbol{1}_{\{k=l\}}),$
\vip
(iii) $|\mathrm{Cov}_{\theta}(Z_{r}^{k,N},\int_{0}^{s}M_{\tau-}dM_{\tau}^{l,N})|=|\mathrm{Cov}_{\theta}(U_{r}^{k,N},\int_{0}^{s}M_{\tau-}dM_{\tau}^{l,N})|\le Ct^{\frac{3}{2}}(N^{-1}+\boldsymbol{1}_{\{k=l\}}),$
\vip
(iv) $|\mathbb{E_{\theta}}[M_{r}^{k,N}M_{s}^{k,N}M_{u}^{l,N}]|\le \frac{Ct}{N},$ if  $\#\{k,l\}=2,$
\vip
(v) $|\mathrm{Cov}_{\theta}(M_{r}^{k,N}M_{s}^{l,N},M_{u}^{a,N}M_{v}^{b,N})|=0,$ if  $\#\{k,l,a,b\}=4,$
\vip
(vi) $|\mathrm{Cov}_{\theta}(M_{r}^{k,N}M_{s}^{l,N},M_{u}^{a,N}M_{v}^{b,N})|\le Ct/N^{2},$  if  $\#\{k,a,b\}=3,$
\vip
(vii) $|\mathrm{Cov}_{\theta}(M_{r}^{k,N}M_{s}^{l,N},M_{u}^{a,N}M_{v}^{a,N})|\le CN^{-1}t^{\frac{3}{2}},$  if  $\#\{k,a\}=2,$
\vip
(viii) $|\mathrm{Cov}_{\theta}(M_{r}^{k,N}M_{s}^{l,N},M_{u}^{a,N}M_{v}^{b,N})|\le Ct^{2}.$
\end{lemma}

\begin{lemma}\label{31}
Assume $H(q)$ for some $q\ge$ 1. Then for $t\ge 1$ on  $\Omega_{N,K}$ a.s.,
$$
\mathbb{E_{\theta}}[(\Delta_{t}^{N,K,31})^{2}]\le \frac{CN^{2}}{tK^{2}}\sum_{i=1}^{K}\Big(\ell_{N}(i)-\bar{\ell}_{N}^{K}\Big)^{2}.
$$
\end{lemma}

\begin{proof}
By definition of $\Delta_{t}^{N,K,31},$
$$
\mathbb{E_{\theta}}[(\Delta_{t}^{N,K,31})^{2}]=\frac{4\mu^{2}N^{2}}{t^{2}K^{2}}\sum_{i,j=1}^{K}(\ell_{N}(i)-\bar{\ell}_{N}^{K})(\ell_{N}(j)-\bar{\ell}_{N}^{K})\mathrm{Cov}_{\theta}(U_{2t}^{i,N}-U_{t}^{i,N},U_{2t}^{j,N}-U_{t}^{j,N}).
$$
By Lemma \ref{1234} (i), we have
$
\mathrm{Cov}_{\theta}[U_{2t}^{i,N}-U_{t}^{i,N},U_{2t}^{j,N}-U_{t}^{j,N}]\le Ct(\boldsymbol{1}_{\{i=j\}}+\frac{1}{N}).
$
We deduce that
\begin{align*}
\mathbb{E}_{\theta}[(\Delta_{t}^{N,K,31})^{2}]&\le\frac{C\mu^{2}N^{2}}{t^{2}K^{2}}t
\sum_{i,j=1}^{K}\Big(\boldsymbol{1}_{\{i=j\}}+\frac{1}{N}\Big)\Big\{[\ell_{N}(i)-\bar{\ell}_{N}^{K}]^{2}+[\ell_{N}(j)-\bar{\ell}_{N}^{K}]^{2}\Big\}\\
&\le\frac{C}{t}\frac{N^{2}}{K^{2}}\sum_{i=1}^{K}\Big(\ell_{N}(i)-\bar{\ell}_{N}^{K}\Big)^{2}.
\end{align*}
We ,finally used that $K/N\leq 1$.
\end{proof}

Next, we deal with the term $\Delta_{t}^{N,K,211}$.
\begin{lemma}\label{211}
Assume $H(q)$ for some $q\ge$1. Then for all $t\ge 1$, a.s. on the set $\Omega_{N,K}$, we have
$$\mathbb{E_{\theta}}[(\Delta_{t}^{N,K,211})^{2}]\le \frac{CN^{2}}{Kt^{2}}.$$
\end{lemma}

\begin{proof}
First,  $\mathbb{E_{\theta}}[(\Delta_{t}^{N,K,211})^{2}]= \frac{N^{2}}{K^{2}t^{4}}\sum_{i,j=1}^{K}a_{ij}$,
where $a_{ij}=\mathrm{Cov}_{\theta}[(U_{2t}^{i,N}-U_{t}^{i,N})^{2},(U_{2t}^{j,N}-U_{t}^{j,N})^{2}]$.
Let $\Gamma_{k,l,a,b}(t)=\sup_{r,s,u,v\in [0,2t]}|\mathrm{Cov}_{\theta}(M_{r}^{k,N}M_{s}^{l,N},M_{u}^{a,N}M_{v}^{b,N})|$. 
By the proof of \cite[Lemma 24 lines 9 to 12]{A}, we have 
\begin{align*}
a_{ij}\le C\sum_{k,l,a,b=1}^{N}(\boldsymbol{1}_{\{i=k\}}+N^{-1})(\boldsymbol{1}_{\{i=l\}}+N^{-1})(\boldsymbol{1}_{\{j=a\}}+N^{-1})(\boldsymbol{1}_{\{j=b\}}+N^{-1})\Gamma_{k,l,a,b}(t).
\end{align*}
Hence,
$$
\sum_{i,j=1}^{K}a_{ij}\le C[R^{K}_{1}+R^{K}_{2}+R^{K}_{3}+R^{K}_{4}+R^{K}_{5}+R^{K}_{6}],
$$
where
\begin{align*}
R^{K}_{1}&=\frac{1}{N^4}\sum_{i,j=1}^{K}\sum_{k,l,a,b=1}^{N}\Gamma_{k,l,a,b}(t)=\frac{K^{2}}{N^{4}}\sum_{k,l,a,b=1}^{N}\Gamma_{k,l,a,b}(t),\\
R^{K}_{2}&=\frac{1}{N^3}\sum_{i,j=1}^{K}\sum_{k,l,a,b=1}^{N}\boldsymbol{1}_{\{i=k\}}\Gamma_{k,l,a,b=1}(t)=\frac{K}{N^{3}}\sum_{i=1}^{K}\sum_{l,a,b=1}^{N}\Gamma_{i,l,a,b}(t),\\
R^{K}_{3}&=\frac{1}{N^2}\sum_{i,j=1}^{K}\sum_{k,l,a,b=1}^{N}\boldsymbol{1}_{\{i=k\}}\boldsymbol{1}_{\{j=a\}}\Gamma_{k,l,a,b}(t)=\frac{1}{N^{2}}\sum_{k,a=1}^{K}\sum_{b,l=1}^{N}\Gamma_{k,l,a,b}(t),\\
R^{K}_{4}&=\frac{1}{N^2}\sum_{i,j=1}^{K}\sum_{k,l,a,b=1}^{N}\boldsymbol{1}_{\{i=k\}}\boldsymbol{1}_{\{i=l\}}\Gamma_{k,l,a,b}(t)=\frac{K}{N^{2}}\sum_{k=1}^{K}\sum_{a,b=1}^{N}\Gamma_{k,k,a,b}(t),\\
R^{K}_{5}&=\frac{1}{N}\sum_{i,j=1}^{K}\sum_{k,l,a,b=1}^{N}\boldsymbol{1}_{\{i=k\}}\boldsymbol{1}_{\{i=l\}}\boldsymbol{1}_{\{j=a\}}\Gamma_{k,l,a,b}(t)=\frac{1}{N}\sum_{k,a=1}^{K}\sum_{b=1}^{N}\Gamma_{k,k,a,b}(t),\\
R^{K}_{6}&=\sum_{i,j=1}^{K}\sum_{k,l,a,b=1}^{N}\boldsymbol{1}_{\{i=k\}}\boldsymbol{1}_{\{i=l\}}\boldsymbol{1}_{\{j=a\}}\boldsymbol{1}_{\{j=b\}}\Gamma_{k,l,a,b}(t)
=\sum_{k,a=1}^{K}\Gamma_{k,k,a,a}(t).
\end{align*}
By Lemma \ref{1234}-(v)-(viii), we  see that $\Gamma_{k,l,a,b}(t)\le Ct^{2}\boldsymbol{1}_{\{\#\{k,l,a,b\}<4\}}$,
so that 
$$
R^{K}_{1}\le Ct^{2}\frac{K^{2}}{N}, \quad R^{K}_{2}\le Ct^{2}\frac{K^{2}}{N}, \quad \hbox{and} \quad 
R^{K}_{3}\le Ct^{2}K.
$$ 
Also, from Lemma \ref{1234}-(vi)-(viii),  we have 
$\Gamma_{k,k,a,b}(t)\le C(\boldsymbol{1}_{\{\#\{k,a,b\}=3\}}N^{-2}t+\boldsymbol{1}_{\{\#\{k,a,b\}<3\}}t^{2})$, whence
$$
R_{4}^{K}\le C\Big(\frac{K^{2}}{N^{2}}t+\frac{K^{2}}{N}t^{2}\Big)\le C\frac{K^{2}}{N}t^{2}
\quad \hbox{and} \quad 
R_{5}^{K}\le C\Big(Kt^{2}+\frac{K^{2}}{N^{2}}t\Big)\le CKt^{2}. 
$$
Finally, notice that from Lemma \ref{1234}-(vii)-(viii), 
$\Gamma_{k,k,a,a}(t)\le C(\boldsymbol{1}_{\{\#\{k,a\}=2\}}N^{-1}t^{\frac{3}{2}}+\boldsymbol{1}_{\{\#\{k,a\}=1\}}t^{2})$, 
so that
$$
R_{6}^{K}\le C\Big(\frac{K^{2}}{N}t^{\frac{3}{2}}+Kt^{2}\Big)\le CKt^{2}.
$$ 
All in all, we deduce that $\sum_{i,j}^{K}a_{ij}\le CKt^{2}.$
\end{proof}

Then we can give prove of Theorem \ref{vv}.
\vip
\begin{proof}
Recalling that 
\begin{align*}
|\mathcal{V}_{t}^{N,K}-\mathcal{V}_{\infty}^{N,K}|=&\Delta_{t}^{N,K,1}+\Delta_{t}^{N,K,211}+\Delta_{t}^{N,K,212}+\Delta_{t}^{N,K,213}+\Delta_{t}^{N,K,22}\\
&+\Delta_{t}^{N,K,23}+\Delta_{t}^{N,K,31}+\Delta_{t}^{N,K,32},
\end{align*} 
Lemmas \ref{12345}, \ref{212}, \ref{31} and \ref{211} allow us to conclude that
\begin{align*}
\boldsymbol{1}_{\Omega_{N,K}}\mathbb{E}_{\theta}[|\mathcal{V}_{t}^{N,K}-\mathcal{V}_{\infty}^{N,K}|] &
\le C\Big(\frac{N}{t\sqrt{K}}+\frac{N}{K^{\frac{1}{2}}t^{\frac{3}{2}}}+\frac{N}{t^{q}}+\frac{N}{t^{2q}}
+\frac{N}{tK}+\frac{N}{K\sqrt{t}}\Big[\sum_{i=1}^{K}(\ell_{N}(i)-\bar{\ell}_{N}^{K})^{2}\Big]^{\frac{1}{2}}\Big) \\
&\le C\Big(\frac{N}{t\sqrt{K}}+\frac{N}{t^{q}}+\frac{N}{K\sqrt{t}}\|\boldsymbol{x}_{N}^{K}\|_{2}\Big)
\end{align*}
as desired.
\end{proof}

\begin{cor}\label{vmain}
Assume $H(q)$ for some $q > 3$. There exists some constants $C>0$ and $C'>0$ 
depending only on $p$, $\mu$, $\phi$, $q$ such that for all $\varepsilon\in (0,1)$, such that, for $t\geq 1$,
$$
P\Big(\Big|\mathcal{V}_{t}^{N,K}-\frac{\mu^{2}\Lambda^{2}p(1-p)}{(1-\Lambda p)^{2}}\Big|\ge\varepsilon\Big)\le CNe^{-C'K}+\frac{C}{\varepsilon}\Big(\frac{1}{\sqrt{K}}+\frac{N}{t\sqrt{K}}\Big).
$$
\end{cor}

\begin{proof}
By Theorem \ref{vv} and Lemma \ref{mainM2} (since $\mathcal{V}_{\infty}^{N,K}=\mu^2\frac{N}{K}
||\boldsymbol{x}^N_K||^2=
\mu^2\frac{N}{K}\sum_{i=1}^{K}
(\ell_{N}(i)-\bar{\ell}_{N}^{K})^{2}$), we have 
\begin{align*}
&\mathbb{E}\Big[\boldsymbol{1}_{\Omega_{N,K}}\Big|\mathcal{V}_{t}^{N,K}-\frac{\mu^{2}\Lambda^{2}p(1-p)}{(1-\Lambda p)^{2}}\Big|\Big]\\ 
&\le \mathbb{E}[\boldsymbol{1}_{\Omega_{N,K}}|\mathcal{V}_{t}^{N,K}-\mathcal{V}_{\infty}^{N,K}|]+\mu^{2}\mathbb{E}\Big[\boldsymbol{1}_{\Omega_{N,K}}\Big|\frac{N}{K}\sum_{i=1}^{K}\Big(\ell_{N}(i)-\bar{\ell}_{N}^{K}\Big)^{2}-\frac{\Lambda^{2}p(1-p)}{(1-\Lambda p)^{2}}\Big|\Big]\\ 
&\le C\mathbb{E}\Big[\boldsymbol{1}_{\Omega_{N,K}}\Big(\frac{N}{t\sqrt{K}}+\frac{N}{t^{q}}+\frac{N}{K\sqrt{t}}\|\boldsymbol{x}_{N}^{K}\|_{2}\Big)\Big] +\frac{C}{\sqrt{K}}\\
&\le C\Big(\frac{1}{\sqrt{K}}+\frac{\sqrt{N}}{\sqrt{tK}}+\frac{N}{t^{q}}+\frac{N}{t\sqrt{K}}\Big).
\end{align*}
By the classical inequality $\frac{N}{t\sqrt{K}}+\frac{1}{\sqrt{K}}\ge 2\frac{\sqrt{N}}{\sqrt{Kt}}$, we end with
$$\mathbb{E}\Big[\boldsymbol{1}_{\Omega_{N,K}}\Big|\mathcal{V}_{t}^{N,K}-\frac{\mu^{2}\Lambda^{2}p(1-p)}{(1-\Lambda p)^{2}}\Big|\Big]\le C\Big(\frac{N}{t^{q}}+\frac{N}{t\sqrt{K}}+\frac{1}{\sqrt{K}}\Big).$$  
Using Lemma \ref{ONK} and Chebyshev's inequality,  we conclude that
$$
P\Big(\Big|\mathcal{V}_{t}^{N,K}-\frac{\mu^{2}\Lambda^{2}p(1-p)}{(1-\Lambda p)^{2}}\Big|\ge\varepsilon\Big)\le CNe^{-CK}+\frac{C}{\varepsilon}\Big(\frac{1}{\sqrt{K}}+\frac{N}{t^{q}}+\frac{N}{t\sqrt{K}}\Big).
$$ 
Next, we get rid of the term $\frac{N}{t^{q}}$. 
We assume without loss of generality that $C\geq 1$. 
When $t\le \sqrt{K}$, then $\frac{N}{t\sqrt{K}}\ge 1$, so that 
$$P\Big(\Big|\mathcal{V}_{t}^{N,K}-\mu^{2}\Lambda^{2}p(1-p)/(1-\Lambda p)^{2}\Big|\ge\varepsilon\Big)\le 1\le CNe^{-CK}+\frac{C}{\varepsilon}\Big(\frac{1}{\sqrt{K}}+\frac{N}{t\sqrt{K}}\Big).$$ 
When now $t\ge\sqrt{K}$, then $\frac{N}{t\sqrt{K}}\ge \frac Nt \geq \frac{N}{t^{q}}$. 
So 
$$P\Big(\Big|\mathcal{V}_{t}^{N,K}-\mu^{2}\Lambda^{2}p(1-p)/(1-\Lambda p)^{2}\Big|\ge\varepsilon\Big)
\le CNe^{-CK}+\frac{C}{\varepsilon}\Big(\frac{1}{\sqrt{K}}+\frac{N}{t\sqrt{K}}\Big).$$ 
This completes the proof.
\end{proof}

\section{The third estimator in the subcritical case}
Recall that by definition,
\begin{gather*}
\mathcal{W}_{\Delta,t}^{N,K}=2\mathcal{Z}_{2\Delta,t}^{N,K}-\mathcal{Z}_{\Delta,t}^{N,K},\quad \mathcal{Z}_{\Delta,t}^{N,K}=\frac{N}{t}\sum_{i=\frac{t}{\Delta}+1}^{\frac{2t}{\Delta}}\Big(\bar{Z}_{i\Delta}^{N,K}-\bar{Z}_{(i-1)\Delta}^{N,K}-\Delta\varepsilon_{t}^{N,K}\Big)^{2},\\
\mathcal{X}_{\Delta,t}^{N,K}=\mathcal{W}_{\Delta,t}^{N,K}-\frac{N-K}{K}\varepsilon_t^{N,K}.
\end{gather*}
The goal of this section is to check that 
$\mathcal{X}_{\Delta,t}^{N,K}\simeq\frac{\mu}{(1-\Lambda p)^{3}}$, and more precisely to prove the following
estimate. 
\begin{theorem}\label{noname}
Assume $H(q)$ for some $q\geq 3$. Then a.s., for all $t\geq 4$ and all $\Delta \in [1,t/4]$
such that $t/(2\Delta)$ is a positive integer,
$$
\mathbb{E}\Big[\boldsymbol{1}_{\Omega_{N,K}}\Big|\mathcal{X}_{\Delta,t}^{N,K}-\frac{\mu}{(1-\Lambda p)^{3}}\Big|\Big]\le C\Big(\frac{N}{K}\sqrt{\frac{\Delta}{t}}+\frac{N^{2}}{K\Delta^{\frac{1}{2}(q+1)}}+\frac{Nt}{K\Delta^{\frac{q}{2}+1}}+\frac{N}{K\sqrt{Kt}}\Big).
$$ 
\end{theorem}

In the whole section, we assume that $t\geq 4$ and that $\Delta \in [1,t/4]$ is such that
$t/(2\Delta)$ is a positive integer.
First, we recall that 
$\mathcal{W}^{N,K}_{\infty,\infty}:=\frac{\mu N}{K^{2}}\sum_{j=1}^{N}(c_{N}^{K}(j))^{2}\ell_{N}(j)$
and write
\begin{align*}
|\mathcal{X}_{\Delta,t}^{N,K}-&\mathcal{X}_{\infty,\infty}^{N,K}|\le |\mathcal{W}_{\Delta,t}^{N,K}-\mathcal{W}_{\infty,\infty}^{N,K}|+\frac{N-K}{K}\Big|\varepsilon_t^{N,K}-\bar{\ell}_N^K\Big|\\
\le& D_{\Delta,t}^{N,K,1}+2D_{2\Delta,t}^{N,K,1}+D_{\Delta,t}^{N,K,2}+2D_{2\Delta,t}^{N,K,2}+D_{\Delta,t}^{N,K,3}+2D_{2\Delta,t}^{N,K,3}+D_{\Delta,t}^{N,K,4}+\frac{N}{K}\Big|\varepsilon_t^{N,K}-\bar{\ell}_N^K\Big|,
\end{align*}
where
\begin{align*}
D_{\Delta,t}^{N,K,1}=&\frac{N}{t}\Big|\sum_{a=\frac{t}{\Delta}+1}^{\frac{2t}{\Delta}}\Big(\bar{Z}_{a\Delta}^{N,K}-\bar{Z}_{(a-1)\Delta}^{N,K}-\Delta\varepsilon_{t}^{N,K}\Big)^{2}-\sum_{a=\frac{t}{\Delta}+1}^{\frac{2t}{\Delta}}\Big(\bar{Z}_{a\Delta}^{N,K}-\bar{Z}_{(a-1)\Delta}^{N,K}-\Delta \mu\bar{\ell}_{N}^{K}\Big)^{2}\Big|, \\
D_{\Delta,t}^{N,K,2}=&\frac{N}{t}\Big|\sum_{a=\frac{t}{\Delta}+1}^{\frac{2t}{\Delta}}\Big(\bar{Z}_{a\Delta}^{N,K}-\bar{Z}_{(a-1)\Delta}^{N,K}-\Delta\mu\bar{\ell}_K^{K}\Big)^{2}\\
&-\sum_{a=\frac{t}{\Delta}+1}^{\frac{2t}{\Delta}}\Big(\bar{Z}_{a\Delta}^{N,K}-\bar{Z}_{(a-1)\Delta}^{N,K}-\mathbb{E}_{\theta}[\bar{Z}_{a\Delta}^{N,K}-\bar{Z}_{(a-1)\Delta}^{N,K}]\Big)^{2}\Big|,\\
D_{\Delta,t}^{N,K,3}=&\frac{N}{t}\Big|\sum_{a=\frac{t}{\Delta}+1}^{\frac{2t}{\Delta}}\Big(\bar{Z}_{a\Delta}^{N,K}-\bar{Z}_{(a-1)\Delta}^{N,K}-\mathbb{E}_{\theta}[\bar{Z}_{a\Delta}^{N,K}-\bar{Z}_{(a-1)\Delta}^{N,K}]\Big)^{2}\\
&-\mathbb{E}_{\theta}\Big[\sum_{a=\frac{t}{\Delta}+1}^{\frac{2t}{\Delta}}\Big(\bar{Z}_{a\Delta}^{N,K}-\bar{Z}_{(a-1)\Delta}^{N,K}-\mathbb{E}_{\theta}[\bar{Z}_{a\Delta}^{N,K}-\bar{Z}_{(a-1)\Delta}^{N,K}]\Big)^{2}\Big]\Big|,\\
\end{align*}
and finally
\begin{align*}
D_{\Delta,t}^{N,K,4}=&\Big|\frac{2N}{t}\mathbb{E}_{\theta}\Big[\sum_{a=\frac{t}{2\Delta}+1}^{\frac{t}{\Delta}}\Big(\bar{Z}_{2a\Delta}^{N,K}-\bar{Z}_{2(a-1)\Delta}^{N,K}-\mathbb{E}_{\theta}[\bar{Z}_{2a\Delta}^{N,K}-\bar{Z}_{2(a-1)\Delta}^{N,K}]\Big)^{2}\Big]\\&-\frac{N}{t}\mathbb{E}_{\theta}\Big[\sum_{a=\frac{t}{\Delta}+1}^{\frac{2t}{\Delta}}\Big(\bar{Z}_{a\Delta}^{N,K}-\bar{Z}_{(a-1)\Delta}^{N,K}
-\mathbb{E}_{\theta}[\bar{Z}_{a\Delta}^{N,K}-\bar{Z}_{(a-1)\Delta}^{N,K}]\Big)^{2}\Big]-\mathcal{W}^{N,K}_{\infty,\infty}\Big|.
\end{align*}

For the first term $D_{\Delta,t}^{N,K,1}$, we have the following lemma. 

\begin{lemma}\label{DNK1}
Assume $H(q)$ for some $q\ge 1$. Then a.s. on the set $\Omega_{N,K}$ $$\mathbb{E}_{\theta}[D_{\Delta,t}^{N,K,1}]\le C\Delta\Big(\frac{N}{t^{2q}}+\frac{N}{Kt}\Big).$$
 \end{lemma}
\begin{proof}
Recalling that $\varepsilon _{t}^{N,K}:=t^{-1}(\bar{Z}_{2t}^{N,K}-\bar{Z}_{t}^{N,K}),$ we have
\begin{align*}
D_{\Delta,t}^{N,K,1}=&\frac{N}{t}\Big|\sum_{a=\frac{t}{\Delta}+1}^{\frac{2t}{\Delta}}[\bar{Z}_{a\Delta}^{N,K}-\bar{Z}_{(a-1)\Delta}^{N,K}-\Delta\varepsilon_{t}^{N,K}]^{2}-\sum_{a=\frac{t}{\Delta}+1}^{\frac{2t}{\Delta}}[\bar{Z}_{a\Delta}^{N,K}-\bar{Z}_{(a-1)\Delta}^{N,K}-\Delta \mu\bar{\ell}_{N}^{K}]^{2}\Big|\\
=&N\Delta(\mu \bar{\ell}_{N}^{K}-\varepsilon_{t}^{N,K})^{2},
\end{align*}
Lemma \ref{emain} completes the proof.
\end{proof}

Next, we consider the term $D_{\Delta,t}^{N,K,2}$.

\begin{lemma}\label{DNK2}
Assume $H(q)$ for some $q\ge 1$. Then a.s. on the set $\Omega_{N,K}$, $$\mathbb{E}_{\theta}[D_{\Delta,t}^{N,K,2}]\le CNt^{1-q} $$
\end{lemma}
\begin{proof}
First, we have 
\begin{align*}
D_{\Delta,t}^{N,K,2}=&
\frac{2N}{t}\Big|\sum_{a=\frac{t}{\Delta}+1}^{\frac{2t}{\Delta}}\Big(\Delta\mu\bar{\ell}_{N}^{K}-\mathbb{E}_\theta[\bar{Z}_{a\Delta}^{N,K}-\bar{Z}_{(a-1)\Delta}^{N,K}]\Big)\\
&\hskip3cm\Big(2(\bar{Z}_{a\Delta}^{N,K}-\bar{Z}_{(a-1)\Delta}^{N,K})-\mathbb{E}_\theta\Big[\bar{Z}_{a\Delta}^{N,K}-\bar{Z}_{(a-1)\Delta}^{N,K}\Big]-\Delta \mu \bar{\ell}_N^K\Big)\Big|,
\end{align*}
whence
\begin{align*}
\mathbb{E}_{\theta}[D_{\Delta,t}^{N,K,2}]&\le \frac{2N}{t}\sum_{a=\frac{t}{\Delta}+1}^{\frac{2t}{\Delta}}\Big|\Delta\mu\bar{\ell}_{N}^{K}-\mathbb{E}_\theta[\bar{Z}_{a\Delta}^{N,K}-\bar{Z}_{(a-1)\Delta}^{N,K}]\Big|\Big(\mathbb{E}_\theta\Big[\bar{Z}_{a\Delta}^{N,K}-\bar{Z}_{(a-1)\Delta}^{N,K}\Big]+\Delta \mu \bar{\ell}_N^K\Big).
\end{align*}
By Lemma \ref{Zt}-(i)-(ii) with $r=1$, since $(a-1)\Delta\ge t$, we conclude that on $\Omega_{N,K}$, a.s.,
$$\Big|\Delta\mu\bar{\ell}_{N}^{K}-\mathbb{E}_\theta[\bar{Z}_{a\Delta}^{N,K}-\bar{Z}_{(a-1)\Delta}^{N,K}]\Big|\le Ct^{1-q}
\quad\hbox{and} \quad 
\mathbb{E}_\theta\Big[\bar{Z}_{a\Delta}^{N,K}-\bar{Z}_{(a-1)\Delta}^{N,K}\Big]\le C\Delta \bar{\ell}_N^K+C\leq
C\Delta$$
since $\bar{\ell}_{N}^{K}$ is bounded on $\Omega_{N,K}$. The conclusion follows.
\end{proof}

Next we consider the term $D_{\Delta,t}^{N,K,4}$.

\begin{lemma}\label{UUW}
Assume $H(q)$ for some $q\ge 1$. 
On $\Omega_{N,K}$, there is a $\sigma((\theta_{ij})_{i,j=1...N})$-measurable finite random variable
$\mathcal{Y}^{N,K}$ such that for all $1\le \Delta\le \frac{x}{2}$, a.s. on $\Omega_{N,K}$,
$$
\mathrm{Var}_{\theta}(\bar{U}_{x+\Delta}^{N,K}-\bar{U}_{x}^{N,K})=\frac{\Delta}{N}\mathcal{W}_{\infty,\infty}^{N,K}-\mathcal{Y}^{N,K}+r_{N,K}(x,\Delta),
$$
where, for some constant $C$, $|r_{N,K}(x,\Delta)|\le Cx\Delta^{-q}K^{-1}$.
\end{lemma}

\begin{proof}
Recalling (\ref{ee2}), we write
$$
\bar{U}_{x+\Delta}^{N,K}-\bar{U}_{x}^{N,K}
=\sum_{n\ge 0}\int_{0}^{x+\Delta}\beta_{n}(x,x+\Delta,s)\frac{1}{K}\sum_{i=1}^{K}\sum_{j=1}^{N}A_{N}^{n}(i,j)M_{s}^{j,N}ds,
$$
where 
$\beta_n(x,x+\Delta,s)=\phi^{\star n}(x+\Delta-s) 
- \phi^{\star n}(x-s).$
Set $V_{x,\Delta}^{N,K}=\mathrm{Var}_{\theta}(\bar{U}_{x+\Delta}^{N,K}-\bar{U}_{x}^{N,K})$. Recall  that  $\mathbb{E}[M_{s}^{i,N}M_{t}^{j,N}]=\boldsymbol{1}_{\{i=j\}}\mathbb{E}_{\theta}[Z^{i,N}_{s\wedge t}]$, see (\ref{ee3}). We thus have
\begin{align*}
V_{x,\Delta}^{N,K}\!=\!&\sum_{m,n\ge 0}\!\int_{0}^{x+\Delta}\!\!\!\int_{0}^{x+\Delta}\!\!\!\beta_{m}(x,x+\Delta,r)\beta_{n}(x,x+\Delta,s)\frac{1}{K^{2}}\sum_{i,k=1}^{K}\sum_{j=1}^{N}A_{N}^{m}(i,j)A_{N}^{n}(k,j)
\mathbb{E}_{\theta}[Z^{j,N}_{s\wedge r}]drds.
\end{align*}
In view of \cite[Lemma 28, Step 2]{A}, we have $\mathbb{E}_{\theta}[Z_{s}^{j,N}]=\mu\ell_{N}(j)s-X_{j}^{N}+R_{j}^{N}(s),$
where
\begin{align*}
X_j^N=& \mu\kappa \sum_{n\geq 0} n \Lambda^n \sum_{l=1}^N A_N^n(j,l)\quad \hbox{and}\quad
R_j^N(s) = \mu \sum_{n \geq 0} \varepsilon_n(s)  \sum_{l=1}^N A_N^n(j,l).
\end{align*}
Recall that $\kappa$ and $\varepsilon_n(s)$ were defined in Lemma \ref{EZ}. Also, there is a constant $C$
such that, for all $j=1,...,N$, we have $0\le X_{j}^{N}\le C$ and $|R_{j}^{N}(s)|\le C(s^{1-q}\wedge 1)$.
Then we can write that $V_{x,\Delta}^{N,K}=I-M+Q$,
where
\begin{align*}
I\!=\!&\sum_{n,m\ge 0}\!\!\int_{0}^{x+\Delta}\!\!\!\int_{0}^{x+\Delta}\!\!\!\beta_{n}(x,x+\Delta,s)\beta_{m}(x,x+\Delta,r)\frac{1}{K^{2}}\!\sum_{i,k=1}^{K}\!\sum_{j=1}^{N}\!A_{N}^{m}(i,j)A^{n}_{N}(k,j)\mu \ell_{N}(j)(r\wedge s)drds.\\
M\!=\!&\sum_{n,m\ge 0}\int_{0}^{x+\Delta}\int_{0}^{x+\Delta}\beta_{n}(x,x+\Delta,s)\beta_{m}(x,x+\Delta,r)\frac{1}{K^{2}}\sum_{i,k=1}^{K}\sum_{j=1}^{N}A_{N}^{m}(i,j)A^{n}_{N}(k,j)X_{j}^{N}drds.\\
Q\!=\!&\sum_{n,m\ge 0}\int_{0}^{x+\Delta}\!\!\int_{0}^{x+\Delta}\!\!
\beta_{n}(x,x+\Delta,s)\beta_{m}(x,x+\Delta,r)\frac{1}{K^{2}}\sum_{i,k=1}^{K}\sum_{j=1}^{N}A_{N}^{m}(i,j)A^{n}_{N}(k,j)R_{j}^{N}(r\wedge s)drds.
\end{align*}

First, we consider $M$. Using that $|\int_{0}^{x+\Delta}\beta_{n}(x,x+\Delta,r)dr|\le Cn^{q}\Lambda^{n}x^{-q}$,
see \cite[Lemma 15 (ii)]{A} and that 
$X_{j}^{N}$ is bounded by some constant not depending on  $t$, we conclude that on $\Omega_{N,K}$,
\begin{align*}
|M|\le& C\sum_{m,n\ge 0}m^{q}n^{q}\Lambda^{m+n}x^{-2q}K^{-2}\sum_{i,k=1}^{K}\sum_{j=1}^{N}A_{N}^{m}(i,j)A_{N}^{n}(k,j)\\
\leq & Cx^{-2q}NK^{-2}\sum_{m,n\ge 1}m^{q}n^{q}\Lambda^{m+n}|||I_{K}A_{N}^n|||_{1}|||I_KA_{N}^m|||_{1}\\
\le& Cx^{-2q}NK^{-2}\sum_{m,n\ge 1}m^{q}n^{q}\Lambda^{m+n}|||I_{K}A_{N}|||_{1}^{2}|||A_{N}|||_{1}^{m+n-2}\\
\le& Cx^{-2q}N^{-1}\le Cx\Delta^{-q}K^{-1}.
\end{align*}

Next, we consider $Q.$ We write
\begin{align*}
|Q|\le& C\sum_{m,n\ge 1}\int_{0}^{x+\Delta}\int_{0}^{x+\Delta}\Big|\beta_{m}(x,x+\Delta,r)\Big|\Big|\beta_{n}(x,x+\Delta,s)\Big|\\
&\hskip5cm\frac{N}{K^{2}}|||I_{K}A_{N}|||_{1}^{2}|||A_{N}|||_{1}^{m+n-2}[(r\wedge s)^{1-q}\wedge 1]drds\\
&+2C\sum_{m\ge 0}\int_{0}^{x+\Delta}\int_{0}^{x+\Delta}\Big|\beta_{0}(x,x+\Delta,s)\Big|\Big|\beta_{m}(x,x+\Delta,r)\Big|\frac{1}{K}|||I_{K}A^m_{N}|||_{1}[(r\wedge s)^{1-q}\wedge 1]drds\\
\le& Q_{1}+Q_{2}+2Q_{3}+2Q_{4}.
\end{align*}
where, using that $x-\Delta\ge \frac{x}{2}$ and that $(r\wedge s)^{1-q}\le x^{1-q}$ if $r\wedge s\ge x-\Delta$,
\begin{align*}
Q_{1}=&\frac{C}{x^{q-1}}\sum_{m,n\ge 1}\int_{x-\Delta}^{x+\Delta}\int_{x-\Delta}^{x+\Delta}\Big|\beta_{m}(x,x+\Delta,r)\Big|\Big|\beta_{n}(x,x+\Delta,s)\Big|\frac{N}{K^{2}}|||I_{K}A_{N}|||_{1}^{2}|||A_{N}|||_{1}^{m+n-2}drds,\\
Q_{2}=&C\sum_{m,n\ge 1}\int_{0}^{x-\Delta}\int_{0}^{x+\Delta}\Big|\beta_{m}(x,x+\Delta,r)\Big|\Big|\beta_{n}(x,x+\Delta,s)\Big|\frac{N}{K^{2}}|||I_{K}A_{N}|||_{1}^{2}|||A_{N}|||_{1}^{m+n-2}drds,\\
Q_3=& \frac{C}{x^{q-1}}\sum_{m\ge 0}\int_{0}^{x+\Delta}\int_{x-\Delta}^{x+\Delta}\Big|\beta_{0}(x,x+\Delta,s)\Big|\Big|\beta_{m}(x,x+\Delta,r)\Big|\frac{1}{K}|||I_{K}A^m_{N}|||_{1}drds,\\
Q_4=&C\sum_{m\ge 0}\int_{0}^{x+\Delta}\int_{0}^{x-\Delta}\Big|\beta_{0}(x,x+\Delta,s)\Big|\Big|\beta_{m}(x,x+\Delta,r)\Big|\frac{1}{K}|||I_{K}A^m_{N}|||_{1}drds.
\end{align*}
In view of \cite[Lemma 15-(ii)]{A}, we have the inequalities 
$\int_0^{x+\Delta} |\beta_n(x,x+\Delta,s)| ds \leq 2 \Lambda^n$
and $\int_0^{x-\Delta}|\beta_m(x,x+\Delta,r)|dr \leq Cm^q \Lambda^m \Delta^{-q}$. Hence, on $\Omega_{N,K}$, 
\begin{align*}
Q_{1}&\le Cx^{1-q}\sum_{m,n\ge 1}\Lambda^{m+n}NK^{-2}|||I_{K}A_{N}|||_{1}^{2}|||A_{N}|||_{1}^{m+n-2}\le CN^{-1}x^{1-q}\le Cx\Delta^{-q}K^{-1},\\
 Q_{2}&\le C\Delta^{-q}\sum_{m,n\ge 1}m^{q}\Lambda^{m+n}NK^{-2}|||I_{K}A_{N}|||_{1}^{2}|||A_{N}|||_{1}^{m+n-2}
\le C\Delta^{-q}N^{-1}\le Cx\Delta^{-q}K^{-1}.
\end{align*}
Since furthermore $|\beta_0(x,x+\Delta,s)|=|\delta_{\{s=x+\Delta\}}-
\delta_{\{s=x\}}| \leq \delta_{\{s=x+\Delta\}}+\delta_{\{s=x\}}$, we have
\begin{align*}
Q_{3}&\le Cx^{1-q}\sum_{m\ge 0}\Lambda^{m}K^{-1}|||I_{K}A^m_{N}|||_{1}\le Cx^{1-q}K^{-1}\le Cx\Delta^{-q}K^{-1},\\
Q_{4}&\le C\Delta^{-q}\sum_{m\ge 0}m^{q}\Lambda^{m}K^{-1}|||I_{K}A^m_{N}|||_{1}
\le C \Delta^{-q}K^{-1} \le C x \Delta^{-q}K^{-1}.
\end{align*}
All in all, on $\Omega_{N,K}$, we have $Q\le Cx\Delta^{-q}K^{-1}$.

\vip

Finally we consider  $I$. 
We recall from \cite[Lemma 15 (iii)]{A} that there are $0\leq \kappa_{m,n}\leq (m+n)\kappa$
and a function $\varepsilon_{m,n}:(0,\infty)^2\mapsto \R$ satisfying $|\varepsilon_{m,n}(t,t+\Delta)|\leq 
C (m+n)^q \Lambda^{m+n} t \Delta^{-q}$ such that
\begin{align*}
\gamma_{m,n}(x,x+\Delta)=&\int_{0}^{x+\Delta}\int_{0}^{x+\Delta}(s\wedge u)\beta_{m}(x,x+\Delta,s)\beta_{n}(x,x+\Delta,u)duds\\
=&\Delta \Lambda^{m+n}-\kappa_{m,n}\Lambda^{m+n}+\varepsilon_{m,n}(x,x+\Delta).
\end{align*}
Then we can write $I$ as:
$$
I=\mu \sum_{m,n\ge 0}\gamma_{m,n}(x,x+\Delta)\frac{1}{K^{2}}\sum_{i,k=1}^{K}\sum_{j=1}^{N}A^{m}_{N}(i,j)A_{N}^{n}(k,j)\ell_{N}(j)=I_{1}-I_{2}+I_{3},
$$
where
\begin{align*}
I_{1}=&\mu \Delta\sum_{m,n\ge 0}\Lambda^{m+n}\frac{1}{K^{2}}\sum_{i,k=1}^{K}\sum_{j=1}^{N}A_{N}^{m}(i,j)A_{N}^{n}(k,j)\ell_{N}(j),\\
I_{2}=&\mu\sum_{m,n\ge 0}\kappa_{m,n}\Lambda^{m+n}\frac{1}{K^{2}}\sum_{i,k=1}^{K}\sum_{j=1}^{N}A_{N}^{m}(i,j)A_{N}^{n}(k,j)\ell_{N}(j),\\
I_{3}=&\mu\sum_{m,n\ge 0}\varepsilon_{m,n}(x,x+\Delta)\frac{1}{K^{2}}\sum_{i,k=1}^{K}\sum_{j=1}^{N}A_{N}^{m}(i,j)A_{N}^{n}(k,j)\ell_{N}(j).
\end{align*}
Recalling that $\mathcal{W}_{\infty,\infty}^{N,K}:=\frac{\mu N}{K^{2}}\sum_{j=1}^{N}(c_{N}^{K}(j))^{2}\ell_{N}(j)$ by definition
and that $\sum_{m\ge 0}\Lambda^{m}A_{N}^{m}(i,j)=Q_{N}(i,j)$,
\begin{align*}
I_{1}&=\mu \Delta\sum_{m,n\ge 0}\Lambda^{m+n}\frac{1}{K^{2}}\sum_{i,k}^{K}\sum_{j=1}^{N}A_{N}^{m}(i,j)A_{N}^{n}(k,j)\ell_{N}(j)\\
&=\mu\Delta \frac{1}{K^{2}}\sum_{j=1}^{N}\Big(c_{N}^{K}(j)\Big)^{2}\ell_{N}(j)=\Delta\frac{1}{N}\mathcal{W}_{\infty,\infty}^{N,K}.
\end{align*}
Next, we set $\mathcal{Y}^{N,K}=I_{2}$. It is obvious that $\mathcal{Y}^{N,K}$ is a $(\theta_{ij})_{i,j=1...N}$ measurablefunction and well-defined on $\Omega_{N,K}$.
Finally, using that $\varepsilon_{m,n}(x,x+\Delta)\le C(m+n)^{q}\Lambda^{m+n}x\Delta^{-q}$ and that
$\ell_N$ is bounded on $\Omega_{N,K}$ (we have to treat separately the case $n=0$ or $m=0$),
\begin{align*}
I_3 \leq & C\frac{x}{K^2\Delta^q} \sum_{m\geq 0} m^{q} \Lambda^{m} \sum_{i,k=1}^K A_N^m(k,i) 
+ C\frac{x N}{\Delta^q K^{2}}\sum_{m,n\ge 1}(n+m)^{q}\Lambda^{m+n}|||I_{K}A_{N}^n|||_{1}|||I_K A_{N}^m|||_{1}\\
\leq & C\frac{x}{K \Delta^q} \sum_{m\geq 0} m^{q} \Lambda^{m} |||I_K A_N^m|||_1 +  C\frac{x N}{\Delta^q K^{2}}\sum_{m,n\ge 1}(n+m)^{q}\Lambda^{m+n}|||I_{K}A_{N}|||_{1}^{2}|||A_{N}|||_{1}^{m+n-2}\\
\leq & C \frac{x}{N \Delta^q},
\end{align*}
still on $\Omega_{N,K}$. All in all, we have verified that
$V^{N,K}_{x,\Delta}=I-M+Q$, with 
$$
|M|+|Q|+|I-\Delta N^{-1}\mathcal{W}_{\infty,\infty}^{N,K}+\mathcal{Y}^{N,K}|\le Cx\Delta^{-q}K^{-1},
$$
which completes the proof.
\end{proof}

Next, we  consider the term $D_{\Delta,t}^{N,K,4}$.

\begin{lemma}\label{DNK4}
Assume $H(q)$ for some $q\ge 1$. Then $a.s.$ on $\Omega_{N,K}$, for $1\le \Delta\le \frac{t}{4}$, we have:
$$
\mathbb{E}_{\theta}[D_{\Delta,t}^{N,K,4}]\le C\frac{Nt}{K\Delta^{1+q}}.
$$
\end{lemma}

\begin{proof}
Recalling that $U_{t}^{i,N}=Z_{t}^{i,N}-\mathbb{E}_{\theta}[Z_{t}^{i,N}]$, we see that
\begin{align*}
D_{\Delta,t}^{N,K,4}=\Big|\frac{2N}{t}\sum_{a=t/(2\Delta)+1}^{t/\Delta}\mathrm{Var}(\bar{U}_{2a\Delta}^{N,K}-\bar{U}_{2(a-1)\Delta}^{N,K})-\frac{N}{t}\sum_{a=t/\Delta+1}^{2t/\Delta}\mathrm{Var}(\bar{U}_{a\Delta}^{N,K}-\bar{U}_{(a-1)\Delta}^{N,K})-\mathcal{W}_{\infty,\infty}^{N,K}\Big|.
\end{align*}
By Lemma \ref{UUW}, we have 
$$\mathrm{Var}_{\theta}(\bar{U}_{x+\Delta}^{N,K}-\bar{U}_{x}^{N,K})=\frac{\Delta}{N}\mathcal{W}_{\infty,\infty}^{N,K}-\mathcal{Y}^{N,K}+r_{N,K}(x,\Delta).$$
Since $a\in \{t/(2\Delta)+1,...,t/\Delta\}$, $x=2(a-1)\Delta\ge t$ satisfies $2\Delta\le \frac{x}{2}$ and for $a\in \{t/\Delta+1,...,2t/\Delta\}$, $x=(a-1)\Delta\ge t$ satisfies $\Delta\le x/2$. Then we conclude that
\begin{align*}
D_{\Delta,t}^{N,K,4}=&\Big|\frac{2N}{t}\sum_{a=t/(2\Delta)+1}^{t/\Delta}\Big[\frac{2\Delta}{N}\mathcal{W}_{\infty,\infty}^{N,K}-\mathcal{Y}^{N,K}+r_{N,K}(2(a-1)\Delta,2\Delta)\Big]\\&-\frac{N}{t}\sum_{a=\frac{t}{\Delta}+1}^{2t/\Delta}\Big[\frac{\Delta}{N}\mathcal{W}_{\infty,\infty}^{N,K}-\mathcal{Y}^{N,K}+r_{N,K}((a-1)\Delta,\Delta)\Big]-\mathcal{W}_{\infty,\infty}^{N,K}\Big|\\
=&\Big|\frac{2N}{t} \sum_{a=t/(2\Delta)+1}^{t/\Delta}r_{N,K}(2(a-1)\Delta,2\Delta)-\frac{N}{t}\sum_{a=t/\Delta+1}^{2t/\Delta}r_{N,K}((a-1)\Delta,\Delta)\Big|.
\end{align*}
But $|r_{N,K}(x,\Delta)|\le Cx\Delta^{-q}K^{-1}$, whence finally
$$
D_{\Delta,t}^{N,K,4}\le C\frac{N}{t}\frac{t}{\Delta}\Big(\frac t{\Delta^{q}K}\Big)=\frac{CN t}{K\Delta^{1+q}}
$$
as desired.
\end{proof}

To treat the last term $D_{\Delta,t}^{N,K,3}$, we need this following Lemma.

\begin{lemma}\label{Udelta}
Assume $H(q)$ for some $q\ge 1$. On the set $\Omega_{N,K}$,  for  all $t,x,\Delta\ge 1$, we have
\begin{equation}\label{tp1}
\mathrm{Var}\Big[\Big(\bar{U}_{x+\Delta}^{N,K}-\bar{U}_{x}^{N,K}\Big)^{2}\Big]\le C\Big(\frac{\Delta^{2}}{K^{2}}+\frac{t^{2}}{K^{2}\Delta^{4q}}\Big) \quad \hbox{if} \quad \frac{t}{2}\le x-\Delta\le x+\Delta\le 2t
\end{equation}
and
\begin{align}
\label{covU}
&\mathrm{Cov}_{\theta}\Big((\bar{U}_{x+\Delta}^{N,K}-\bar{U}_{x}^{N,K})^{2},(\bar{U}_{y+\Delta}^{N,K}-\bar{U}_{y}^{N,K})^{2}\Big)\le C\Big(\frac{\sqrt{t}}{K\Delta^{q-1}}+\frac{t^{2}}{K^{2}\Delta^{4q}}+\frac{\sqrt{t}}{K^{2}\Delta^{q-\frac{3}{2}}}\Big) \\
&\hskip6cm \hbox{if} \quad\frac{t}{2}\le y-\Delta\le y+\Delta\le x-2\Delta\le x+\Delta\le 2t. \notag
\end{align}
\end{lemma}

\begin{proof}
Step 1: recalling (\ref{ee2}), for $z \in [x,x+\Delta]$, we write 
$$
U_{z}^{i,N}-U_{x}^{i,N}=\sum_{n\ge 0}\int_{0}^{z}\beta_{n}(x,z,r)\sum_{j=1}^{N}A^{n}_{N}(i,j)M_{r}^{j,N}dr=\Gamma_{x,z}^{i,N}+X_{x,z}^{i,N},
$$
where $\beta_n(x,z,r)=\phi^{\star n}(z-r) - \phi^{\star n}(x-r)$ and where
\begin{align*}
\Gamma_{x,z}^{i,N}&=\sum_{n\ge 0}\int_{x-\Delta}^{z}\beta_{n}(x,z,r)\sum_{j=1}^{N}A_{N}^{n}(i,j)(M_{r}^{j,N}-M_{x-\Delta}^{j,N})dr,\\
X_{x,z}^{i,N}&=\sum_{n\ge 0}\Big(\int_{x-\Delta}^{z}\beta_{n}(x,z,r)dr\Big)\sum_{j=1}^{N}A_{N}^{n}(i,j)M_{x-\Delta}^{j,N}+\sum_{n\ge 0}\int_{0}^{x-\Delta}\beta_{n}(x,z,r)\sum_{j=1}^{N}A_{N}^n(i,j)M_{r}^{j,N}dr.
\end{align*}
We set $\bar{\Gamma}_{x,z}^{N,K}=K^{-1}\sum_{i=1}^{K}\Gamma_{x,z}^{i,N}$ and 
$\bar{X}_{x,z}^{N,K}=K^{-1}\sum_{i=1}^{K}X_{x,z}^{i,N}$.
We write
$$
\bar{X}_{x,z}^{N,K}=\sum_{n\ge 0}\Big(\int_{x-\Delta}^{z}\beta_{n}(x,z,r)dr\Big)O_{x-\Delta}^{N,K,n}+\sum_{n\ge 0}\int_{0}^{x-\Delta}\beta_{n}(x,z,r)O_{r}^{N,K,n}dr.
$$
where
$$
O_{r}^{N,K,n}=\frac{1}{K}\sum_{i=1}^{K}\sum_{j=1}^{N}A_{N}^{n}(i,j)M_{r}^{j,N}.
$$
By (\ref{ee3}), we have $[M^{i,N},M^{j,N}]_{t}=\boldsymbol{1}_{\{i=j\}}Z_{t}^{i,N}$. Hence, 
for $n\ge 1$,
$$
[O^{N,K,n},O^{N,K,n}]_{r}=\frac{1}{K^{2}}\sum_{j=1}^{N}\Big(\sum_{i=1}^{K}A^{n}_{N}(i,j)\Big)^{2}Z_{r}^{j,N}
\le \frac{N}{K^{2}}|||I_{K}A_{N}|||^{2}_{1}|||A_{N}|||^{2n-2}_{1}\bar{Z}_{r}^{N}.
$$
And when $n=0$, we have 
$$
[O^{N,K,0},O^{N,K,0}]_{r}=\frac{1}{K^{2}}\sum_{j=1}^{N}\Big(\sum_{i=1}^{K}A^{0}_{N}(i,j)\Big)^{2}Z_{r}^{j,N}=\frac{1}{K}\bar{Z}_{r}^{N,K}.
$$
By Lemma \ref{eU}, we have, on $\Omega_{N,K},$
$$
\mathbb{E}_\theta[(\bar{Z}_{t}^{N,K})^2]\le 2\mathbb{E}_\theta[\bar{Z}_{t}^{N,K}]^2+2\mathbb{E}_\theta[(\bar{U}_{t}^{N,K})^2]\le Ct^2.
$$
Hence, by the Doob's inequality,
when $n\ge 1$:
\begin{align}
\label{O1}\mathbb{E}_{\theta}\Big[\sup_{[0,2t]}\Big(O_{r}^{N,K,n}\Big)^{4}\Big]\le \frac{CN^{2}}{K^{4}}|||I_{K}A_{N}|||_{1}^{4}|||A_{N}|||_{1}^{4n-4}\mathbb{E}_{\theta}\Big[\Big(\bar{Z}_{2t}^{N}\Big)^{2}\Big] \le\frac{C}{N^{2}}|||A_{N}|||_{1}^{4n-4}t^{2}.
\end{align}
By the same way, 
\begin{align}
\label{O0}\mathbb{E}_{\theta}\Big[\sup_{[x-\Delta,x+\Delta]}\Big(O_{r}^{N,K,n}-O_{x-\Delta}^{N,K,n}\Big)^{4}\Big]
\le\frac{C}{N^{2}}|||A_{N}|||_{1}^{4n-4}\Delta^{2},
\end{align}
and in the case $n=0$, by Doob's inequality, 
\begin{align}
\label{O2}\mathbb{E}_{\theta}\Big[\sup_{[x-\Delta,x+\Delta]}\Big(O_{r}^{N,K,0}-O_{x-\Delta}^{N,K,0}\Big)^{4}\Big]\le CK^{-2}\Delta^{2}.
\end{align}

Step 2: We recall the result of \cite[Lemma 15]{A}:
$$
\Big|\int_{x-\Delta}^{z}\beta_{n}(x,z,r)dr\Big|+\int_{0}^{x-\Delta}\Big|\beta_{n}(x,z,r)\Big|dr\le Cn^{q}\Lambda^{n}\Delta^{-q}.$$
So we  conclude that
$$
|\bar{X}_{x,z}^{N}|\le C\sum_{n\ge 0}n^{q}\Lambda^{n}\Delta^{-q}\sup_{[0,2t]}|O_{r}^{N,K,n}|= C\sum_{n\ge 1}
n^{q}\Lambda^{n}\Delta^{-q}\sup_{[0,2t]}|O_{r}^{N,K,n}|.
$$
Recalling (\ref{O1}),  on the set $\Omega_{N,K}$, by using the Minkowski inequality we conclude that
\begin{align*}
\mathbb{E}[(\bar{X}_{x,z}^{N})^{4}]^{\frac{1}{4}}\le C\sum_{n\ge 1}n^{q}\Lambda^{n}|||A_N|||_1^{n-1}\Delta^{-q}N^{-\frac{1}{2}}\sqrt t
\le C\Delta^{-q}N^{-\frac{1}{2}}\sqrt t.
\end{align*}

Step 3: We rewrite
$$\bar\Gamma_{x,z}^{N,K}=\sum_{n\ge 0}\int_{x-\Delta}^{z}\beta_{n}(x,z,r)[O_{r}^{N,K,n}-O_{x-\Delta}^{N,K,n}]dr.$$
Since $\int_{x-\Delta}^{x}|\beta_{n}(x,z,r)|dr\le 2\Lambda^{n}$  by \cite[Lemma 15]{A}, using 
(\ref{O0})-(\ref{O2}) and the Minkowski inequality,
$$\mathbb{E}[(\bar{\Gamma}_{x,z}^{N,K})^{4}]^{\frac{1}{4}}\le C\Big\{\Delta^{\frac{1}{2}}K^{-\frac{1}{2}}+\sum_{n\ge 1}\Lambda^{n}\frac{1}{\sqrt{N}}|||A_{N}|||^{n-1}_{1}\Delta^{\frac{1}{2}}\Big\}\le C\Delta^{\frac{1}{2}}(K^{-\frac{1}{2}}+N^{-\frac{1}{2}})
\le C\Delta^{\frac{1}{2}}K^{-\frac{1}{2}}.$$

Step 4: Since, see Step 1,
$$\Big(\bar{U}^{N,K}_{x+\Delta}-\bar{U}_{x}^{N,K}\Big)^{4}=\Big(\bar{\Gamma}_{x,x+\Delta}^{N,K}+\bar{X}_{x,x+\Delta}^{N,K}\Big)^{4}\le 8\Big[(\bar{\Gamma}_{x,x+\Delta}^{N,K})^{4}+(\bar{X}_{x,x+\Delta}^{N,K})^{4}\Big],$$
we deduce from Steps 2 and 3 that \eqref{tp1} holds true.

\vip

Step 5: 
The aim of this step is to show that, for $x,y,\Delta$ as in the statement, it holds true that
$$\mathrm{Cov}_{\theta}\Big((\bar{U}_{x+\Delta}^{N,K}-\bar{U}_{x}^{N,K})^{2},(\bar{U}_{y+\Delta}^{N,K}-\bar{U}_{y}^{N,K})^{2}\Big)\le |\mathrm{Cov}_{\theta}[(\bar{\Gamma}^{N,K}_{x,x+\Delta})^{2},(\bar{\Gamma}^{N,K}_{y,y+\Delta})^{2}]|+\frac{C}{K^{2}}\Big(\frac{t^{2}}{\Delta^{4q}}+\frac{\sqrt t}{\Delta^{q-\frac{3}{2}}}\Big).$$
We write
$$(\bar{U}_{x+\Delta}^{N,K}-\bar{U}_{x}^{N,K})^{2}=(\bar{\Gamma}_{x,x+\Delta}^{N,K})^{2}+(\bar{X}_{x,x+\Delta}^{N,K})^{2}+2\bar{\Gamma}_{x,x+\Delta}^{N,K}\bar{X}_{x,x+\Delta}^{N,K},$$
and the same formula for $y$.
Then we use the bilinearity of the covariance. We have the term 
$\mathrm{Cov}_{\theta}[(\bar{\Gamma}_{x,x+\Delta}^{N,K})^{2},(\bar{\Gamma}_{y,y+\Delta}^{N,K})^{2}],$
and it remains to verify that
\begin{align*}
R:=\mathbb{E}_{\theta}\Big[&(\bar{\Gamma}_{x,x+\Delta}^{N,K})^{2}(\bar{X}_{y,y+\Delta}^{N,K})^{2}+2(\bar{\Gamma}_{x,x+\Delta}^{N,K})^{2}|\bar{\Gamma}_{y,y+\Delta}^{N,K}\bar{X}_{y,y+\Delta}^{N,K}|+(\bar{X}_{x,x+\Delta}^{N,K})^{2}(\bar{\Gamma}_{y,y+\Delta}^{N,K})^{2}
\\&+(\bar{X}_{x,x+\Delta}^{N,K})^{2}(\bar{X}_{y,y+\Delta}^{N,K})^{2}+2(\bar{X}_{x,x+\Delta}^{N,K})^{2}|\bar{\Gamma}_{y,y+\Delta}^{N,K}\bar{X}_{y,y+\Delta}^{N,K}|+2|\bar{\Gamma}_{x,x+\Delta}^{N,K}\bar{X}_{x,x+\Delta}^{N,K}|(\bar{\Gamma}_{y,y+\Delta}^{N,K})^{2}
\\&+2|\bar{X}_{x,x+\Delta}^{N,K}\bar{\Gamma}_{x,x+\Delta}^{N,K}|(\bar{X}_{y,y+\Delta}^{N,K})^{2}+4|\bar{\Gamma}_{x,x+\Delta}^{N,K}\bar{X}_{x,x+\Delta}^{N,K}\bar{\Gamma}_{y,y+\Delta}^{N,K}\bar{X}_{y,y+\Delta}^{N,K}|\Big]
\end{align*}
is bounded by $\frac{C}{K^{2}}(\frac{t^{2}}{\Delta^{4q}}+\frac{\sqrt t}{\Delta^{q-\frac{3}{2}}})$.

\vip

By Steps 2 and 3, we know that $\mathbb{E}[(\bar{\Gamma}_{x,z}^{N,K})^{4}]\le C\Delta^{2}K^{-2}$
and $\mathbb{E}[(\bar{X}_{x,z}^{N,K})^{4}]\le Ct^{2}\Delta^{-4q}N^{-2}$, and the sames inequalities hold true 
with $y$ instead of $x$. Using furthermore the  H\"older inequality, one may verify that,
setting $a=C\Delta^{2}K^{-2}$ and $b=Ct^{2}\Delta^{-4q}N^{-2}$, we have
$$
R\leq \sqrt{ab} + 2 a^{3/4}b^{1/4}+ \sqrt{ab}+b+2 a^{1/4}b^{3/4}+ 2 a^{3/4}b^{1/4}+2 a^{1/4}b^{3/4}+4\sqrt{ab},
$$
which is easily bounded by $C(b+b^{1/4}a^{3/4})$, from which the conclusion follows.

\vip

Step 6: Here we want to verify that
$$
\mathcal{I}:= |\mathrm{Cov}_{\theta}[(\bar{\Gamma}^{N,K}_{x,x+\Delta})^{2},(\bar{\Gamma}^{N,K}_{y,y+\Delta})^{2}]|
\le  \frac{C\sqrt{t}}{K\Delta^{q-1}}.
$$
We recall from \cite[Lemma 30, Step 6]{A} that
for any $r$, $s$ in $[x-\Delta,x+\Delta]$, any $u$, $v$ in $[y-\Delta,y+\Delta]$, 
any  $j$, $l$, $\delta$, $\varepsilon$ in $\{1,\ ...N\},$
$$
\Big|\mathrm{Cov}_{\theta}\Big[(M_{r}^{j,N}-M_{x-\Delta}^{j,N})(M_{s}^{l,N}-M_{x-\Delta}^{l,N}),(M_{u}^{\delta,N}-M_{y-\Delta}^{\delta,N})(M_{v}^{\varepsilon,N}-M_{y-\Delta}^{\varepsilon,N})\Big|\le C\boldsymbol{1}_{\{j=l\}}\sqrt t\Delta^{1-q}.
$$
We start from
$$\bar{\Gamma}_{x,x+\Delta}^{N,K}=\sum_{n\ge 0}\int_{x-\Delta}^{x+\Delta}\beta_{n}(x,x+\Delta,r)\frac{1}{K}\sum_{i=1}^{K}\sum_{j=1}^{N}A_N(i,j)(M_{r}^{j,N}-M_{x-\Delta}^{j,N})dr.$$
So
\begin{align*}
\mathcal{I}=&\sum_{m, n, a, b\ge 0}\int_{x-\Delta}^{x+\Delta}\int_{x-\Delta}^{x+\Delta}\int_{y-\Delta}^{y+\Delta}\int_{y-\Delta}^{y+\Delta}\beta_{m}(x,x+\Delta,r)\beta_{n}(x,x+\Delta,s)\\
&\qquad\beta_{a}(y,y+\Delta,u)\beta_{b}(y,y+\Delta,v)\frac{1}{K^{4}}\sum_{i,k,\alpha,\gamma=1}^{K}\sum_{j,l,\delta,\varepsilon=1}^{N}A^{m}_{N}(i,j)A_{N}^{n}(k,l)A_{N}^{a}(\alpha,\delta)A_{N}^{b}(\gamma,\varepsilon)\\
&\qquad\mathrm{Cov}_{\theta}\Big[(M_{r}^{j,N}-M_{x-\Delta}^{j,N})(M_{s}^{l,N}-M_{x-\Delta}^{l,N}),(M_{u}^{\delta,N}-M_{y-\Delta}^{\delta,N})(M_{v}^{\varepsilon,N}-M_{y-\Delta}^{\varepsilon,N})\Big]dvdudsdr\\
\leq &C \sqrt t \Delta^{1-q} \sum_{m, n, a, b\ge 0} \Lambda^{m+n+a+b} 
\frac{1}{K^{4}}\sum_{i,k,\alpha,\gamma=1}^{K}\sum_{j,\delta,\varepsilon=1}^{N}A^{m}_{N}(i,j)A_{N}^{n}(k,j)A_{N}^{a}(\alpha,\delta)A_{N}^{b}(\gamma,\varepsilon).
\end{align*}
We used again the result \cite[Lemma 15]{A}:
$\int_{x-\Delta}^{x+\Delta}|\beta_{m}(x,x+\Delta,r)|dr\le 2\Lambda^{m}.$
And we observe one more time that $A_{N}^{0}(i,j)=\boldsymbol{1}_{\{i=j\}}$ and, 
when $m\ge 1$, $\sum_{i=1}^{K}A_{N}^{m}(i,j)\le |||I_{K}A_{N}|||_{1}|||A_{N}|||^{m-1}_{1}.$
We now treat separately the cases where $m,n,a,b$ vanish and find, on $\Omega_{N,K}$,
\begin{align*}
\mathcal{I}\le& \frac{C\sqrt t}{K^4\Delta^{q-1}} \sum_{m,n,a,b\ge 1}\sum^{N}_{j,\delta,\varepsilon=1} \Lambda^{m+n+a+b}|||I_KA_{N}|||^{4}_{1}|||A_{N}|||^{m+n+a+b-4}_{1}\\
+&\frac{4C\sqrt t}{K^4\Delta^{q-1}} \sum_{n,a,b\ge 1}\sum^{N}_{j,\delta,\varepsilon=1}\sum_{i=1}^K\indiq_{\{i=j\}} \Lambda^{n+a+b}|||I_KA_{N}|||^{3}_{1}|||A_{N}|||^{n+a+b-3}_{1}\\
+&\frac{2C\sqrt t}{K^4\Delta^{q-1}} \sum_{a,b\ge 1}\sum^{N}_{j,\delta,\varepsilon=1}\Big\{\sum_{i,k=1}^K\indiq_{\{i=k=j\}}+\sum_{\alpha,\gamma=1}^K\indiq_{\{\alpha=\delta,\varepsilon=\gamma\}}+\sum_{i,\alpha=1}^K\indiq_{\{i=j,\alpha=\delta\}}\Big\}\\
& \hskip9cm\times \Lambda^{a+b}|||I_KA_{N}|||^{2}_{1}|||A_{N}|||^{a+b-2}_{1}\\
+&\frac{2C\sqrt t}{K^4\Delta^{q-1}} \sum_{a\ge 1}\sum^{N}_{j,\delta,\varepsilon=1}\Big\{2\sum_{i,k,\alpha=1}^K\indiq_{\{i=k=j,\alpha=\delta\}}+2\sum_{i,\alpha,\gamma=1}^K\indiq_{\{i=j,\alpha=\delta,\varepsilon=\gamma\}}\Big\}\Lambda^{a}|||I_KA_{N}|||_{1}|||A_{N}|||^{a-1}_{1}\\
+&\frac{2C\sqrt t}{K^4\Delta^{q-1}}\sum_{i,k,\alpha,\gamma=1}^{K}\sum_{j,\delta,\varepsilon=1}^{N}\indiq_{\{i=k=j,\alpha=\delta,\varepsilon=\gamma\}}\\
\le& \frac{2C\sqrt t}{K^4\Delta^{q-1}} \Big(\frac{K^4}{N^4}N^3 + \frac{K^3}{N^3}N^2 K + 
\frac{K^2}{N^2}(K N^2+K^2 N +K^2 N)+  \frac K N (K^2 N+K^3) + K^3  \Big)\\
\le &\frac{C\sqrt{t}\Delta^{1-q}}{K}.
\end{align*}

Step 7: We conclude from Steps 5 and 6 that on the set $\Omega_{N,K}$,
\begin{align*}
\Big|\mathrm{Cov}_{\theta}\Big[(\bar{U}_{x+\Delta}^{N,K}-\bar{U}_{x}^{N,K})^{2},
(\bar{U}^{N,K}_{y+\Delta}-\bar{U}_{y}^{N,K})^{2}\Big]\Big|
\le C\Big[\frac{\sqrt{t}}{K\Delta^{q-1}}+\frac{t^2}{K^{2}\Delta^{4q}}+\frac{\sqrt{t}}
{K^{2}\Delta^{q-3/2}}\Big],
\end{align*}
which proves \eqref{covU}.
\end{proof}

We can now study $D_{\Delta,t}^{N,K,3}$.

\begin{lemma}\label{dk3}
Assume $H(q)$ for some $q\ge 1$. On the set $\Omega_{N,K}$, for all $1\le \Delta\le \frac{t}{2}$,
$$\mathbb{E}_{\theta}[(D_{\Delta,t}^{N,K,3})^{2}]\le C\Big(\frac{N^{2}}{K^{2}}\frac{\Delta}{t}+\frac{N^{2}}{K^{2}}\frac{t}{\Delta^{4q+1}}+\frac{N^{2}}{K}\frac{\sqrt t}{\Delta^{q+1}}+\frac{N^{2}}{K^{2}}\frac{t^{2}}{\Delta^{4q+2}}+\frac{N^{2}}{K^{2}}\frac{\sqrt t}{\Delta^{q+\frac{1}{2}}}\Big).$$
\end{lemma}

\begin{proof}
Recall that by definition
\begin{align*}
D_{\Delta,t}^{N,K,3}=&\frac{N}{t}\Big|\sum_{a=t/\Delta+1}^{2t/\Delta}\Big(\bar{Z}_{a\Delta}^{N,K}-\bar{Z}_{(a-1)\Delta}^{N,K}-\mathbb{E}_{\theta}[\bar{Z}_{a\Delta}^{N,K}-\bar{Z}_{(a-1)\Delta}^{N,K}]\Big)^{2}\\
&-\mathbb{E}_{\theta}\Big[\sum_{a=t/\Delta+1}^{2t/\Delta}\Big(\bar{Z}_{a\Delta}^{N,K}-\bar{Z}_{(a-1)\Delta}^{N,K}-\mathbb{E}_{\theta}[\bar{Z}_{a\Delta}^{N,K}-\bar{Z}_{(a-1)\Delta}^{N,K}]\Big)^{2}\Big]\Big|.
\end{align*}
Since now $\bar{U}_{r}^{N,K}=\bar{Z}_{r}^{N,K}-\mathbb{E}_{\theta}[\bar{Z}_{r}^{N,K}]$,
$$\mathbb{E}_{\theta}[(D_{\Delta,t}^{N,K,3})^{2}]=\frac{N^{2}}{t^{2}}\mathrm{Var}_{\theta}\Big(\sum_{a=t/\Delta+1}^{2t/\Delta}(\bar{U}_{a\Delta}^{N,K}-\bar{U}_{(a-1)\Delta}^{N,K})^{2}\Big)=\frac{N^{2}}{t^{2}}\sum_{a,b=t/\Delta+1}^{2t/\Delta}K_{a,b},$$
where $K_{a,b}=\mathrm{Cov}_{\theta}[(\bar{U}_{a\Delta}^{N,K}-\bar{U}_{(a-1)\Delta}^{N,K})^{2},(\bar{U}_{b\Delta}^{N,K}-\bar{U}_{(b-1)\Delta}^{N,K})^{2}].$ By Lemma \ref{Udelta}, for $|a-b|\le 2$,   
$$|K_{a,b}|\le \Big\{\mathrm{Var}_{\theta}[(\bar{U}^{N,K}_{a\Delta}-\bar{U}^{N,K}_{(a-1)\Delta})^{2}]\mathrm{Var}[(\bar{U}^{N,K}_{b\Delta}-\bar{U}^{N,K}_{(b-1)\Delta})^{2}]\Big\}^{\frac{1}{2}}\le C\Big(\frac{\Delta^{2}}{K^{2}}+\frac{t^{2}}{K^{2}\Delta^{4q}} \Big).$$
If now $|a-b|\ge 3$,  we set $x=(a-1)\Delta$, $y=(b-1)\Delta$ in (\ref{covU}) and get
$$|K_{a,b}|\le C\Big(\frac{\sqrt t}{K\Delta^{q-1}}+\frac{t^{2}}{K^{2}\Delta^{4q}}+\frac{\sqrt t}{K^{2}\Delta^{q-\frac{3}{2}}}\Big).$$
Finally we conclude that 
\begin{align*}
\mathbb{E}_{\theta}[(D_{\Delta,t}^{N,K,3})^{2}]\le& C\frac{N^{2}}{t^{2}}\frac{t}{\Delta}\Big(\frac{\Delta^{2}}{K^{2}}+\frac{t^{2}}{K^{2}\Delta^{4q}}\Big)+C\frac{N^{2}}{t^{2}}\frac{t^{2}}{\Delta^{2}}\Big(\frac{\sqrt t}{K\Delta^{q-1}}+\frac{t^{2}}{K^{2}\Delta^{4q}}+\frac{\sqrt t}{K^{2}\Delta^{q-\frac{3}{2}}}\Big)\\
\le& C\Big(\frac{N^{2}}{K^{2}}\frac{\Delta}{t}+\frac{N^{2}}{K^{2}}\frac{t}{\Delta^{4q+1}}+\frac{N^{2}}{K}\frac{\sqrt t}{\Delta^{q+1}}+\frac{N^{2}}{K^{2}}\frac{t^{2}}{\Delta^{4q+2}}+\frac{N^{2}}{K^{2}}\frac{\sqrt t}{\Delta^{q+\frac{1}{2}}}\Big).
\end{align*}
which completes the proof.
\end{proof}

\begin{lemma}\label{WWWdelta}
Under the assumption $H(q)$ for some $q\ge 3$ and the the set $\Omega_{N,K}$, we  have:
$$\mathbb{E}_{\theta}\Big[\Big|\mathcal{W}_{\Delta,t}^{N,K}-\mathcal{W}_{\infty,\infty}^{N,K}\Big|\Big]\le C\Big(\frac{N}{K}\sqrt{\frac{\Delta}{t}}+\frac{N^{2}}{K\Delta^{\frac{1}{2}(q+1)}}+\frac{Nt}{K\Delta^{\frac{q}{2}+1}}\Big).$$
\end{lemma}
\begin{proof}
We summarize all the above Lemmas and conclude that, on $\Omega_{N,K}$,
\begin{align*}
&\mathbb{E}_{\theta}\Big[\Big|\mathcal{W}_{\Delta,t}^{N,K}-\mathcal{W}_{\infty,\infty}^{N,K}\Big|\Big]\\
\le& \mathbb{E}_{\theta}\Big[D_{\Delta,t}^{N,K,1}+2D_{2\Delta,t}^{N,K,1}+D_{\Delta,t}^{N,K,2}+2D_{2\Delta,t}^{N,K,2}+D_{\Delta,t}^{N,K,3}+2D_{2\Delta,t}^{N,K,3}+D_{\Delta,t}^{N,K,4}\Big]
\\\le& C\Big(\frac{N}{K}\frac{\Delta}{t}+\frac{N\Delta}{t^{2q}}+\frac{N}{t^{q-1}}+\frac{Nt}{K\Delta^{1+q}}\Big)\\
&+C\sqrt{\frac{N^{2}}{K^{2}}\frac{\Delta}{t}+\frac{N^{2}}{K^{2}}\frac{t}{\Delta^{4q+1}}+\frac{N^{2}}{K}\frac{\sqrt t}{\Delta^{q+1}}+\frac{N^{2}}{K^{2}}\frac{t^{2}}{\Delta^{4q+2}}+\frac{N^{2}}{K^{2}}\frac{\sqrt t}{\Delta^{q+\frac{1}{2}}}}.
\end{align*}
Since $1\leq \Delta\le  t$ and $q\geq 3$, we conclude, after some tedious but direct computations, 
that
\begin{align*}
\mathbb{E}_{\theta}\Big[\Big|\mathcal{W}_{\Delta,t}^{N,K}-\mathcal{W}_{\infty,\infty}^{N,K}\Big|\Big]
\le C\Big(\frac{N}{K}\sqrt{\frac{\Delta}{t}}+\frac{N^{2}}{K\Delta^{\frac{1}{2}(q+1)}}
+\frac{Nt}{K\Delta^{\frac{q}{2}+1}}\Big).
\end{align*}
The most difficult terms are 
$$\sqrt{\frac{N^2t^{1/2}}{K\Delta^{q+1}}}= \sqrt{\frac{N^2}{K\Delta^{(q+1)/2}}}
\sqrt{\frac{t^{1/2}}{\Delta^{(q+1)/2}}} 
\leq \frac{N^2}{K\Delta^{(q+1)/2}}+\frac{t^{1/2}}{\Delta^{(q+1)/2}}
\leq \frac{N^2}{K\Delta^{(q+1)/2}}+\frac{Nt}{K \Delta^{q/2+1}}
$$
and 
$$
\sqrt{\frac{N^{2}}{K^{2}}\frac{\sqrt t}{\Delta^{q+\frac{1}{2}}}}\leq \frac{N}{K}\Big(\sqrt{\frac{\Delta}t}+
\frac{t}{\Delta^{q+1}}\Big)\leq \frac{N}{K}\Big(\sqrt{\frac{\Delta}t}+
\frac{t}{\Delta^{q/2+1}}\Big).
$$
The proof is complete.
\end{proof}

Next we  prove the main result of this section.
\begin{proof}[Proof of Theorem \ref{noname}] We start from
\begin{align*}
&\mathbb{E}\Big[\boldsymbol{1}_{\Omega_{N,K}}\Big|\mathcal{X}_{\Delta,t}^{N,K}-\frac{\mu}{(1-\Lambda p)^{3}}\Big|\Big]\\
\le& \mathbb{E}\Big[\boldsymbol{1}_{\Omega_{N,K}}\Big|\mathcal{W}_{\Delta,t}^{N,K}-\mathcal{W}_{\infty,\infty}^{N,K}\Big|\Big]+\frac{N}{K}\mathbb{E}\Big[\boldsymbol{1}_{\Omega_{N,K}}\Big|\varepsilon_t^{N,K}-\mu\bar{\ell}_N^K\Big|\Big]+\mathbb{E}\Big[\boldsymbol{1}_{\Omega_{N,K}}\Big|\mathcal{X}_{\infty,\infty}^{N,K}-\frac{\mu}{(1-\Lambda p)^{3}}\Big|\Big]\\
\le & C\Big(\frac{N}{K}\sqrt{\frac{\Delta}{t}}+\frac{N^{2}}{K\Delta^{\frac{1}{2}(q+1)}}+\frac{Nt}{K\Delta^{\frac{q}{2}+1}}+\frac{1}{\sqrt{K}}\Big)+  C \frac{N}{Kt^q}+ C\frac{N}{K\sqrt{Kt}}+\frac{C}{K}
\end{align*}
by Lemmas \ref{WWWdelta}, \ref{emain} and \ref{WWW}.
Since $t\geq \Delta\geq 1$, we have $\frac{N}{Kt^q}\le \frac{Nt}{K\Delta^{\frac{q}{2}+1}}$ and we conclude that
\begin{align*}
\mathbb{E}\Big[\boldsymbol{1}_{\Omega_{N,K}}\Big|\mathcal{X}_{\Delta,t}^{N,K}-\frac{\mu}{(1-\Lambda p)^{3}}\Big|\Big]
\le C\Big(\frac{N}{K}\sqrt{\frac{\Delta}{t}}+\frac{N^{2}}{K\Delta^{\frac{1}{2}(q+1)}}+\frac{Nt}{K\Delta^{\frac{q}{2}+1}}+\frac{1}{K}+\frac{N}{K\sqrt{Kt}}\Big),
\end{align*}
which was our goal.
\end{proof}

Next, we write down the probability estimate.

\begin{cor}\label{pwnk}
Assume $H(q)$ for some $q\ge 1$. We have
\begin{align*}
&P\Big(\Big|\mathcal{X}_{\Delta,t}^{N,K}-\frac{\mu}{(1-\Lambda p)^{3}}\Big|\ge \varepsilon\Big)\\
&\le \frac{C}{\varepsilon}\Big(\frac{N}{K}\sqrt{\frac{\Delta}{t}}+\frac{N^{2}}{K\Delta^{\frac{1}{2}(q+1)}}+\frac{Nt}{K\Delta^{\frac{q}{2}+1}}+\frac{1}{K}+\frac{N}{K\sqrt{Kt}}\Big)+CNe^{-C'K}.
\end{align*}
Under $H(q)$ for some $q> 3$ and with the choice $\Delta_{t}\sim t^{\frac{4}{(q+1)}}$, this gives
$$
P\Big(\Big|\mathcal{X}_{\Delta_{t},t}^{N,K}-\frac{\mu}{(1-\Lambda p)^{3}}\Big|\ge \varepsilon\Big)
\le \frac{C}{\varepsilon}\Big(\frac{1}{K}+\frac{N}{K\sqrt{t^{1-\frac{4}{1+q}}}}+\frac{N^2}{Kt^2}\Big)
+CNe^{-C'K}.
$$
\end{cor}

\begin{proof}
The first assertion immediately follows from Theorem \ref{noname} and Lemma \ref{ONK}.
The second assertion is not difficult.
\end{proof}

\section{The final result in the subcritical case.}

We summarize the rates we obtained for the three estimators:
by Theorem \ref{lmain} and Corollaries \ref{vmain} and \ref{pwnk}, we have, under $H(q)$ for some $q>3$,
for all $\epsilon\in(0,1)$, all $t\geq 1$, all $N\geq K\geq 1$,
\begin{gather*}
 P\Big( \Big|\varepsilon ^{N,K}_{t}-\frac{\mu}{1-\Lambda   p}\Big|\ge\varepsilon\Big)\le CNe^{-C^{'}K}+\frac{C}{\varepsilon}\Big(\frac{1}{\sqrt{NK}}+\frac{1}{\sqrt{Kt}}+\frac{1}{t^{q}}\Big),\\
P\Big(\Big|\mathcal{V}_{t}^{N,K}-\frac{\mu^{2}\Lambda^{2}p(1-p)}{(1-\Lambda p)^{2}}\Big|\ge\varepsilon\Big)\le CNe^{-C'K}+\frac{C}{\varepsilon}\Big(\frac{1}{\sqrt{K}}+\frac{N}{t\sqrt{K}}\Big),\\
P\Big(\Big|\mathcal{X}_{\Delta_{t},t}^{N,K}-\frac{\mu}{(1-\Lambda p)^{3}}\Big|\ge \varepsilon\Big)
\le \frac{C}{\varepsilon}\Big(\frac{1}{K}+\frac{N}{K\sqrt{t^{1-\frac{4}{1+q}}}}+\frac{N^2}{Kt^2}\Big)+CNe^{-C'K}.
\end{gather*}

\begin{proof}[Proof of Theorem \ref{abcd}]
One easily verifies that $\Psi$ is $C^{\infty}$ in the domain $D$, that 
$$
(u,v,w)=\Big(\frac{\mu}{1-\Lambda p},\frac{\mu^{2}\Lambda^{2}p(1-p)}{(1-\Lambda p)^{2}}
,\frac{\mu}{(1-\Lambda p)^3}\Big) \in D
$$ 
and that $\Psi(u,v,w)=(\mu,\Lambda,p)$.
Hence there is a constant $c$ such that for any $N\ge 1$, $t\ge 1$, any $\varepsilon\in (0,1/c)$,
\begin{align*}
&P\Big(\Big|\Psi(\varepsilon_{t}^{N,K},\mathcal{V}_{t}^{N,K},\mathcal{X}_{\Delta_{t},t}^{N,K})-(\mu,\Lambda,p)\Big|\ge \varepsilon\Big)\\
\le& P\Big(\Big|\varepsilon_{t}^{N,K}-u\Big|+\Big|\mathcal{V}_{t}^{N,K}-v\Big|+\Big|\mathcal{X}_{\Delta_{t},t}^{N,K}-w\Big|\ge c\varepsilon\Big)
\\ \le& \frac{C}{\varepsilon}\Big(\frac{1}{\sqrt{K}}+\frac{N}{K\sqrt{t^{1-\frac{4}{1+q}}}}+\frac{N}{t\sqrt{K}}\Big)+CNe^{-C'K},
\end{align*}
which completes the proof.
\end{proof}

\section{Analysis of a random matrix for the supercritical case}

We define the matrix $A_N$ by $A_{N}(i,j):=N^{-1}\theta_{ij}$, $i,j\in\{1,...,N\}$.
We assume here that $\ p\in (0,1]$ and we introduce the events:
\begin{align*}
\Omega_{N}^{2}:=&\Big\{\frac{1}{N}\sum_{i=1}^{N}\sum_{j=1}^{N}A_{N}(i,j)>\frac{p}{2} \quad\hbox{and}
\quad |NA^{2}_{N}(i,j)-p^{2}|<\frac{p^{2}}{2N^{3/8}} \quad 
 \text{for all  $i,j=1,...,N$ \Big\}},\\
\Omega_{N}^{K,2}:=&\Big\{\frac{1}{K}\sum_{i=1}^{K}\sum_{j=1}^{N}A_{N}(i,j)>\frac{p}{2}\Big\} \cap \Omega_{N}^{2}.
\end{align*}
\begin{lemma}\label{ONK2} One has
$$P(\Omega_{N}^{K,2})\ge 1-Ce^{-cN^{\frac{1}{4}}}.$$
\end{lemma}

\begin{proof}
By   \cite[lemma 33]{A}, we already have  $P(\Omega_{N}^{2})\ge 1-Ce^{-cN^{\frac{1}{4}}}$. 
We recall the Hoeffding inequality for the Binomial$(n,q)$ random variables. For all $x\ge 0$ and $X$ is a Binomial$(n,q)$ distributed, we have:

$$
P\Big(|X-nq|\ge x\Big)\le 2\exp(-2x^{2}/n).
$$
Since $N\sum_{i=1}^{K}\sum_{j=1}^{N}A_{N}(i,j)=\sum_{i=1}^{K}\sum_{j=1}^{N}\theta_{ij}$ is  Binomial$(NK,p)$ distributed, 
$$P\Big(K^{-1}\sum_{i=1}^{K}\sum_{j=1}^{N}A_{N}(i,j)\le \frac{p}{2}\Big)
\le P\Big(\Big|N\sum_{i=1}^{K}\sum_{j=1}^{N}A_{N}(i,j)-NKp\Big|\ge \frac{NKp}{2}\Big)\le 2\exp\Big(-\frac{NKp^{2}}{2}\Big).$$
So we have
$$
P(\Omega_{N}^{K,2})\ge 1-2\exp\Big(-\frac{NKp^{2}}{2}\Big)-Ce^{-cN^{\frac{1}{4}}} \ge 1-Ce^{-cN^{\frac{1}{4}}}.
$$ 
\end{proof}

Next we apply the Perron-Frobenius theorem and recall some lemma in [1].

\begin{lemma}
On the event $\Omega_{N}^{K,2}$, the spectral radius $\rho_{N}$ of $A_{N}$ is a simple eigenvalue of $A_{N}$ and $\rho_{N}\in [p(1-\frac{1}{2N^{\frac{3}{8}}}),p(1+\frac{1}{2N^{\frac{3}{8}}})].$ There is a row eigenvector $\cV_{N}\in \mathbb{R}_{+}^{N}$ of $A_{N}$ for the eigenvalue $\rho_{N}$ such that $||\cV_{N}||_{2}=\sqrt{N}.$ We also have $\cV_{N}(i)>0$ for all $i=1,...N$.
\end{lemma}

\begin{proof}
See \cite[lemma 34]{A}.
\end{proof}
We set  $\cV_{N}^{K}:=I_{K}\cV_{N}$ and let $(e_1,\dots,e_N)$ the canonical basis of $\mathbb{R}^N$.
Recall that $\boldsymbol{1}_N=\sum_{i=1}^N e_i$.

\begin{lemma}\label{123456} There exists $N_{0}\ge 1$ (depending only on $p$) such that for all $N\ge N_{0}$, on the set  $\Omega_{N}^{K,2}$, these properties hold true for all $i,j,k,l=1,...,N$:
\vip
$(i)$ for all $n\ge 2$, $A_{N}^{n}(i,j)\le (\frac{3}{2})A_{N}^{n}(k,l),$
\vip
$(ii)$ $V_{N}(i)\in [\frac{1}{2},2],$
\vip
$(iii)$ for all $n\ge 0$, $||A^{n}_{N}\boldsymbol{1}_{N}||_{2}\in [\sqrt{N}\frac{\rho_{N}^{n}}{2},2\sqrt{N}\rho_{N}^{n}]$,
\vip
$(iv)$ for all $n\ge 2, $ $A_{N}^{n}(i,j)\in [\rho_{N}^{n}/(3N),3\rho_{N}^{n}/N]$,
\vip

$(v)$ for all $n\ge 0$, all $r\in [1,\infty]$, $\bigl\|A_{N}^{n}\ce_{j}/||A_{N}^{n}\ce_{j}||_{r}-\cV_{N}/||\cV_{N}||_{r}\bigr\|_{r}\le 12(2N^{-\frac{3}{8}})^{\lfloor\frac{n}{2}\rfloor}$,
\vip
$(vi)$ for all $n\ge 1$, $||A_{N}^{n}e_{j}||_{2}\le 3\rho_{N}^{n}/(p\sqrt{N})$ and for all $n\ge 0$, $||A_{N}^{n}\boldsymbol{1}_{N}||_{\infty}\le 3\rho_{N}^{n}/p.$
\vip
$(vii)$ for all $n\ge 0$, all $r\in [1,\infty]$, $\bigl\| I_{K}A_{N}^{n}\boldsymbol{1}_{N}/\|I_{K}A_{N}^{n}\boldsymbol{1}_{N}\|_{r}-\cV^{K}_{N}/\|\cV^{K}_{N} \|_{r} \bigr\|_{r}\le 3(2N^{-\frac{3}{8}})^{\lfloor\frac{n}{2}\rfloor+1}$,
\vip
$(viii)$ for all $n\ge 0$, all $r\in [1,\infty]$, $\bigl\|I_KA_{N}^{n}\ce_{j}/||I_KA_{N}^{n}\ce_{j}||_{r}-\cV^K_{N}/||\cV^K_{N}||_{r}\bigr\|_{r}\le 12(2N^{-\frac{3}{8}})^{\lfloor\frac{n}{2}\rfloor}$,
\vip

$(ix)$ for all $n\ge 0$ $\|I_{K}A^{n}_{N}\boldsymbol{1}_{N}\|_{2}\in \Big[\sqrt{K}\rho^{n}_{N}/8,8\sqrt{K}\rho^{n}_{N}\Big] $.
\end{lemma}

\begin{proof}
The proof of $(i)$-$(vi)$ see \cite[Lemma 35]{A}. For the point $(vii)$, we set for $\boldsymbol{x},\boldsymbol{y}\in(0,\infty)^N$
$$
d_{K}(\boldsymbol{x},\boldsymbol{y})=\log\Big[\frac{\max_{i=1,...,K}(\frac{x_{i}}{y_{i}})}{\min_{i=1,...,K}(\frac{x_{i}}{y_{i}})}\Big].
$$ 
Clearly one has $d_{K}(I_{K}A_{N}^{n}\boldsymbol{1}_{N},I_{K}\cV_{N})\le d_{N}(A_{N}^{n}\boldsymbol{1}_{N},\cV_{N})$. Moreover from \cite[Step 3 of the proof of Lemma 35]{A} one has 
$d_{N}(A_{N}^{n}\boldsymbol{1}_{N},\cV_{N})\le (2N^{-3/8})^{\lfloor n/2\rfloor+1}$.
Therefore we can apply \cite[Lemma 39]{A} and we obtain that
$$
||I_{K}A_{N}^{n}\boldsymbol{1}_{N}/||I_{K}A_{N}^{n}\boldsymbol{1}_{N}||_{r}-\cV^{K}_{N}/||\cV^{K}_{N}||_{r}||_{r} \le 
3d_{K}(I_{K}A_{N}^{n}\boldsymbol{1}_{N},I_{K}\cV_{N})
\le3(2N^{-\frac{3}{8}})^{\lfloor\frac{n}{2}\rfloor+1}.
$$

Let us prove $(viii)$. The case $n\in\{0,1\}$ is straightforward. In \cite[Lemma 35 step 4]{A}, we already have for all $n\geq 2$, $d_N(A_N^ne_j,V_N)\leq 4(2N^{-3/8})^{\lfloor n/2\rfloor}$.
Therefore 
$$
||I_{K}A_{N}^{n}\ce_j/||I_{K}A_{N}^{n}\ce_J||_{r}-\cV^{K}_{N}/||\cV^{K}_{N}||_{r}||_{r} \le 
3d_{K}(I_{K}A_{N}^{n}\ce_j,I_{K}\cV_{N})
\le4(2N^{-\frac{3}{8}})^{\lfloor\frac{n}{2}\rfloor}.
$$
which finishes the proof of $(viii)$.

We now verify $(ix)$. We write $A_N^n \bun = ||A_N^n\bun||_2(\|\boldsymbol{V}_N\|_2^{-1}\boldsymbol{V}_N + Z_{N,n})$,
where $Z_{N,n}=||A_N^n\bun||_2^{-1}A_N^n\bun - \|\boldsymbol{V}_N\|_2^{-1}\boldsymbol{V}_N$. By $(vii)$, we already have $||Z_{N,n}||_2 \leq 3(2N^{-3/8})^{\lfloor n/2\rfloor+1}$. Multiplying each side by $I_K$, we obtain that $I_KA_N^n \bun = ||A_N^n\bun||_2(\|\boldsymbol{V}_N\|_2^{-1}\boldsymbol{V}^K_N + I_KZ_{N,n})$

Thus
$$
\Big|\frac{||I_KA_N^n\bun||_2}{||A_N^n\bun||_2}-\frac{\|\boldsymbol{V}^K_N\|_2}{\|\boldsymbol{V}_N\|_2}\Big|\le \|I_KZ_{N,n}\|_2\le\|Z_{N,n}\|_2\le  3(2N^{-\frac{3}{8}})^{\lfloor\frac{n}{2}\rfloor+1}. $$
So for all $n\geq 0$, we have
$$
||I_KA_N^{n}\bun||_2\in \Big[\Big(\frac{\|\boldsymbol{V}_N^K\|_2}{\|\boldsymbol{V}_N\|_2}-CN^{-\frac{3}{8}}\Big) ||A_N^n\bun||_2,\Big(\frac{\|\boldsymbol{V}_N^K\|_2}{\|\boldsymbol{V}_N\|_2}+CN^{-\frac{3}{8}}\Big) ||A_N^n\bun||_2\Big].
$$ 
Finally, recalling $(ii)$ and $(iii),$ we deduce $(ix)$.

\end{proof}

 \begin{lemma}\label{ellL} We have 
 $$
 \mathbb{E}\Big[ \|\boldsymbol{\mathcal{L}}_{N}^{K}-(\bar{L}_{N}^{K})^{5}\cL_{N}^{K}\|^{2}_{2}\Big]\le\frac{C}{N}
 $$
 where $\bL_N:=A^6_N\bun$, $\bL^{K}_{N}=I_{K}\bL_N$ and $\mL_N(i)=\sum_{j=1}^NA^6_N(i,j),$ $\bar{L}^{K}_{N}=\frac{1}{K}\sum_{i=1}^{K}L_{N}(i)$.
 \end{lemma}
 
 \newcommand{\UN}{\boldsymbol{1}_N}
 
 \begin{proof}
 We write
 \begin{align}
\|\boldsymbol{\mathcal{L}}^{K}_{N}-(\bar{L}_{N}^{K})^{5}\cL_{N}^{K}\|_{2} &=\|I_{K}A_{N}^{6}\boldsymbol{1}_{N}-(\bar{L}_{N}^{K})^{5}I_{K}A_{N}\UN\|_{2} \notag\\ 
 \label{coupeL} &\le \sum_{k=1}^{5} \|(\bar{L}_{N}^{K})^{5-k}I_{K}A^{k+1}_{N}\boldsymbol{1}_{N}-(\bar{L}_{N}^{K})^{6-k}I_{K}A^{k}_{N}\boldsymbol{1}_{N}\|_{2}\\
 &\le \sum_{k=1}^{5} \| I_{K}A^{k+1}_{N}\boldsymbol{1}_{N}-(\bar{L}_{N}^{K})I_{K}A^{k}_{N}\boldsymbol{1}_{N}\|_{2}
\end{align}
 First we study the term  corresponding to $k=1$. We have
 \begin{align*}
\mathbb{E}\Big[ \|I_{K}A^{2}_{N}\boldsymbol{1}_{N}-\bar{L}_{N}^{K}I_{K}A_{N}\boldsymbol{1}_{N}\|_{2}^{2}\Big]
 &\le 2\mathbb{E}\Big[ \|I_{K}A_{N}\cL_{N}-\bar{L}_{N}I_{K}A_{N}\boldsymbol{1}_{N}\|_{2}^{2}+\|(\bar{L}_{N}-\bar{L}_{N}^{K})I_{K}A_{N}\boldsymbol{1}_{N}\|^{2}_{2} \Big]
  \end{align*}
 By Lemma \ref{IAX} we have $\mathbb{E}[\|I_{K}A_{N}(\cL_{N}-\bar{L}_{N}\boldsymbol{1}_{N})\|_{2}^{2}]\le \frac{CK}{N^2}$. Besides we have 
 \begin{align*}
&\mathbb{E}\Big[ \|(\bar{L}_{N}-\bar{L}_{N}^{K})I_{K}A_{N}\boldsymbol{1}_{N}\|^{2}_{2}\Big]\\
&\le 2 \mathbb{E}\Big[ \|(\bar{L}_{N}-p)I_{K}A_{N}\boldsymbol{1}_{N}\|^{2}_{2}\Bigr]+ 2 \mathbb{E}\Bigl[\|(p-\bar{L}_{N}^{K})I_{K}A_{N}\boldsymbol{1}_{N}\|^{2}_{2}\Bigr]\\
 &\le 2 \mathbb{E}[(\bar{L}_{N}-p)^{4}]^{\frac{1}{2}}\mathbb{E}\Big[ \|I_{K}A_{N}\boldsymbol{1}_{N}\|^{4}_{2}\Big]^{\frac{1}{2}}+2 \mathbb{E}[(p-\bar{L}_{N}^{K})^{4}]^{\frac{1}{2}}\mathbb{E}\Big[ \|I_{K}A_{N}\boldsymbol{1}_{N}\|^{4}_{2}\Big]^{\frac{1}{2}}\\
\intertext{\hfill using the Cauchy-Schwarz inequality}
 &\le C  ( \frac{1}{N^2} K + \frac{1}{NK} K )\le \frac{C}{N}
\end{align*}
since $\|I_{K}A_{N}\boldsymbol{1}_{N}\|_2\le \sqrt{K}$, $\mathbb{E}[(\bar{L}_{N}-p)^{4}]^{\frac{1}{2}}\le \frac{C}{N^2}$ and $\mathbb{E}[(\bar{L}^K_{N}-p)^{4}]^{\frac{1}{2}}\le \frac{C}{NK}$
($NL_N(1),\dots,NL_N(K)$ are i.i.d. and Binomial$(N,p)$).
 So $$\mathbb{E}\Big[ \|I_{K}A^{2}_{N}\boldsymbol{1}_{N}-\bar{L}_{N}^{K}I_{K}A_{N}\boldsymbol{1}_{N}\|_{2}^{2}\Big]\le \frac{C}{N}.$$
 
 Next, we consider the other terms, for any $k\ge 2$. We have
 \begin{align*}
&\mathbb{E}\Big[\|I_{K}A^{k+1}_{N}\boldsymbol{1}_{N}-(\bar{L}_{N}^{K}) I_{K}A^{k}_{N}\boldsymbol{1}_{N}\|^2_{2}\Big]\\
 &\le \mathbb{E}\Big[ |||I_{K}A_{N}|||^{2}_{2}\ |||A_{N}|||^{2k-4}_{2}\, \|A^{2}_{N}\boldsymbol{1}_{N}-(\bar{L}^{K}_{N})A_{N}\boldsymbol{1}_{N}\|^{2}_{2}\Big]\\
 &\le \Big(\frac{K}{N}\Big)^2\mathbb{E}\Big[\|A^{2}_{N}\boldsymbol{1}_{N}-(\bar{L}^{K}_{N})A_{N}\boldsymbol{1}_{N}\|^{2}_{2}\Big]\\
 \intertext{\hfill since $|||I_{K}A_{N}|||_2\le {K/N}$}
 &\le 2\Big(\frac{K}{N}\Big)^2 \Big\{\mathbb{E}\Big[ \|A^{2}_{N}\boldsymbol{1}_{N}-\bar{L}_{N}A_{N}\boldsymbol{1}_{N}\|^{2}_{2}\Big]+\mathbb{E}\Big[ |\bar{L}_{N}-\bar{L}_{N}^{K}|^{2}\|A_{N}\boldsymbol{1}_{N}\|_{2}^{2}\Big]\Big\}\\
 &\le 2\Big(\frac{K}{N}\Big)^2 \Big\{\mathbb{E}\Big[ \|A_N\cX_N\|^{2}_{2}\Big]\\
 &\qquad+2\mathbb{E}\Big[ |\bar{L}_{N}-p|^{2}\|A_{N}\boldsymbol{1}_{N}\|_{2}^{2}\Big]+2\mathbb{E}\Big[|p-\bar{L}_{N}^{K}|^{2}\|A_{N}\boldsymbol{1}_{N}\|_{2}^{2}\Big]\Big\}
 \\
 &\le C \Big(\frac{K}{N}\Big)^2 \Big[\frac{1}{N}+ \frac{1}{N^{2}}N+\frac{1}{NK}N\Big]\le\frac{C}{N}.
 \end{align*}
 Recalling  (\ref{coupeL}), we conclude that
  $$
  \mathbb{E}[ \|\boldsymbol{\mathcal{L}}_{N}^{K}-(\bar{L}_{N}^{K})^{5}\cL_{N}^{K}\|^{2}_{2}]\le\frac{C}{N}$$ 
  which completes the proof.
\end{proof}

 \begin{lemma}\label{HNK}
 We have
 $$
 \mathbb{E}\Big[\boldsymbol{1}_{\Omega_{N}^{K,2}} \Big|H_{N}^{K}-\Big(\frac{1}{p}-1\Big)\Big|\Big]\le \frac{C}{\sqrt{K}},\  \hbox{where}\quad
 H_{N}^{K}:=\frac{N}{K}\sum_{i=1}^{K}\Big(\frac{L_{N}(i)-\bar{L}^{K}_{N}}{\bar{L}^{K}_{N}}\Big)^{2}.
 $$
 
 \end{lemma}
 
 \begin{proof}
 Since $\bar{L}_N^K \ge p/2$ on $\Omega_{N}^{K,2}$,  we have
 \begin{align*}
\Big|H^{K}_{N}-\Big(\frac{1}{p}-1\Big)\Big|
 &\le \Big|\frac{N}{K}\frac{\|\cL_{N}^{K}-\bar{L}_{N}^{K}\boldsymbol{1}_{K}\|^{2}_{2}}{(\bar{L}_{N}^{K})^{2}}-\frac{p(1-p)}{(\bar{L}_{N}^{K})^{2}}\Big|+p(1-p)\Big|\frac{1}{(\bar{L}_{N}^{K})^{2}}-\frac{1}{p^{2}}\Big|\\
 &\le C\Big|\frac{N}{K}\|\cX_{N}^{K}\|^{2}_{2}-p(1-p)\Big|+C|\bar{L}_{N}^{K}-p|.
 \end{align*}
  Using (\ref{ZE}) and the fact that $\mathbb{E}[(\bar{L}_{N}^{K}-p)^2]\le \frac{C}{NK},$ we obtain

 $$\mathbb{E}\Big[ \boldsymbol{1}_{\Omega_{N}^{K,2}} \Big|H_{N}^{K}-(\frac{1}{p}-1)\Big|\Big]
 \le C\mathbb{E}\Big[\Big|\frac{N}{K}\|\cX_{N}^{K}\|^{2}_{2}-p(1-p)\Big|+|\bar{L}_{N}^{K}-p| \Big]
 \le \frac{C}{\sqrt{K}}.$$
 \end{proof}

 \begin{prop} \label{VmVb}  We set $\bar{V}_{N}^{K}=\frac{1}{K}\sum_{i=1}^K V_N(i)$ and
 \begin{equation} \label{defUinf}
 \mathcal{U}_{\infty}^{N,K}:=\frac{N}{K}(\bar{V}_{N}^{K})^{-2}\sum_{i=1}^{K}(V_{N}(i)-\bar{V}_{N}^{K})^{2} \quad \text{on $\Omega_{N}^{K,2}$.}
 \end{equation}
 There exists $N_{0}\ge 1$ and $C>0$ (depending only on $p$) such that for all $N\ge N_{0}$,
 \begin{align*}
\frac{N}{K}\mathbb{E}\Big[\boldsymbol{1}_{\Omega_{N}^{K,2}}||\cV^{K}_{N}-\bar{V}^{K}_{N}\boldsymbol{1}_{K}||^{2}_{2}\Big]\le C,\quad  \mathbb{E}\Big[\boldsymbol{1}_{\Omega_{N}^{K,2}}\Big|\mathcal{U}_{\infty}^{N,K}-\Big(\frac{1}{p}-1\Big)\Big|\Big]\le \frac{C}{\sqrt{K}}.
 \end{align*}
 \end{prop}
 \begin{proof}
 We start from
 $$
 \Big|\mathcal{U}_{\infty}^{N,K}-\Big(\frac{1}{p}-1\Big)\Big|\le \Big|\mathcal{U}_{\infty}^{N,K}-\mathcal{H}^{K}_{N}\Big|+\Big|\mathcal{H}^{K}_{N}-H^{K}_{N}\Big|+\Big|H^{K}_{N}-\Big(\frac{1}{p}-1\Big)\Big|
 $$
 where
 $\mathcal{H}^{K}_{N}=\frac{N}{K}\sum_{i=1}^{K}\Big(\frac{\mathcal{L}_{N}(i)-\bar{\mathcal{L}}^{K}_{N}}{\bar{\mathcal{L}}^{K}_{N}}\Big)^{2}$ and $\bar{\mathcal{L}}^{K}_{N}=\frac{1}{K}\sum_{i=1}^K\mathcal{L}_{N}(i).$

Step 1: First we check that
$\mathbb{E}\Big[\boldsymbol{1}_{\Omega_{N}^{K,2}}\Big|H_{N}^{K}-\mathcal{H}^{K}_{N}\Big|\Big]
\le C/\sqrt{K}$.
We notice that
$$
H_{N}^{K}=\frac{N}{K}||(\bar{L}^{K}_{N})^{5}\cL^{K}_{N}-(\bar{L}^{K}_{N})^{6}\boldsymbol{1}_{K}||_{2}^{2}/(\bar{L}_{N}^{K})^{12}.
$$
 Thus
 \begin{eqnarray*}
 |H_{N}^{K}-\mathcal{H}^{K}_{N}| \le \frac{N}{K}\Big|\|(\bar{L}^{K}_{N})^{5}(\cL_{N}^{K}-\bar{L}_{N}^{K}\boldsymbol{1}_{K})\|_{2}^2 \Big(1/(\bar{L}_{N}^{K})^{12}-1/(\bar{\mathcal{L}}_{N}^{K})^{2}\Big)\\+(1/\bar{\mathcal{L}}_{N}^{K})^{2}\Big(\|(\bar{L}_{N}^{K})^{5}(\cL^{K}_{N}-\bar{L}_{N}^{K}\boldsymbol{1}_{K})\|_{2}^{2}-\|\boldsymbol{\mathcal{L}}^{K}_{N}-\bar{\mathcal{L}}^{K}_{N}\boldsymbol{1}_{K}\|_2^{2}\Big)\Big|.
 \end{eqnarray*}
  On the set $\Omega_{N}^{K,2}$, by Lemma \ref{123456} (iv), we have that $(\bar{L}_{N}^{K})^{6}\ge \frac{p^{6}}{64}$ and $\bar{\mathcal{L}}_{N}^{K}\ge \frac{(\rho_N)^{6}}{3}\ge \frac{p^6}{192}$,
   and the function $\frac{1}{x^{2}}$ is globally Lipschitz and bounded on the interval $[\frac{p^{6}}{192},\infty)$. So
  \begin{align*}
\boldsymbol{1}_{\Omega_{N}^{K,2}}|H_{N}^{K}-\mathcal{H}^{K}_{N}| &\le \frac{N}{K}\Big|\|(\bar{L}^{K}_{N})^{5}(\cL_{N}^{K}-\bar{L}_{N}^{K}\boldsymbol{1}_{K})\|^{2}_{2}\Big(1/(\bar{L}_{N}^{K})^{12}-1/(\bar{\mathcal{L}}_{N}^{K})^{2}\Big)\\
&+(1/\bar{\mathcal{L}}_{N}^{K})^{2}\Big(\|(\bar{L}_{N}^{K})^{5}(\cL^{K}_{N}-\bar{L}_{N}^{K}\boldsymbol{1}_{K})\|_{2}^{2}-\|\boldsymbol{\mathcal{L}}^{K}_{N}-\bar{\mathcal{L}}^{K}_{N}\boldsymbol{1}_{K}\|_2^{2}\Big)\Big|\\
  &\le C\frac{N}{K}\Big(\|(\bar{L}^{K}_{N})^{5}(\cL_{N}^{K}-\bar{L}_{N}^{K}\boldsymbol{1}_{K})\|^{2}_{2}\Big|(\bar{L}_{N}^{K})^{6}-\bar{\mathcal{L}}_{N}^{K}\Big|\\
  &+\Big|\|(\bar{L}_{N}^{K})^{5}(\cL^{K}_{N}-\bar{L}_{N}^{K}\boldsymbol{1}_{K})\|_{2}^{2}-\|\boldsymbol{\mathcal{L}}^{K}_{N}-\bar{\mathcal{L}}^{K}_{N}\boldsymbol{1}_{K}\|_2^{2}\Big|\Big).
   \end{align*}
  
 Next, we use the inequality $|a^{2}-b^{2}|\le (a-b)^{2}+2a|a-b|$ for $a,\ b\ge 0$. 
  So  
  \begin{align*}
\boldsymbol{1}_{\Omega_{N}^{K,2}}|H_{N}^{K}-\mathcal{H}^{K}_{N}| &\le C\frac{N}{K}\Big(\|(\bar{L}^{K}_{N})^{5}(\cL_{N}^{K}-\bar{L}_{N}^{K}\boldsymbol{1}_{K})\|^{2}_{2}\Big|(\bar{L}_{N}^{K})^{6}-\bar{\mathcal{L}}_{N}^{K}\Big|\\
&+\Big|\|(\bar{L}_{N}^{K})^{5}(\cL^{K}_{N}-\bar{L}_{N}^{K}\boldsymbol{1}_{K})\|_{2}^{2}-\|\boldsymbol{\mathcal{L}}^{K}_{N}-\bar{\mathcal{L}}^{K}_{N}\boldsymbol{1}_{K}\|_2^{2}\Big|\Big)\\
&\le C\frac{N}{K}\Big\{\|(\bar{L}_{N}^{K})^{5})(\cL_{N}^{K}-\bar{L}_{N}^{K}\boldsymbol{1}_{K})\|_{2}^{2}\frac{1}{\sqrt{K}}J_{N}^{K}+(I_{N}^{K})^{2}\\&+\|(\bar{L}_{N}^{K})^{5}(\cL_{N}^{K}-\bar{L}_{N}^{K}\boldsymbol{1}_{K})\|_{2}^{2}I_{N}^{K}\Big\},
  \end{align*}
   where $$J_{N}^{K}=\|[\bar{\mathcal{L}}_{N}^{K}-(\bar{L}_{N}^{K})^{6}]\boldsymbol{1}_{K}\|_{2}=\sqrt{K}\Big|\bar{\mathcal{L}}_{N}^{K}-(\bar{L}_{N}^{K})^{6}\Big|,\ I_{N}^{K}=\|(\boldsymbol{\mathcal{L}}_{N}^{K}-\bar{\mathcal{L}}_{N}^{K}\boldsymbol{1}_{K})-(\bar{L}_{N}^{K})^{5}(\cL_{N}^{K}-\bar{L}_{N}^{K}\boldsymbol{1}_{K})\|_{2}.$$

  Because 
  $$
  \Big((\boldsymbol{\mathcal{L}}_{N}^{K}-\bar{\mathcal{L}}_{N}^{K}\boldsymbol{1}_{K})-(\bar{L}_{N}^{K})^{5}(\cL_{N}^{K}-\bar{L}_{N}^{K}\boldsymbol{1}_{K}), \boldsymbol{1}_{K}\Big)=0.
  $$
   it implies $$(J_{N}^{K})^{2}+(I_{N}^{K})^{2}=\|\boldsymbol{\mathcal{L}}_{N}^{K}-(\bar{L}^{K}_{N})^{5}\cL_{N}^{K}\|^{2}_{2}.$$
   And by Lemma \ref{ellL}, we conclude  $$\frac{N}{K}\mathbb{E}\Big[\Big\{(I_{N}^{K})^{2}+(J_{N}^{K})^{2}\Big\}\Big]= \frac{N}{K}\mathbb{E}\Big[ \|\boldsymbol{\mathcal{L}}_{N}^{K}-(\bar{L}^{K}_{N})^{5}\cL_{N}^{K}\|^{2}_{2}\Big]\le \frac{C}{K}.$$

   By (\ref{ZE}), we conclude that $\mathbb{E}\Big[\Big(\frac{N}{K}\Big)^{2}\|X_{N}^{K}\|^{4}_{2}\Big]\le C$.
   Finally,
   \begin{align*}
&\mathbb{E}\Big[\boldsymbol{1}_{\Omega_{N}^{K,2}}\Big|H_{N}^{K}-\mathcal{H}^{K}_{N}\Big|\Big]
   \\&\le C\frac{N}{K}\mathbb{E}\Big[\Big\{\|\cX_{N}^{K}\|_{2}^{2}\frac{1}{\sqrt{K}}J_{N}^{K}+(I_{N}^{K})^{2}+\|\cX_{N}^{K}\|_{2}^{2}I_{N}^{K}\Big\}\Big]
   \\&\le C\frac{N}{K}\mathbb{E}\Big[\Big\{\|\cX_{N}^{K}\|_{2}^{2}J_{N}^{K}+(I_{N}^{K})^{2}+\|\cX_{N}^{K}\|_{2}^{2}I_{N}^{K}\Big\}\Big] \\
   &\le C\frac{N}{K}\mathbb{E}\Big[\Big\{(I_{N}^{K})^{2}+(J_{N}^{K})^{2}\Big\}\Big]+C\mathbb{E}\Big[\Big(\frac{N}{K}\Big)^{2}\|\cX_{N}^{K}\|^{4}_{2}\Big]^{\frac{1}{2}}\mathbb{E}\Big[\Big\{(I_{N}^{K})^{2}+(J_{N}^{K})^{2}\Big\}\Big]^{\frac{1}{2}} \\
   &\le \frac{C}{\sqrt{K}}.
  \end{align*}
   
   Step 2:
  By (ii) and (iv) in Lemma \ref{123456},
    we have the following inequality under the set $\Omega_{N}^{K,2}$:
   \begin{align*}
    \Big|\mathcal{U}_{\infty}^{N,K}-\mathcal{H}_{N,K}\Big|
   =\frac{N}{K}\Big|\sum_{i=1}^{K}\Big[(V_{N}(i)/\bar{V}_{N}^{K})^{2}-(\mathcal{L}_{N}(i)/\bar{\mathcal{L}}^{K}_{N})^{2}\Big]\Big|\\ \le C\frac{N}{K}\sum_{i=1}^{K}\Big|V_{N}(i)/\bar{V}_{N}^{K}-\mathcal{L}_{N}(i)/\bar{\mathcal{L}}_{N}^{K}\Big|
   \end{align*}
  Then we use the lemma \ref{123456} $(v)$: on the set $\Omega_{N}^{K,2}$ we have
  \begin{align*}
\frac{N}{K}\sum_{i=1}^{K}\Big|V_{N}(i)/\bar{V}_{N}^{K}-\mathcal{L}_{N}(i)/\bar{\mathcal{L}}_{N}^{K}\Big|
  &=N\Big\|\|I_{K}A_{N}^{6}\boldsymbol{1}_{N}\|^{-1}_{1}I_{K}A_{N}^{6}\boldsymbol{1}_{N}-\|\cV_{N}^{K}\|_{1}^{-1}\cV_{N}^{K}\Big\|_{1}\\
  &\le CN(N^{-\frac{3}{8}})^{3+1}\le\frac{C}{\sqrt{N}} 
  \end{align*}
  So we have the following inequality:
  $$
  \mathbb{E}\Big[\boldsymbol{1}_{\Omega_{N}^{K,2}}\Big|\mathcal{U}_{\infty}^{N,K}-\mathcal{H}_{N,K}\Big|\Big]\le \frac{C}{\sqrt{N}}.
  $$
  
  Step 3: From the two previous steps and  lemma \ref{HNK}, it follows that
$$\mathbb{E}\Big[\boldsymbol{1}_{\Omega_{N}^{K,2}}\Big|\mathcal{U}_{\infty}^{N,K}-\Big(\frac{1}{p}-1\Big)\Big|\Big]\le \frac{C}{\sqrt{K}}.$$
Moreover, by lemma \ref{123456} $(ii)$, $\bar{V}^{K}_{N}$ is bounded by $2$ on the set $\Omega_{N}^{K,2}$,
thus $$\frac{N}{K}\mathbb{E}\Big[\boldsymbol{1}_{\Omega_{N}^{K,2}}\|\boldsymbol{V}^{K}_{N}-\bar{V}^{K}_{N}\boldsymbol{1}_{K}\|^{2}_{2}\Big] = \mathbb{E}\Big[\boldsymbol{1}_{\Omega_{N}^{K,2}}(\bar{V}_{N}^{K})^{2}|\mathcal{U}_{\infty}^{N,K}|\Big]\le C.$$
\end{proof}
\section{The  estimator in the supercritical case}
 Recall the definition in (\ref{UP}), the aim of this section is to prove $\mathcal{P}_{t}^{N,K}\simeq p.$ Recall (\ref{ee1}) and (\ref{ee2}). We start from
 \begin{eqnarray}
 \label{defI}
 \mathbb{E}_{\theta}[\boldsymbol{Z}_{t}^{N,K}]=\mu\sum_{n\ge 0}\Big[\int^{t}_{0}s\phi^{*n}(t-s)ds\Big]I_{K}A^{n}_{N}\boldsymbol{1}_{N}=v_{t}^{N,K}\boldsymbol{V}_{N}^{K}+\boldsymbol{I}_{t}^{N,K},\\
 \label{defJ}
 \boldsymbol{U}_{t}^{N,K}=\boldsymbol{Z}_{t}^{N,K}-\mathbb{E}_{\theta}[\boldsymbol{Z}_{t}^{N,K}]=\sum_{n\ge 0}\Big[\int^{t}_{0}\phi^{*n}(t-s)\Big]I_{K}A^{n}_{N}\boldsymbol{M}_{s}^{N}ds=\boldsymbol{M}_{t}^{N,K}+\boldsymbol{J}_{t}^{N,K}
  \end{eqnarray}

 where
 \begin{align}
     \label{vtNK}v_{t}^{N,K}=\mu \sum_{n\ge 0}\frac{\|I_{K}A_{N}^{n}\boldsymbol{1}_{N}\|_{2}}{\|\cV_{N}^{K}\|_{2}}\int_{0}^{t}s\phi^{*n}(t-s)ds,\\
     \label{ItNK}\boldsymbol{I}_{t}^{N,K}=\mu\sum_{n\ge 0}\Big[\int^{t}_{0}s\phi^{*n}(t-s)ds\Big]\Big[I_{K}A^{n}_{N}\boldsymbol{1}_{N}-\frac{\|I_{K}A_{N}^{n}\boldsymbol{1}_{N}\|_{2}}{\|\boldsymbol{V}_{N}^{K}\|_{2}}\boldsymbol{V}_{N}^{K}\Big]
 \end{align}

 and
 \begin{align}
\label{JtNK}\boldsymbol{J}_{t}^{N,K}=\sum_{n\ge 1}\Big[\int_{0}^{t}\phi^{*n}(t-s)\Big]I_{K}A^{n}_{N}\boldsymbol{M}_{s}^{N}ds.
 \end{align}

 \begin{lemma}\label{INK}
 Assume (A). For all $\eta >0$, there exists $N_{\eta}\ge 1$ and $C_{\eta}<\infty$ such that for all $N\ge N_{\eta}$, $t\ge 0$, on the set $\Omega_{N}^{K,2}$, we have

 $$\|\boldsymbol{I}_{t}^{N,K}\|_{2}\le C_{\eta}t\sqrt{K}N^{-\frac{3}{8}}.$$
 \end{lemma}
 
 \begin{proof}
In view of (\ref{ItNK}), Lemma \ref{123456} (vii) yields
 \begin{align*}
\|\boldsymbol{I}_{t}^{N,K}\|_{2} &\le\mu\sum_{n\ge 0}\Big[\int^{t}_{0}s\phi^{*n}(t-s)ds\Big]\Big\|I_{K}A^{n}_{N}\boldsymbol{1}_{N}-\frac{\|I_{K}A_{N}^{n}\boldsymbol{1}_{N}\|_{2}}{\|\boldsymbol{V}_{N}^{K}\|_{2}}\boldsymbol{V}_{N}^{K}\Big\|_{2}\\
&\le C_{\eta}t\sqrt{K}\sum_{n\ge 0}\Big[\int_{0}^{t}\phi^{*n}(t-s)ds\Big](N^{-\frac{3}{8}})^{\lfloor\frac{n}{2}\rfloor+1}\\
&\le C_{\eta}t\sqrt{K}N^{-\frac{3}{8}}\sum_{n\ge 0}\Lambda^n(N^{-\frac{3}{8}})^{\lfloor\frac{n}{2}\rfloor}\\
&\le C_{\eta}t\sqrt{K}N^{-\frac{3}{8}}.
 \end{align*}
\end{proof}

 \begin{lemma}\label{JNK}
 Assume (A). For all $\eta >0$, there exists $N_{\eta}\ge 1$ and $C_{\eta}<\infty$ such that for all $N\ge N_{\eta}$, $t\ge 0$, on the set $\Omega_{N}^{K,2}$, we have 
 $$\mathbb{E}_{\theta}\Big[\|\boldsymbol{J}_{t}^{N,K}-\bar{J}_{t}^{N,K}\boldsymbol{1}_{K}\|_{2}^{2}\Big]^\frac{1}{2}\le C_{\eta}\sqrt{\frac{K}{N}}\Big[e^{\frac{1}{2}(\alpha_{0}+\eta)t}+\frac{\|\boldsymbol{V}_{N}^{K}-\bar{V}_{N}^{K}\boldsymbol{1}_{K}\|_{2}}{\|\boldsymbol{V}_{N}^{K}\|_{2}}e^{(\alpha_{0}+\eta)t}\Big]$$
 where $\bar{J}_{t}^{N,K}=\frac{1}{K} ( \boldsymbol{J}_{t}^{N,K}, \boldsymbol{1}_{K})$.
 \end{lemma}
 
 \begin{proof}
 In view of (\ref{JtNK}), by Minkowski inequality we have
 \begin{align*}
&\mathbb{E}_{\theta}\Big[\|\boldsymbol{J}_{t}^{N,K}-\bar{J}_{t}^{N,K}\boldsymbol{1}_{K}\|_{2}^{2}\Big]^{\frac{1}{2}}
 \\&\le \sum_{n\ge 1}\int_{0}^{t}\phi^{*n}(t-s)\mathbb{E}_{\theta}\Big[\|I_{K}A_{N}^{n}\boldsymbol{M}_{s}^{N}-\overline{I_{K}A_{N}^{n}\boldsymbol{M}_{s}^{N}}\boldsymbol{1}_{K}\|^{2}_{2}\Big]^{\frac{1}{2}}.
 \end{align*}
 where $\overline{I_{K}A_{N}^{n}\boldsymbol{M}_{s}^{N}}:=\frac{1}{K}\sum_{j=1}^N\sum_{i=1}^KA_N^n(i,j)M_s^{j,N}.$\\
In \cite[Lemma 44 (i)]{A}, it is shown that $\max_{i=1,...,N}\mathbb{E}_\theta[(Z^{i,N}_t)^2] \leq C_\eta e^{2(\alpha_0+\eta)t}$ on $\Omega^2_N$. Using (\ref{ee3}), we conclude that on $\Omega^2_N$:
\begin{align*}
\mathbb{E}_{\theta}\Big[\|I_{K}A_{N}^{n}\boldsymbol{M}_{s}^{N}-\overline{I_{K}A_{N}^{n}\boldsymbol{M}_{s}^{N}}\boldsymbol{1}_{K}\|^{2}_{2}\Big]
&=\sum_{i=1}^{K}\sum_{j=1}^{N}\Big(A_{N}^{n}(i,j)-\frac{1}{K}\sum_{k=1}^{K}A_{N}^{n}(k,j)\Big)^{2}\mathbb{E}_{\theta}[Z_{s}^{j,N}]\\
&\le C_{\eta}e^{(\alpha_{0}+\eta)s}\sum_{j=1}^{N}\|I_{K}A_{N}^{n}\boldsymbol{e}_{j}-\overline{I_{K}A_{N}^{n}\boldsymbol{e}_{j}}\boldsymbol{1}_{K}\|_{2}^{2}.
\end{align*}
Using $(viii)$ in Lemma \ref{123456} and and the inequality $\big|||\cx-\bar x \indiq_N||_2-||\boldsymbol{y}-\bar y \indiq_N||_2\big|\leq ||\cx-\boldsymbol{y}||_2$ for all $x,\ y\in \mathbb{R}^N,$ we deduce that on $\Omega^{K,2}_N$:
\begin{align*}
&\|I_{K}A_{N}^{n}\boldsymbol{e}_{j}-\overline{I_{K}A_{N}^{n}\boldsymbol{e}_{j}}\boldsymbol{1}_{K}\|_{2}\\
&\le \Big\|I_{K}A_{N}^{n}\boldsymbol{e}_{j}-\frac{1}{\|\boldsymbol{V}_{N}^{K}\|_{2}}\|I_{K}A_{N}^{n}\boldsymbol{e}_{j}\|_{2}\boldsymbol{V}_{N}^{K}\Big\|_{2}+\frac{\|I_{K}A^{n}_{N}\boldsymbol{e}_{j}\|_{2}}{\|\boldsymbol{V}_{N}^{K}\|_{2}}\|\boldsymbol{V}_{N}^{K}-\bar{V}_{N}^{K}\boldsymbol{1}_{K}\|_{2}\\
&=\|I_{K}A^{n}_{N}\boldsymbol{e}_{j}\|_{2}\Big(\Big\|\frac{I_{K}A^{n}_{N}\boldsymbol{e}_{j}}{\|I_{K}A^{n}_{N}\boldsymbol{e}_{j}\|_{2}}-\frac{\boldsymbol{V}_{N}^{K}}{\|\boldsymbol{V}_{N}^{K}\|_{2}}\Big\|_{2}+\frac{\|\boldsymbol{V}_{N}^{K}-\bar{V}_{N}^{K}\boldsymbol{1}_{K}\|_{2}}{\|\boldsymbol{V}_{N}^{K}\|_{2}}\Big)
\\ &\le C\|I_{K}A^{n}_{N}\boldsymbol{e}_{j}\|_{2}\Big(N^{-\frac{3}{8}\lfloor\frac{n}{2}\rfloor}+\frac{\|\boldsymbol{V}_{N}^{K}-\bar{V}_{N}^{K}\boldsymbol{1}_{K}\|_{2}}{\|\boldsymbol{V}_{N}^{K}\|_{2}}\Big).
\end{align*}
From Lemma \ref{123456} $(iv)$ it follows that on the event $\Omega_{N}^{K,2}$ for all $n\ge 2$, $\|I_{K}A^{n}_{N}\boldsymbol{e}_{j}\|_{2}\le \frac{3\sqrt{K}}{N}\rho_{N}^{n}$.
So on the event $\Omega_{N}^{K,2}$, 
\begin{align*}
&\mathbb{E}_{\theta}[\|\boldsymbol{J}_{t}^{N,K}-\bar{J}_{t}^{N,K}\boldsymbol{1}_{K}\|^{2}_{2}]^{\frac{1}{2}}
\\&\le C_{\eta}\sqrt{\frac{K}{N}}\sum_{n\ge 1}\rho_{N}^{n}\Big[(2N^{-\frac{3}{8}})^{\lfloor\frac{n}{2}\rfloor}+\frac{\|\boldsymbol{V}_{N}^{K}-\bar{V}_{N}^{K}\boldsymbol{1}_{K}\|_{2}}{\|\boldsymbol{V}_{N}^{K}\|_{2}}\Big]\int_{0}^{t}\phi^{*n}(t-s)e^{\frac{(\alpha_0+\eta)s}{2}}ds.
\end{align*}
Using \cite[lemma 43 (iii) and (iv)]{A}, we deduce that on the event $\Omega_{N}^{K,2}$
$$\mathbb{E}_{\theta}[\|\boldsymbol{J}_{t}^{N,K}-\bar{J}_{t}^{N,K}\boldsymbol{1}_{K}\|_{2}^{2}]^\frac{1}{2}\le C_{\eta}\sqrt{\frac{K}{N}}\Big[e^{\frac{1}{2}(\alpha_{0}+\eta)t}+\frac{\|\boldsymbol{V}_{N}^{K}-\bar{V}_{N}^{K}\boldsymbol{1}_{K}\|_{2}}{\|\boldsymbol{V}_{N}^{K}\|_{2}}e^{(\alpha_{0}+\eta)t}\Big].$$

\end{proof}

\begin{lemma}\label{DNK}
There exists $N_0\ge1$ such that for all $N\ge N_{0}$, for all $t\ge 0$, on the event $\Omega_{N}^{K,2}\cap \{\bar{Z}_{t}^{N,K}\ge \frac{1}{4}v_{t}^{N,K}>0\}$,
we have the following inequality:
$$\cD_{t}^{N,K}\le 16\cD_{t}^{N,K,1}+128\frac{N}{K}\|\cV_{N}^{K}-\bar{V}_{N}^{K}\boldsymbol{1}_{K}\|_{2}^{2} \, \cD_{t}^{N,K,2}+\Big|\mathcal{U}_{\infty}^{N,K}-\Big(\frac{1}{p}-1\Big)\Big|$$
where
\begin{align}
\cD_{t}^{N,K}&=\Big|\mathcal{U}_{t}^{N,K}-\Big(\frac{1}{p}-1\Big)\Big|,\\
 \label{cD1}\cD_{t}^{N,K,1}&=\frac{1}{(v_{t}^{N,K})^{2}}\Big|\frac{N}{K}\|\boldsymbol{Z}_{t}^{N,K}-\bar{Z}_{t}^{N,K}\boldsymbol{1}_{K}\|_{2}^{2}-N\bar{Z}_{t}^{N,K}-\frac{N}{K}(v_{t}^{N,K})^{2}\|\boldsymbol{V}_{N}^{K}-\bar{V}_{N}^{K}\boldsymbol{1}_{K}\|_{2}^{2}\Big|,\\
\cD_{t}^{N,K,2}&=\Big|\frac{\bar{Z}_{t}^{N,K}}{v_{t}^{N,K}}-\bar{V}_{N}^{K}\Big|,\ 
\end{align}

\end{lemma}

\begin{proof}
Recall definitions (\ref{UP}) and \eqref{defUinf}. On the event $\Omega_{N}^{K,2}\cap \{\bar{Z}_{t}^{N,K}\ge \frac{1}{4}v_{t}^{N,K}>0\}$, we have
\begin{multline*}
|\mathcal{U}_{t}^{N,K}-\mathcal{U}_{\infty}^{N,K}| \le \frac{1}{(\bar{Z}_{t}^{N,K})^{2}}\Big|\frac{N}{K}\|\boldsymbol{Z}_{t}^{N,K}-\bar{Z}_{t}^{N,K}\boldsymbol{1}_{K}\|_{2}^{2}-N\bar{Z}_{t}^{N,K}-(v_{t}^{N,K})^{2}\frac{N}{K}\|\boldsymbol{V}_{N}^{K}-\bar{V}_{N}^{K}\boldsymbol{1}_{K}\|_{2}^{2}\Big|\\
+\frac{N}{K}\|\boldsymbol{V}_{N}^{K}-\bar{V}_{N}^{K}\boldsymbol{1}_{K}\|_{2}^{2}\Big|\Big(\frac{v_{t}^{N,K}}{\bar{Z}_{t}^{N,K}}\Big)^{2}-\frac{1}{(\bar{V}_{N}^{K})^{2}}\Big|.
\end{multline*}
  By  \cite[lemma 35 (ii)]{A}, we have $\bar{V}_{N}^{K}\ge \frac{1}{2}$ on $\Omega_{N}^{K,2}$. Since
  $|\frac{1}{x^{2}}-\frac{1}{y^{2}}|=|\frac{(x-y)(x+y)}{x^{2}y^{2}}|\le 128|x-y|$, for $x,y \ge \frac{1}{4}$, on the event $\Omega_{N}^{K,2}\cap \{\bar{Z}_{t}^{N,K}\ge \frac{1}{4}v_{t}^{N,K}>0\}$, we have 
$$
\Big|\Big(\frac{v_{t}^{N,K}}{\bar{Z}_{t}^{N,K}}\Big)^{2}-\frac{1}{(\bar{V}_{N}^{K})^{2}}\Big|\le 128\cD_{t}^{N,K,2}.
$$
Finally on the event $\Omega_{N}^{K,2}\cap \{\bar{Z}_{t}^{N,K}\ge \frac{1}{4}v_{t}^{N,K}>0\}$, we obtain
\begin{align*}
\cD_{t}^{N,K}&\le \Big|\mathcal{U}_{t}^{N,K}-\mathcal{U}_{\infty}^{N,K}\Big|+\Big|\mathcal{U}_{\infty}^{N,K}-\Big(\frac{1}{p}-1\Big)\Big|
\\ 
&\le 16\cD_{t}^{N,K,1}+128\frac{N}{K}\|\boldsymbol{V}_{N}^{K}-\bar{V}_{N}^{K}\boldsymbol{1}_{K}\|_{2}^{2}\cD_{t}^{N,K,2}+\Big|\mathcal{U}_{\infty}^{N,K}-\Big(\frac{1}{p}-1\Big)\Big|.
\end{align*}

\end{proof}
Before the analysis of the term $\cD_{t}^{N,K,2},$ we still need the following fact:
\begin{lemma}\label{Cvc}
Assume $(A)$. For any $\eta>0$, we can find $N_{\eta}\ge1$, $t_{\eta}>0$ and $0<c_{\eta}<C_{\eta}<\infty$, such that for all $N\ge N_{\eta}$, $t\ge t_{\eta}$ on the set $\Omega_{N}^{K,2}$
\begin{align*}
 c_{\eta}e^{(\alpha_{0}-\eta)t}\le v_{t}^{N,K}\le C_{\eta}e^{(\alpha_{0}+\eta)}
\end{align*}
where $v_t^{N,K}$ is defined in (\ref{vtNK}).
\end{lemma}
\begin{proof}
We work on the set $\Omega_{N}^{K,2}$.
Recall Lemma \ref{123456} (ii) and (ix). We can conclude that $\frac{1}{2}\sqrt{K}\le \|\cV_{N}^{K}\|_{2}\le 2\sqrt{K}$ and $\|I_{K}A^{n}_{N}\boldsymbol{1}_{N}\|_{2}\in [\sqrt{K}\rho^{n}_{N}/8,8\sqrt{K}\rho^{n}_{N}] .$  So there exists $0<c<C<\infty$ such that
$$
c\frac{\|A_{N}^{n}\boldsymbol{1}_{N}\|_{2}}{\|\cV_{N}\|_{2}}\le \frac{\|I_{K}A_{N}^{n}\boldsymbol{1}_{N}\|_{2}}{\|\cV_{N}^{K}\|_{2}}\le C\frac{\|A_{N}^{n}\boldsymbol{1}_{N}\|_{2}}{\|\cV_{N}\|_{2}}.
$$
Therefore we have $c v_{t}^{N,N}\le v_{t}^{N,K}\le C v_{t}^{N,N}.$
Moreover, in view of \cite[(i) and (ii) Lemma 43]{A}, we already have $c_{\eta}e^{(\alpha_{0}-\eta)t}\le v_{t}^{N,N}\le C_{\eta}e^{(\alpha_{0}+\eta)}.$
The proof is finished.
\end{proof}

\begin{lemma}\label{D2K}
Assume (A). For all $\eta>0$, there exists $N_{\eta}\ge 1$, $t_{\eta}\ge 0$ and $C_{\eta}<\infty$ such that for all $N\ge N_{\eta}$, all $t\ge t_{\eta}$, on the event $\Omega_{N}^{K,2}$,

$$(i)\quad \mathbb{E}_{\theta}[\cD_{t}^{N,K,2}]\le C_{\eta}e^{2\eta t}\Big(\frac{1}{\sqrt{K}}+e^{-\alpha_{0}t}\Big).$$
$$(ii)\quad P_{\theta}\Big(\bar{Z}_{t}^{N,K}\le \frac{1}{4}v_{t}^{N,K}\Big)\le C_{\eta}e^{2\eta t}\Big(\frac{1}{\sqrt{K}}+e^{-\alpha_{0}t}\Big).$$

\end{lemma}

\begin{proof}
Recalling \eqref{defI} and \eqref{defJ}, we can write
 $$
 \cD_{t}^{N,K,2}=\Big|\frac{\bar{Z}_{t}^{N,K}}{v_{t}^{N,K}}-\bar{V}_{N}^{K}\Big|
 \le (v_{t}^{N,K})^{-1}\Big(|\bar{I}_{t}^{N,K}|+|\bar{U}_{t}^{N,K}|\Big).
 $$
 We fix $\eta>0$ and work with $N$ large enough and on $\Omega_N^{K,2}$.
  By the lemma \ref{INK}, we have:
 $|\bar{I}_{t}^{N,K}|\le \frac{1}{\sqrt{K}}\|\boldsymbol{I}_{t}^{N,K}\|_{2}\le C_{\eta}t N^{-\frac{3}{8}}.$\\
 From \cite[proof of Lemma 44, step 3]{A}, we have $\Et[(J^{i,N}_t)^2] \leq C_\eta N^{-1}e^{2(\alpha_0+\eta)t}$. Thus
 $$\mathbb{E}_{\theta}[(\bar{J}_{t}^{N,K})^{2}]\le K^{-1}\sum_{i=1}^{K}\mathbb{E}_{\theta}[(J_{t}^{i,N})^{2}]\le C_{\eta}\frac{1}{N}e^{2(\alpha_{0}+\eta)t}.$$
  In view of  \cite[Lemma 44 (i)]{A}, we already have $\max_{i=1,\dots,N} \Et[(Z^{i,N}_t)^2] \leq C_\eta e^{2(\alpha_0+\eta)t}$. Then by (\ref{ee3}) we deduce that
 $$
 \mathbb{E}[(\bar{M}_{t}^{N,K})^{2}]=\frac{1}{K^{2}}\sum_{i=1}^{K}\mathbb{E}_{\theta}[Z_{t}^{i,N}]
 \le C_{\eta}\frac{1}{K}e^{(\alpha_{0}+\eta)t}.
 $$
Over all, we deduce that $\mathbb{E}[|\bar{U}_{t}^{N,K}|]\le \frac{C}{\sqrt{K}}e^{(\alpha_{0}+\eta) t}$.
 According to  Lemma \ref{Cvc}, there exists $t_\eta\ge 0$ such that for all $t\ge t_{\eta}$, $v_{t}^{N,K}\ge c_{\eta}e^{(\alpha_{0}-\eta)t}$ and we finally obtain $(i)$:
$$\mathbb{E}_{\theta}[\cD_{t}^{N,K,2}]
 =\mathbb{E}_{\theta}\Big[(v_{t}^{N,K})^{-1}\Big(|\bar{I}_{t}^{N,K}|+|\bar{U}_{t}^{N,K}|\Big)\Big] \le C_{\eta}e^{2\eta t}\Big(\frac{1}{\sqrt{K}}+e^{-\alpha_{0}t}\Big).$$
 
Now we prove $(ii)$. Because of $\bar{V}_{N}^{K}\ge \frac{1}{2}$ we have $\bigl\{\bar{Z}_{t}^{N,K}\le \frac{v_{t}^{N}}{4}\bigr\} \subset \bigl\{ \cD_{t}^{N,K,2}=\Big|\frac{\bar{Z}_{t}^{N,K}}{v_{t}^{N}}-\bar{V}_{N}^{K}\Big|\ge \frac{1}{4} \bigr\}$.
 Hence
 $$P_{\theta}\Big(\bar{Z}_{t}^{N,K}\le \frac{1}{4}v_{t}^{N,K}\Big)\le 4\mathbb{E}_{\theta}[\cD_{t}^{N,K,2}]
 \le C_{\eta}e^{2\eta t}\Big(\frac{1}{\sqrt{K}}+e^{-\alpha_{0}t}\Big).$$
 
 \end{proof}

 \begin{lemma}\label{MVV}
 Assume $(A)$. For all $\eta>0$, there exists $N_\eta\geq 1$ and $C_\eta<\infty$ such that for all $N\geq N_\eta$, all $t\geq 0$,
on $\Omega_N^{K,2}$:
\begin{itemize}
\item[(i)] $\mathbb{E}_{\theta}[(\boldsymbol{M}_{t}^{N,K}-\bar{M}_{t}^{N,K}\boldsymbol{1}_{K},\cV_{N}^{K}-\bar{V}_{N}^{K}\boldsymbol{1}_{K})^{2}]\le C_{\eta}\|\cV_{N}^{K}-\bar{V}_{N}^{K}\boldsymbol{1}_{K}\|_{2}^{2} \, e^{(\alpha_{0}+\eta)t}.$
 \item[(ii)] $\mathbb{E}_\theta[|X_{t}^{N,K}|]\le C_{\eta}\frac{N}{\sqrt{K}}e^{(\alpha_{0}+\eta)t},$  where
  $X_{t}^{N,K}:=\frac{N}{K}(\|\boldsymbol{M}_{t}^{N,K}-\bar{M}_{t}^{N,K}\boldsymbol{1}_{K}\|^{2}_{2}-K\bar{Z}_{t}^{N,K}).$
\item[(iii)] $\mathbb{E}_{\theta}[\|\boldsymbol{M}_{t}^{N,K}-\bar{M}_{t}^{N,K}\boldsymbol{1}_{K}\|_{2}^{2}]\le CNe^{(\alpha_{0}+\eta)t}.$
\end{itemize}
 \end{lemma}
 
 \begin{proof}
  We fix $\eta>0$ and work with $N$ large enough and on $\Omega_N^{K,2}$.
 We already from \cite[Lemma 44 (i)]{A} that $\max_{i=1,\dots,N} \Et[(Z^{i,N}_t)^2] \leq C_\eta e^{2(\alpha_0+\eta)t}$.  Thus
 \begin{align*}
\mathbb{E}_{\theta}\Big[\Big(\boldsymbol{M}_{t}^{N,K}-\bar{M}_{t}^{N,K}\boldsymbol{1}_{K},\cV_{N}^{K}-\bar{V}_{N}^{K}\boldsymbol{1}_{K}\Big)^{2}\Big]
 &=\sum_{i=1}^{K}(V_{N}(i)-\bar{V}_{N}^{K})^{2}\mathbb{E}_{\theta}[Z_{t}^{i,N}]\\
 &\le C_{\eta}\|\cV_{N}^{K}-\bar{V}_{N}^{K}\|_{2}^{2}e^{(\alpha_{0}+\eta)t}
 \end{align*}
 which completes the proof of (i).
 
 By It\^o's formula, we have
  $$\|\boldsymbol{M}_{t}^{N,K}\|_{2}^{2}=\sum_{i=1}^{K}(M_{t}^{i,N})^{2}
 =2\sum_{i=1}^{K}\int_{0}^{t}M_{s-}^{i,N}dM_{s}^{i,N}+\sum_{i=1}^{K}Z_{t}^{i,N},$$
 hence
  \begin{align*}
X_{t}^{N,K} &= \frac{N}{K}\Big( \|\boldsymbol{M}_{t}^{N,K}\|_{2}^{2} -K (\bar{M}_{t}^{N,K})^2
 -K \bar{Z}_{t}^{N,K}\Big) \\
&=\frac{N}{K}\Big( 2\sum_{i=1}^{K}\int_{0}^{t}M_{s-}^{i,N}dM_{s}^{i,N}
-K (\bar{M}_{t}^{N,K})^2  \Big).
 \end{align*}
 It follows that
 $$
\mathbb{E}_{\theta}[|X_{t}^{N,K}|]
 \le \frac{N}{K}\Big(2\mathbb{E}_{\theta}\Big[\Big|\sum_{i=1}^{K}\int_{0}^{t}M_{s-}^{i,N}dM_{s}^{i,N}\Big|\Big]+\mathbb{E}_\theta[\bar{Z}_{t}^{N,K}]\Big).
$$

Besides,  using Cauchy-Schwartz inequality 
\begin{align*}
 \mathbb{E}_{\theta}\Big[\Big(\sum_{i=1}^{K}\int_{0}^{t}M_{s-}^{i,N}dM_{s}^{i,N}\Big)^2\Big]&=\sum_{i=1}^{K}\mathbb{E}_{\theta}\Big[\int_{0}^{t}(M_{s-}^{i,N})^{2}dZ_{s}^{i,N}\Big]\\
 &\le \sum_{i=1}^{K}\mathbb{E}_\theta \Big[\sup_{[0,t]}(M_{s}^{i,N})^{4}\Big]^{\frac{1}{2}}
 \mathbb{E}_{\theta}\Big[(Z_{t}^{i,N})^{2}\Big]^{\frac{1}{2}}\\
& \le C \sum_{i=1}^{K} \mathbb{E}_{\theta}\Big[(Z_{t}^{i,N})^{2}\Big]
 \end{align*}
 since  $\mathbb{E}_{\theta}[\sup_{[0,t]}(M_{s}^{i,N})^{4}]\le C\mathbb{E}_{\theta}[(Z_{t}^{i,N})^{2}]$
 by Doob's inequality.
 So
 \begin{align*}
\mathbb{E}_{\theta}[|X_{t}^{N,K}|]
 &\le \frac{N}{K}\Big(2\mathbb{E}_{\theta}\Big[\Big|\sum_{i=1}^{K}\int_{0}^{t}M_{s-}^{i,N}dM_{s}^{i,N}\Big|\Big]+\mathbb{E}_\theta[\bar{Z}_{t}^{N,K}]\Big)
 \le C_{\eta}\frac{N}{\sqrt{K}}e^{(\alpha_{0}+\eta)t}
 \end{align*}
 This completes the proof of (ii).
 Finally, we have
 $$\frac{N}{K}\mathbb{E}_{\theta}\Big[\|\boldsymbol{M}_{t}^{N,K}-\bar{M}_{t}^{N,K}\boldsymbol{1}_{K}\|^{2}_{2}\Big]
 \le \mathbb{E}_{\theta}[|X_{t}^{N,K}|]+N\mathbb{E}_{\theta}[\bar{Z}_{t}^{N,K}]\le C_{\eta}Ne^{(\alpha_{0}+\eta)t}.$$
 This completes the proof of (iii).
 \end{proof}

Next we consider the term $\cD_{t}^{N,K,1}$.
\begin{lemma}\label{DK1}
Assume (A). For all $\eta>0$, there are $N_{\eta}\ge 1$, $t_{\eta}\ge 0$ and $C_{\eta}<\infty$ such that for all $N\ge N_{\eta}$, all $t\ge t_{\eta}$, we have:

$$\mathbb{E}[\boldsymbol{1}_{\Omega_{N}^{K,2}}\cD_{t}^{N,K,1}]\le C_{\eta}e^{4\eta t}\Big(\frac{1}{\sqrt{K}}+\Big(\frac{\sqrt{N}}{e^{\alpha_{0}t}}\Big)^{\frac{3}{2}}+\frac{N}{\sqrt{K}}e^{-\alpha_{0}t}\Big).$$

\end{lemma}

\begin{proof}
Recalling \eqref{defI} and \eqref{defJ}, we start from 
$\boldsymbol{Z}_{t}^{N,K}=\boldsymbol{M}_{t}^{N,K}+\boldsymbol{J}_{t}^{N,K}+v_{t}^{N,K}\boldsymbol{V}_{N}^{K}+\boldsymbol{I}_{t}^{N,K}$.  In view of (\ref{cD1}), we have:
\begin{align*}
\cD_{t}^{N,K,1}
&=\frac{1}{(v_{t}^{N,K})^{2}}\Big|\frac{N}{K}\|\boldsymbol{I}_{t}^{N,K}-\bar{I}_{t}^{N,K}\boldsymbol{1}_{K}+\cJ_{t}^{N,K}-\bar{J}_{t}^{N,K}\boldsymbol{1}_{K}\|_{2}^{2}+\frac{N}{K}\|\cM_{t}^{N,K}-\bar{M}_{t}^{N,K}\boldsymbol{1}_{K}\|_{2}^{2}\\
&\qquad-NZ_{t}^{N,K}+2\frac{N}{K}\Big(\boldsymbol{I}_{t}^{N,K}-\bar{I}^{N,K}_{t}\boldsymbol{1}_{K}+\cJ_{t}^{N,K}-\bar{J}_{t}^{N,K}\boldsymbol{1}_{K},v_{t}^{N,K}(V_{N}^{K}-\bar{V}_{N}^{K}\boldsymbol{1}_{K})
\\&\qquad+\boldsymbol{M}_{t}^{N,K}-\bar{M}_{t}^{N,K}\boldsymbol{1}_{K}\Big)
+2\frac{N}{K}v_{t}^{N,K}\Big(\boldsymbol{V}_{N}^{K}-\bar{V}_{N}^{K}\boldsymbol{1}_{K},\boldsymbol{M}_{t}^{N,K}-\bar{M}_{t}^{N,K}\boldsymbol{1}_{K}\Big)\Big|\\
&\le \frac{1}{(v_{t}^{N,K})^{2}}\Big[2\frac{N}{K}\|\boldsymbol{I}_{t}^{N,K}-\bar{I}_{t}^{N,K}\boldsymbol{1}_{K}\|^{2}_{2}+2\frac{N}{K}\|\boldsymbol{J}_{t}^{N,K}-\bar{J}_{t}^{N,K}\boldsymbol{1}_{K}\|_{2}^{2}+|X_{t}^{N,K}|
\\&\qquad+2\frac{N}{K}\Big(\|\boldsymbol{I}_{t}^{N,K}-\bar{I}_{t}^{N,K}\boldsymbol{1}_{K}\|_{2}+\|\boldsymbol{J}_{t}^{N,K}-\bar{J}_{t}^{N,K}\boldsymbol{1}_{K}\|_{2}\Big)\Big(v_{t}^{N,K}\|\boldsymbol{V}_{N}^{K}-\bar{V}_{N}^{K}\boldsymbol{1}_{K}\|_{2}
\\&\qquad+\|\boldsymbol{M}_{t}^{N,K}-\bar{M}_{t}^{N,K}\boldsymbol{1}_{K}\|_{2}\Big)+2\frac{N}{K}\Big|v_{t}^{N,K}\Big(\boldsymbol{V}_{N}^{K}-\bar{V}_{N}^{K}\boldsymbol{1}_{K},\boldsymbol{M}_{t}^{N,K}-\bar{M}_{t}^{N,K}\boldsymbol{1}_{K}\Big)\Big|\Big].
\end{align*}
We fix $\eta>0$ and work with $N$ and $t$ large enough and on $\Omega_N^{K,2}$.
Using Lemmas \ref{INK}, \ref{JNK}, \ref{Cvc}, \ref{MVV} together with the fact that $c \sqrt{K} \le \|\boldsymbol{V}_{N}^{K}\|_{2}\le C\sqrt{K}$ on $\Omega_{N}^{K,2}$ (by Lemma \ref{123456} (ii)),
we deduce the following bound on the set $\Omega_{N}^{K,2}$:
\begin{align*}
&\mathbb{E}_{\theta}\bigl[\cD_{t}^{N,K,1}\bigr] \le C_\eta e^{-2(\alpha_{0}-\eta)t}\Big\{N^{\frac{1}{4}}e^{{2\eta}t}+e^{(\alpha_{0}+\eta)t}+e^{2(\alpha_{0}+\eta)t}\frac{1}{\|\boldsymbol{V}_{N}^{K}\|^{2}_{2}}\|\boldsymbol{V}_{N}^{K}-\bar{V}_{N}^{K}\boldsymbol{1}_{K}\|_{2}^{2}\\&+\frac{N}{\sqrt{K}}e^{(\alpha_{0}+\eta)t}+\frac{2N}{K}\Big[t\sqrt{K}N^{-\frac{3}{8}}+\sqrt{\frac{K}{N}}e^{\frac{\alpha_{0}+\eta}{2}t}+\sqrt{\frac{K}{N}}e^{(\alpha_{0}+\eta)}
\frac{1}{\|\boldsymbol{V}_{N}^{K}\|_{2}}\|\boldsymbol{V}_{N}^{K}-\bar{V}_{N}^{K}\|_{2}\Big]\Big[\sqrt{K}e^{\frac{\alpha_{0}+\eta}{2}t}\\&+e^{(\alpha_{0}+\eta)t}\|\boldsymbol{V}_{N}^{K}-\bar{V}_{N}^{K}\boldsymbol{1}_{K}\|_{2}\Big]+e^{\frac{1}{2}(\alpha_{0}+\eta)t}\frac{N}{K}\|\boldsymbol{V}_{N}^{K}-\bar{V}_{N}^{K}\boldsymbol{1}_{K}\|_{2}\Big\}
\end{align*}

By proposition \ref{VmVb}, we finally obtain:
\begin{align*}
\mathbb{E}[\boldsymbol{1}_{\Omega_{N}^{K,2}}\cD_{t}^{N,K,1}] 
&\le  C_\eta e^{-2(\alpha_{0}-\eta)t}\Big|N^{\frac{5}{8}}te^{\frac{\alpha_{0}+\eta}{2}}+\frac{N}{\sqrt{K}}e^{(\alpha_{0}+\eta)t}+N^{\frac{1}{8}}e^{(\alpha_{0}+\eta)t}\\
&\qquad+e^{\frac{3(\alpha_{0}+\eta)}{2}t}
+e^{\frac{3}{2}(\alpha_{0}+\eta)t}+e^{2(\alpha_{0}+\eta)t}\frac{1}{\sqrt{N}}+N^{\frac{1}{4}}e^{2\eta t}\Big|\\
&\le C_\eta e^{4\eta t}\Big|N^{\frac{5}{8}}e^{-\frac{3}{2}\alpha_{0}t}+\frac{N}{\sqrt{K}}e^{-\alpha_{0}t}+e^{-\frac{1}{2}\alpha_{0}t}+\frac{1}{\sqrt{N}}\Big|.
\end{align*}
Since $\frac{N}{\sqrt{K}}e^{-\alpha_{0}t}+\frac{1}{\sqrt{N}}\ge e^{-\frac{\alpha_{0}}{2}t},\  N^{\frac{5}{8}}e^{-\frac{3}{2}\alpha_{0}t}\le (\sqrt{N}e^{-\alpha_{0}t})^{\frac{3}{2}}$, one gets
$$\mathbb{E}[\boldsymbol{1}_{\Omega_{N}^{K,2}}\cD_{t}^{N,K,1}]\le C_{\eta}e^{4\eta t}\Big(\frac{1}{\sqrt{N}}+\Big(\frac{\sqrt{N}}{e^{\alpha_{0}t}}\Big)^{\frac{3}{2}}+\frac{N}{\sqrt{K}}e^{-\alpha_{0}t}\Big).$$
\end{proof}

\section{Proof of the main theorem in the supercritical case.}

In this section we prove Theorem \ref{supercrit} and Remark \ref{remsup}.

\subsection{Proof of Theorem \ref{supercrit}}
By Lemma \ref{DNK}, 
on the event $\Omega_{N}^{K,2}\cap \{\bar{Z}_{t}^{N,K}\ge \frac{1}{4}v_{t}^{N,K}>0\}$,\\
we already have the following inequality:

$$\cD_{t}^{N,K}\le 16\cD_{t}^{N,K,1}+128\frac{N}{K}\|\boldsymbol{V}_{N}^{K}-\bar{V}_{N}^{K}\boldsymbol{1}_{N}\|_{2}^{2}\cD_{t}^{N,K,2}+\Big|\mathcal{U}_{\infty}^{N,K}-\Big(\frac{1}{p}-1\Big)\Big|.$$
Thus
\begin{align*}
&\boldsymbol{1}_{\Omega_{N}^{K,2}}\mathbb{E}_{\theta}\Big[\boldsymbol{1}_{\{\bar{Z}_{t}^{N,K}\ge v_{t}^{N,K}/4>0\}}\Big|\mathcal{U}_{t}^{N,K}-\Big(\frac{1}{p}-1\Big)\Big|\Big]
\le \boldsymbol{1}_{\Omega_{N}^{K,2}}\Big|\mathcal{U}_{\infty}^{N,K}-\Big(\frac{1}{p}-1\Big)\Big|+C_{\eta}16\mathbb{E}_{\theta}[\cD_{t}^{N,K,1}]\\&+128\frac{N}{K}\|\boldsymbol{V}_{N}^{K}-\bar{V}_{N}^{K}\boldsymbol{1}_{N}\|_{2}^{2}\ \mathbb{E}_{\theta}[\cD_{t}^{N,K,2}]. 
\end{align*}
From Proposition \ref{VmVb} and Lemmas  \ref{D2K}, \ref{DK1} it follows that
$$\mathbb{E}\Big[\boldsymbol{1}_{\Omega_{N}^{K,2}}\boldsymbol{1}_{\{\bar{Z}_{t}^{N,K}\ge v_{t}^{N,K}/4>0\}}\Big|\mathcal{U}_{t}^{N,K}-\Big(\frac{1}{p}-1\Big)\Big|\Big]\le C_{\eta}e^{4\eta t}\Big(\frac{1}{\sqrt{K}}+\Big(\frac{\sqrt{N}}{e^{\alpha_{0}t}}\Big)^{\frac{3}{2}}+\frac{N}{\sqrt{K}}e^{-\alpha_{0}t}\Big).$$
Moreover, by Lemmas \ref{ONK2} and \ref{D2K} we have:
$$P(\Omega_{N}^{K,2})\ge 1-Ce^{-cN^{\frac{1}{4}}}, \quad
P_{\theta}\Big(\bar{Z}_{t}^{N,K}\le \frac{1}{4}v_{t}^{N,K}\Big)\le C_{\eta}e^{2\eta t}\Big(\frac{1}{\sqrt{K}}+e^{-\alpha_{0}t}\Big).$$
Hence, by the Chebyshev's inequality, we obtain:
\begin{align*}
P(|\mathcal{P}_{t}^{N,K}-p|\ge \varepsilon)
&\le
(C_{\eta}/\varepsilon) e^{4\eta t}\Big(\frac{1}{\sqrt{K}}+\Big(\frac{\sqrt{N}}{e^{\alpha_{0}t}}\Big)^{\frac{3}{2}}+\frac{N}{\sqrt{K}}e^{-\alpha_{0}t}\Big)\\
&\quad+Ce^{-cN^{\frac{1}{4}}}+C_{\eta}e^{2\eta t}\Big(\frac{1}{\sqrt{K}}+e^{-\alpha_{0}t}\Big)\\
&\le (C_{\eta}^{,}/\varepsilon)e^{4\eta t}\Big(\frac{1}{\sqrt{K}}+\Big(\frac{\sqrt{N}}{e^{\alpha_{0}t}}\Big)^{\frac{3}{2}}+\frac{N}{\sqrt{K}}e^{-\alpha_{0}t}\Big).
\end{align*}
Finally, using that
$(\frac{\sqrt{N}}{e^{\alpha_{0}t}})^{\frac{3}{2}}\le \frac{N}{\sqrt{K}}e^{-\alpha_{0}t}$,
we get:
$$P(|\mathcal{P}_{t}^{N,K}-p|\ge \varepsilon)\le \frac{C_{\eta}e^{4\eta t}}{\varepsilon}\Big(\frac{N}{\sqrt{K}e^{\alpha_{0}t}}+\frac{1}{\sqrt{K}}\Big).$$
The proof is complete.

\subsection{Proof of Remark \ref{remsup}}
By the Lemma \ref{D2K}, for $N\ge N_{\eta}$, we have that: 
$$\boldsymbol{1}_{\Omega_{N}^{K,2}}\mathbb{E}_{\theta}\Big[\Big|\frac{\bar{Z}_{t}^{N,K}}{v_{t}^{N,K}}-\bar{V}_{N}^{K}\Big|\Big]=\boldsymbol{1}_{\Omega_{N}^{K,2}}\mathbb{E}_{\theta}[\mathcal{D}_{t}^{N,K,2}]\le C_{\eta}e^{2\eta t}\Big(\frac{1}{\sqrt{K}}+e^{-\alpha_{0}t}\Big).$$
From lemma \ref{123456} $(ii)$, we have for all $V_{N}(i)\in [\frac{1}{2},2]$. So $\bar{V}_{N}^{K}=(\frac{1}{K}\sum_{i=1}^{K}V_{N}(i))\in [\frac{1}{2},2]$ on the set $\Omega_{N}^{K,2}.$ From lemma \ref{ONK2}, we have $P(\Omega_{N}^{K,2})\ge 1-Ce^{-cN^{\frac{1}{4}}}$. From Lemma \ref{Cvc},  for $t\ge t_{\eta}$ we  get $ v_{t}^{N,K}\in [a_{\eta}e^{(\alpha_{0}+\eta)}, b_{\eta}e^{(\alpha_{0}-\eta)t}]$ for some $a_{\eta}<b_{\eta}$. So we  deduce that for $N\ge N_{\eta}$, $t\ge t_{\eta}$,

$$P\Big(\bar{Z}_{t}^{N,K}\in [\frac{a_{\eta}}{2}e^{(\alpha_{0}-\eta)t},2b_{\eta}e^{(\alpha_{0}+\eta)t}]\Big)\ge 1-Ce^{-cN^{\frac{1}{4}}}-C_{\eta}e^{2\eta t}\Big(\frac{1}{\sqrt{K}}+e^{-\alpha_{0}t}\Big).$$ This implies that 
for any $\eta>0$, $$\lim_{t\to \infty}\lim_{(N,K)\to (\infty,\infty)} P(\bar{Z}_{t}^{N,K}\in[e^{(\alpha_{0}-\eta)t},e^{(\alpha_{0}+\eta)t}])=1.$$

\section{Acknowledgement}
I want to thank my supervisors, N. Fournier and S. Delattre. This paper can not be finished without their superb guidance. I am sincerely grateful to them.



\bibliographystyle{amsplain}

\end{document}